\pdfoutput=1
\documentclass[reqno]{amsart}

\usepackage[utf8]{inputenc}
\usepackage[T1]{fontenc}
\usepackage{lmodern}
\usepackage{microtype}

\usepackage{enumitem}
\usepackage[margin={3.5cm,2cm}]{geometry}
\usepackage{graphicx}

\usepackage[foot]{amsaddr}

\usepackage{titletoc}

\dottedcontents{appsection}[2.5em]{}{2.5em}{.5pc}
\dottedcontents{appsubsection}[5.5em]{}{3em}{.5pc}

\numberwithin{equation}{section}

\usepackage{hyperref}
\usepackage[capitalize]{cleveref}
\crefname{section}{section}{sections}
\crefname{subsection}{subsection}{subsections}
\Crefname{figure}{Figure}{Figures}
\usepackage{bookmark}
\usepackage[ocgcolorlinks]{ocgx2}
\usepackage{xcolor}
\hypersetup{
  colorlinks=true,
  linkcolor=blue,
  citecolor=blue,
  urlcolor=blue,
  pdfinfo={
    Author={P. Friz and P. Hager},
    Title={Expected Signature Kernels for %
    Rough Paths},
    Keywords={Signatures, L\'evy processes, rough paths, kernel mean embedding, maximum mean discrepancy},
    Subject={MSC(2020) 60L10, 60L90, 60E10, 60G44, 60G48, 60G51, 60J76}
  }
}

\usepackage{mathtools, amssymb, mathrsfs, shuffle, dsfont}%
\allowdisplaybreaks
\usepackage{mathabx}
\usepackage{mleftright, xspace}
\mleftright
\usepackage[sort,compress]{cite}

\usepackage{todonotes}
\title{Expected Signature Kernels for %
Rough Paths}

\author[Peter K.~Friz]{Peter K.~Friz$^{\dagger,\ddagger}$}
\author[Paul P.~Hager]{Paul P.~Hager$^{\mathsection}$}

\email{friz@math.tu-berlin.de}
\email{paul.peter.hager@univie.ac.at}

\address{$^{\dagger}$Institut f\"ur Mathematik, TU Berlin, Stra\ss e des 17.\ Juni 136, 10623 Berlin, Germany}
\address{$^{\ddagger}$Weierstrass Institute,  Anton-Wilhelm-Amo-Str.\ 39, 10117 Berlin, Germany}
\address{$^{\mathsection}$Department of Statistics and Operations Research, University of Vienna, Kolingasse 14--16, 1090 Wien, Austria}
\date{\today}

\newcommand{\esig}{\pmb{\mu}}

\renewcommand{\d}{\mathrm{d}}

\newcommand{\Id}{\mathrm{Id}}

\newcommand{\dd}{\mathrm{d}}
\newcommand{\TT}{\mathcal{T}}

\newcommand{\R}{\mathbb{R}}
\newcommand{\N}{\mathbb{N}}

\newcommand{\E}{\mathbb{E}}

\newcommand{\bx}{\mathbf{x}}
\newcommand{\by}{\mathbf{y}}
\newcommand{\bz}{\mathbf{z}}

\renewcommand{\P}{\mathbb{P}}
\newcommand{\F}{\mathcal{F}}

\renewcommand{\d}{\mathrm{d}}

\DeclareMathOperator{\Sigop}{Sig}
\newcommand{\Sig}[1]{\Sigop(#1)}

\DeclareMathOperator{\dualLeftOp}{L}
\newcommand{\dualLeft}[2]{\dualLeftOp_{#1}#2}

\DeclareMathOperator{\dualRightOp}{R}
\newcommand{\dualRight}[2]{\dualRightOp_{#1}#2}

\DeclareMathOperator{\dualLeftOpZ}{L^\circ}
\newcommand{\dualLeftZ}[2]{\dualLeftOp^\circ_{#1}#2}

\DeclareMathOperator{\dualRightOpZ}{R^\circ}
\newcommand{\dualRightZ}[2]{\dualRightOp^\circ_{#1}#2}

\newcommand{\onevar}[2]{\Vert #1\Vert_{\mathrm{AC},1;[0,#2]}}
\newcommand{\onevartwo}[3]{\Vert #1\Vert_{\mathrm{AC},1;[#2,#3]}}

\DeclareMathOperator{\MMD}{MMD}

\newcommand{\samples}{M}
\newcommand{\dMMD}{\MMD_{\mathrm{Sig}}}

\newcommand{\ResQV}{R_{\mathcal{S}\text{-}\mathrm{QV}}}

\newcommand{\MMDhatop}{%
    \mathop{%
        \vphantom{\mathrm{MMD}}%
        \smash{\widehat{\mathrm{MMD}}}%
    }\nolimits
}

\newcommand{\dMMDhat}{\MMDhatop_{\mathrm{Sig}}^{2}}
\newcommand{\dcharhat}{\MMDhatop_{\mathcal{S}\text{-}\mathrm{QV}}^{2}}

\newcommand{\dMMDhatpi}{\MMDhatop_{\mathrm{Sig}}^{2,\pi}}
\newcommand{\dcharhatpi}{\MMDhatop_{\mathcal{S}\text{-}\mathrm{QV}}^{2,\pi}}

\newcommand{\lvyalg}{\gamma}

\newcommand{\indic}[1]{\mathds{1}_{#1}} %
\newcommand{\ip}[2]{\langle #1\,,#2\rangle} %
\newcommand{\ceil}[1]{\left\lceil#1\right\rceil} %

\newcommand{\Norm}[1]{|#1|}

\newcounter{cprop}[section]

\newtheorem{theorem}[cprop]{Theorem}
\newtheorem*{theorem*}{Theorem}

\theoremstyle{plain}

\newtheorem{corollary}[cprop]{Corollary}
\newtheorem*{corollary*}{Corollary}

\newtheorem{lemma}[cprop]{Lemma}
\newtheorem{proposition}[cprop]{Proposition}

\numberwithin{equation}{section}

\theoremstyle{definition}
\newtheorem{definition}[cprop]{Definition}
\newtheorem{notation}[cprop]{Notation}
\newtheorem{example}[cprop]{Example}

\theoremstyle{remark}
\newtheorem{remark}[cprop]{Remark}

\makeatletter
\@ifclassloaded{amsart}{%
  \newcommand*\crefthm[3]{%
    \crefname{#1}{#2}{#3}\Crefname{#1}{#2}{#3}%
    \AddToHook{env/#1/begin}{\crefalias{cprop}{#1}}}%
  \crefthm{theorem}{Theorem}{Theorems}%
  \crefthm{algorithm}{Algorithm}{Algorithms}%
  \crefthm{axiom}{Axiom}{Axioms}%
  \crefthm{conjecture}{Conjecture}{Conjectures}%
  \crefthm{corollary}{Corollary}{Corollaries}%
  \crefthm{criterion}{Criterion}{Criteria}%
  \crefthm{exercise}{Exercise}{Exercises}%
  \crefthm{lemma}{Lemma}{Lemmas}%
  \crefthm{proposition}{Proposition}{Propositions}%
  \crefthm{solution}{Solution}{Solutions}%
  \crefthm{definition}{Definition}{Definitions}%
  \crefthm{notation}{Notation}{Notations}%
  \crefthm{example}{Example}{Examples}%
  \crefthm{condition}{Condition}{Conditions}%
  \crefthm{assumption}{Assumption}{Assumptions}%
  \crefthm{problem}{Problem}{Problems}%
  \crefthm{remark}{Remark}{Remarks}%
  \crefthm{claim}{Claim}{Claims}%
  \crefthm{summary}{Summary}{Summaries}%
  \crefthm{acknowledgement}{Acknowledgement}{Acknowledgements}%
  \crefthm{case}{Case}{Cases}%
  \crefthm{conclusion}{Conclusion}{Conclusions}%
  \AtBeginDocument{\g@addto@macro\appendix{%
    \crefalias{section}{appendix}%
    \crefalias{subsection}{subappendix}%
    \crefalias{subsubsection}{subsubappendix}}}%
}{}
\makeatother

\begin{document}
\maketitle

\begin{abstract}
The expected signature kernel is a similarity measure of probability laws on path space and enters directly into signature maximum mean discrepancies. Building on [P.~K. Friz, P.~P. Hager, and N. Tapia, Forum Math. Sigma, 10 (2022), e42] and [M. Lemercier, T. Lyons, and C. Salvi, SIAM J. Numer. Anal., 64 (2026), pp.~1095--1125], we characterize the expected signature kernel of group-valued inhomogeneous L\'evy processes, aka \emph{L\'evy rough paths}, as the scalar component of the solution to an extended Goursat-type system. We give an original description, first in the continuous case where the finite-dimensional system is exact. In the presence of jumps, this system is inherently infinite-dimensional; we prove its well-posedness and derive finite-dimensional truncations with explicit error control. This permits direct evaluation of the reference-dependent terms in one-sample signature-MMD tests without an auxiliary reference sample. Finally, for conditionally Gaussian martingales, conditioning the usual unbiased empirical squared signature-MMD on the sampled quadratic variations yields an unbiased statistic for the same quantity with no larger variance. We derive computable finite-grid versions and illustrate the resulting variance reduction and gain in power.
\end{abstract}

\section{Introduction}

Integral probability metrics (IPMs) provide a flexible way to compare probability measures, with familiar examples including the Kolmogorov--Smirnov, $1$-Wasserstein, and total variation distances. In this paper we focus on the \emph{maximum mean discrepancy} (MMD), obtained by taking the test-function class to be the unit ball of a reproducing kernel Hilbert space (RKHS). For an introduction to IPMs and kernel mean embeddings, see \cite{muandet2017kernel}. Given a positive definite kernel $k$ on $\mathcal{X}$ with associated RKHS $\mathcal{H}_k$, the MMD between $P$ and $Q$ is
\begin{equation}\label{eq:mmd_general_def}
    \mathrm{MMD}_k(P,Q) \coloneqq \sup_{f \in \mathcal{H}_k,\; \|f\|_{\mathcal{H}_k}\le 1}
\Bigl( \mathbb{E}_{X\sim P}[f(X)] - \mathbb{E}_{Y\sim Q}[f(Y)] \Bigr).
\end{equation}
Provided the expectation exists, the \emph{kernel mean embedding} of $P$ is $\mu_P:=\mathbb{E}_{X\sim P}[k(X,\cdot)]$,
and the reproducing property yields
\begin{equation}\label{eq:mmd_kernelmean}
\mathrm{MMD}_k(P,Q)=\|\mu_P-\mu_Q\|_{\mathcal H_k}.
\end{equation}
Alternatively, writing $X,X'\stackrel{\mathrm{i.i.d.}}{\sim}P$ and $Y,Y'\stackrel{\mathrm{i.i.d.}}{\sim}Q$, all independent,
\begin{equation}\label{eq:iid_mmd}
\mathrm{MMD}_k^2(P,Q)
=
\mathbb{E}[k(X,X')]
+
\mathbb{E}[k(Y,Y')]
-
2\mathbb{E}[k(X,Y)].
\end{equation}

This framework applies beautifully to signatures \cite{LyonsICM}. The signature defines a nonlinear feature map on path spaces of suitable regularity into the Hilbert space $(\TT^2,\langle\cdot,\cdot\rangle)$ of square-summable tensor series (definitions below) 
with numerous applications in data science, cf. \cite{bayer2025signature,cass2024lecturenotesroughpaths,crisan2026stochastic,mcleod2026signature} and references therein.
The associated \emph{signature kernel} \cite{kiraly2019kernels},
\begin{equation*}
k_{\mathrm{Sig}}(\gamma\vert_{[0,s]},\widetilde{\gamma}\vert_{[0,t]})
=
\langle\Sig{\gamma}_{0,s},\Sig{\widetilde{\gamma}}_{0,t}\rangle,
\end{equation*}
then features into the above construction yielding the \emph{signature-MMD} \cite{chevyrev2022signature}, with an account of successful applications in \cite{horvath2025signature,andres2024signature,dyer2024approximate,kidger2020deep,liao2024sigwasserstein} among others.
Signature kernels of deterministic finite variation paths were shown to satisfy a simple Goursat partial differential equation in \cite{salvi2021signature}. A remarkable generalization to high-frequency data is developed in \cite{lemercier2024high}.

The present work generalizes \cite{lemercier2024high} to classes of stochastic processes. We derive differential systems for the expected signature kernel, that is, for
\begin{align}\label{eq:expected_signature_kernel}
\E_{X\sim P,\,Y\sim Q}\big[k_{\mathrm{Sig}}(X\vert_{[0,s]},Y\vert_{[0,t]})\big]
=
\big\langle\E_{X\sim P}[\Sig{X}_{0,s}],\E_{Y\sim Q}[\Sig{Y}_{0,t}]\big\rangle.
\end{align}
These quantities therefore enter directly into the corresponding MMD through \eqref{eq:mmd_kernelmean}, since the kernel mean embedding of a law $P$ on path space is represented by its expected signature
$\esig_P = \mathbb{E}_{X\sim P}[\Sig{X}]$,
whenever this expectation is well defined.

More concretely, our main analytical result, \cref{thm:main}, characterizes \eqref{eq:expected_signature_kernel} for group-valued inhomogeneous L\'evy processes, aka \emph{L\'evy rough paths}, as the scalar component of the unique continuous solution to an extended Goursat-type system. This system consists of one linear hyperbolic equation coupled with two tensor-valued evolution equations. The distributional information of each underlying process enters the system entirely through the \emph{characteristic velocity} associated with its time-varying L\'evy triplet $(a,b,K)$, \[
    \mathfrak{y}(t) = b(t) + \frac{1}{2}a(t) + \int_{\mathfrak{g}^N(V)}(\exp_\otimes(x) - 1 - x \indic{|x|\le 1}) K_t(\d{x}).
\] In the presence of a nontrivial L\'evy measure $K$, these characteristic velocities contain infinitely many tensor levels, and the resulting system is inherently infinite-dimensional; we prove its well-posedness in \Cref{thm:main}. \Cref{thm:truncated} provides a finite-dimensional system with explicit control of the truncation error. In the continuous case, where the characteristic velocities have only finitely many nonzero levels, the same system computes the expected signature kernel exactly. In the presence of jumps, truncating characteristic velocities can yield provably faster convergence than truncating expected signatures directly.

This has a direct computational consequence for one-sample signature-MMD goodness-of-fit testing; see also \cite{balasubramanian2021optimality,hagrass2024spectral}. For $\samples$ observations and a known inhomogeneous L\'evy reference law $Q$, approximating the unavailable $Q$-dependent terms in \eqref{eq:iid_mmd} using $n_\mathrm{MC}$ paths from $Q$ effectively turns the problem into the two-sample MMD setting of \cite{gretton2012kernel}. Under the null hypothesis, the covariance of the simulated reference embedding $\widehat{\esig}_Q$ is $\samples/n_\mathrm{MC}$ times that of the observed embedding, so $n_\mathrm{MC}$ must be substantially larger than $\samples$ for this Monte Carlo noise to be small. Yet the complete estimator requires $O(\samples^2+\samples n_\mathrm{MC}+n_\mathrm{MC}^2)$ kernel evaluations; even after precomputing the reference--reference block, $\samples n_\mathrm{MC}$ observation--reference evaluations remain per dataset. This can quickly become a bottleneck because signature-kernel evaluations are themselves costly. Our systems compute the $Q$-dependent terms analytically.

A second consequence addresses the finite-sample variance of empirical MMD statistics, which can reduce the power of the resulting tests \cite{gretton2012optimal,sutherland2017generative}; related effects for the signature-MMD are documented in \cite{horvath2025signature}. For conditionally Gaussian martingales, \cref{lem:conditional_gaussian_orthogonality} identifies
\begin{equation}\label{eq:intro_conditional_esig}
\E\left[\Sig{X}_{0,T}\,\middle|\,\langle X\rangle\right]
=
\mathcal S\left(\tfrac12\langle X\rangle\right)_{0,T}.
\end{equation}
Consequently, the resulting quadratic-variation statistic is the conditional expectation of the usual unbiased empirical squared signature-MMD given the sampled quadratic variations $\langle X\rangle$ and therefore estimates the same quantity with no larger variance. This applies, in particular, under the null hypothesis of a one-sample test against a conditionally Gaussian reference law and, by the same conditioning argument, to two-sample comparisons within this class. \Cref{thm:dMMD_formula} gives computable finite-grid versions. We illustrate the variance reduction and gain in power numerically. The same conditioning argument should extend to processes with conditionally independent increments, including jumps, but we focus on the continuous case for clarity; see \cref{rmk:conditional_pii}. Future research may further exploit the closed form \eqref{eq:intro_conditional_esig} of the conditional mean embedding \cite{song2009hilbert} and the conditional moment restriction $\E[\Sig{X}_{0,T}-\mathcal S(\tfrac12\langle X\rangle)_{0,T}\mid\langle X\rangle]=0$ for kernel specification tests of conditional Gaussianity \cite{muandet2020kernel}.

\noindent\textbf{Section outline.} The remainder of the paper is organized as follows. Section~2 collects the tensor-algebraic and stochastic preliminaries. Section~3 develops the expected signature kernel systems, first in the general infinite-dimensional setting and then through finite-dimensional truncations and explicit second-level equations. Section~4 studies signature-MMDs for conditionally Gaussian martingales, including variance reduction, finite-grid computation, and a goodness-of-fit example.
The appendices provide the extended framework and proofs, the a priori and
truncation estimates, square-integrability results, and further numerical validation.

\noindent\textbf{Code and data availability.}
All numerical results reported here are reproducible with the \texttt{levymmd} package, \url{https://github.com/hagerpa/levymmd}, which builds on the free-development and signature-kernel routines of \texttt{tensordev}, \url{https://github.com/hagerpa/tensordev}. Its notebooks \texttt{truncation}, \texttt{gof} and \texttt{validation} reproduce, respectively, \cref{fig:trunc_circle}, the goodness-of-fit study of \cref{sec:numerics_gof}, and the numerical validation in \cref{sec:numerics_montecarlo}.

\noindent\textbf{Acknowledgment.}
PKF acknowledges partial support from Deutsche Forschungsgemeinschaft (DFG,
German Research Foundation) through CRC/TRR 388 “Rough Analysis, Stochastic Dynamics and Related Fields” (project id 516748464, projects A02, A05), as well as  initial support from the Cluster of Excellence MATH+ EXC-2046, project id 390685689, via a MATH+ Distinguished Fellowship.

\noindent\textbf{Declarations.}
Ideas and proofs are due to the authors. OpenAI’s ChatGPT 5.6 Sol and Anthropic's Fable 5.0 were used only for final polishing and for help with the numerical implementation.
The authors assume responsibility for all content.

\section{Preliminaries}\label{sec:preliminaries}

We briefly collect the tensor-algebraic and stochastic notation used below, together with the identities and estimates required in the main arguments. A more detailed account of this framework is given in \cref{sec:extended_preliminaries}. There we discuss the general $\TT^p$ setting, give the componentwise construction and further estimates for free developments, and recall the relation between inhomogeneous L\'evy processes in free Lie algebras and groups. The more general statements and their proofs are also included there.

\subsection{Tensor algebra and summable tensor series}\label{sec:Tpspaces}

We fix a vector space $V\cong\R^d$ with basis $\{e_1,\dots,e_d\}$ and denote by $\mathcal W$ the set of words over the alphabet $\{1,\dots,d\}$. The empty word is denoted by $\varnothing$, and to $w=i_1\cdots i_n\in\mathcal W$ we associate the tensor and extended tensor algebra\begin{equation*}
   e_w:=e_{i_1}\otimes\cdots\otimes e_{i_n}\in V^{\otimes n},\qquad T(V)=\bigoplus_{n=0}^\infty V^{\otimes n}, \qquad \TT=T((V))=\prod_{n=0}^\infty V^{\otimes n} 
\end{equation*} and represent $\bx\in\TT$ by the formal tensor series $\bx=\sum_{w\in\mathcal W}\bx^we_w$. Tensor multiplication is determined by concatenation, so that $(\bx\otimes\by)^w=\sum_{w_1w_2=w}\bx^{w_1}\by^{w_2}$.
For $n,N\in\N$, let $\pi_n:\TT\to V^{\otimes n}$ denote the homogeneous projection, set $\bx^{(n)}:=\pi_n\bx$ and further define
\begin{equation*}
    \pi_{(0,N)}=\pi_0 + \dots + \pi_N, \qquad \bx^{(0,N)}=\pi_{(0,N)}\bx, \qquad T^N(V)=\pi_{(0,N)}\TT.
\end{equation*}
The truncated tensor algebra $T^N(V)$ is equipped with the truncated multiplication $\bx \otimes_N \by= \pi_{(0,N)}(\bx \otimes \by)$. On $T^N(V)$ we use the norm $\Norm{x}:=\max_{k=0,\dots,N}|x^{(k)}|$. Furthermore, set \begin{equation*}\TT_0:=\{\bx\in\TT:\pi_0\bx=0\}, \qquad \TT_1:=\{\bx\in\TT:\pi_0\bx=1\}.\end{equation*} The tensor exponential is given by $\exp_\otimes(\bx):=1+\sum_{n\ge1}\bx^{\otimes n}/n!$ for $\bx\in\TT_0$.
We equip $V^{\otimes n}$ with the Euclidean inner product and write $|\cdot|$ for the associated norm. For $p\ge1$, define
\begin{equation}\label{eq:def_Tp}
\TT^p:=\left\{\bx\in\TT:\Vert\bx\Vert_p^p=\sum_{n=0}^\infty|\bx^{(n)}|^p<\infty\right\}.
\end{equation}
Then $\TT^p$ is a Banach space, $\TT^1\subset\TT^2$, and $\TT^2$ is a  Hilbert space for the levelwise inner product \begin{equation*}\langle\bx,\by\rangle:=\sum_{n=0}^\infty\langle\bx^{(n)},\by^{(n)}\rangle=\sum_{w\in\mathcal W}\bx^w\by^w.\end{equation*}
Write $\TT^1_0:=\TT^1\cap\TT_0$.
We will use the following inequalities
\begin{align}
\Vert\bx\otimes\by\Vert_1&\le\Vert\bx\Vert_1\Vert\by\Vert_1,
&&\bx,\by\in\TT^1,\label{eq:young_T1}\\
\Vert\bz\otimes\by\Vert_2&\le\Vert\bz\Vert_2\Vert\by\Vert_1,
&&\bz\in\TT^2,\ \by\in\TT^1.\label{eq:young_T2}
\end{align}
In particular, $(\TT^1,+,\otimes)$ is a Banach algebra. These are special cases of Young's convolution inequality recorded in \cref{ext:prop:young}.

\subsection{Adjoint tensor multiplication}\label{sec:adjoint_multiplication}

Following \cite{lemercier2024high}, for $\bx,\bz\in\TT^1$ we define the operations adjoint to tensor multiplication by
\begin{equation*}
\dualLeft{\bx}{\bz}:=\sum_{w\in\mathcal W}e_w\langle\bz,\bx\otimes e_w\rangle,
\qquad
\dualRight{\bx}{\bz}:=\sum_{w\in\mathcal W}e_w\langle\bz,e_w\otimes\bx\rangle.
\end{equation*}
They define continuous bilinear maps from $\TT^1\times\TT^1$ to $\TT^1$ and satisfy for $\bx,\by,\bz\in\TT^1$:
\begin{align}
\langle\dualLeft{\bx}{\bz},\by\rangle
=\langle\dualRight{\by}{\bz},\bx\rangle =&\, \langle\bz,\bx\otimes\by\rangle,\label{eq:dual_multiplication}\\
\Vert\dualLeft{\bx}{\bz}\Vert_1, \;\Vert\dualRight{\bx}{\bz}\Vert_1
\le&\, \Vert\bx\Vert_1\Vert\bz\Vert_1.
\label{eq:adjoint_bounds}
\end{align}
The general mapping properties on $\TT^p$ spaces, of which these are the $\TT^1$ specializations, are proved in \cref{ext:prop:adjoint_mul}.
For $a\in V^{\otimes n}$ and $b\in V^{\otimes k}$, the grading gives
\begin{align}
\dualLeft{b}{a},\;\dualRight{b}{a}&\in V^{\otimes(n-k)}, && n\ge k,\label{eq:adjoint_degree}\\
\dualLeft{b}{a}=\dualRight{b}{a}&=0, && n<k,\label{eq:adjoint_degree_zero}\\
\dualLeft{b}{a}=\dualRight{b}{a}&=\langle a,b\rangle, && n=k.\label{eq:adjoint_equal_degree}
\end{align}
A proof of these and further identities is collected in \cref{ext:prop:left-right-technical}.

\subsection{Dilation and convergence radius}\label{sec:radius}

For $\lambda>0$, define the dilation $\delta_\lambda:\TT\to\TT$ by $\delta_\lambda\bx:=\sum_{n=0}^\infty\lambda^n\bx^{(n)}$. By definition of the tensor product,
\begin{equation}\label{eq:dilation_multiplication}
\delta_\lambda(\bx\otimes\by)=\delta_\lambda\bx\otimes\delta_\lambda\by.
\end{equation}
For $\rho\in(0,\infty]$, we set
\begin{equation}\label{eq:convergence_radius}
\mathcal R^\rho:=\left\{\bx\in\TT:\Vert\delta_\alpha\bx\Vert_1<\infty\text{ for all }\alpha\in(0,\rho)\right\}.
\end{equation}
The spaces $\mathcal R^\rho$ are subalgebras of $\TT$; this follows immediately from \eqref{eq:dilation_multiplication} and \eqref{eq:young_T1}, and is recorded in \cref{ext:sec:radius}. The case $\rho=\infty$ is of particular importance in \cite{chevyrev2016characteristic}, where it is shown that the distribution of a group-like random variable is characterized by its expectation whenever the latter belongs to $\mathcal R^\infty$.

\subsection{Free developments}\label{sec:free_developments}

Let $\gamma:[0,\infty)\to\TT_0$ be componentwise absolutely continuous. Its free development starting in $s\ge0$ is the unique levelwise solution
\begin{align}\label{eq:free_development_start}
\mathcal S(\gamma)_{s,t}=1+\int_s^t\mathcal S(\gamma)_{s,u}\otimes\dot\gamma(u)\dd u,
\qquad t\ge s,
\end{align}
and we write $\mathcal S(\gamma):=\mathcal S(\gamma)_{0,\cdot}$. The componentwise construction and the complete collection of basic properties are given in \cref{ext:sec:free_developments}. In particular, for $s \le u \le t$ and $\lambda > 0$ it holds
\begin{align}
\mathcal S(\gamma)_{s,t}
&=\mathcal S(\gamma)_{s,u}\otimes\mathcal S(\gamma)_{u,t},
\label{eq:free_development_flow}\\
\delta_\lambda\mathcal S(\gamma)
&=\mathcal S(\delta_\lambda\gamma),
\label{eq:free_development_dilation}\\
\mathcal S(\gamma)_{s,t}
&=\exp_\otimes((t-s)\bx),
\qquad\text{if }\dot\gamma\equiv\bx\in\TT_0.\label{eq:free_development_constant}
\end{align}
If $\dot\gamma(t)\in\mathrm{Lie}((V))$ for a.e. $t$, then $\mathcal S(\gamma)_{s,t}\in\mathcal G$, in which case we also write $\Sig{\gamma}_{s,t}:=\mathcal S(\gamma)_{s,t}$.
For a componentwise absolutely continuous path $\beta$ set $\onevartwo{\beta}{s}{t}:=\int_s^t\Vert\dot\beta(u)\Vert_1\dd u$. The estimates needed below are
\begin{align}
\Vert\mathcal S(\gamma)_{s,t}\Vert_1
&\le e^{\onevartwo{\gamma}{s}{t}},\label{eq:free_development_bound}\\
\onevartwo{\mathcal S(\gamma)}{s}{t}
&\le e^{\onevartwo{\gamma}{s}{t}}-1,\label{eq:free_development_variation}\\
\Vert\mathcal S(\gamma)_{s,t}-\mathcal S(\pi_{(0,N)}\gamma)_{s,t}\Vert_1
&\le e^{\onevartwo{\gamma}{s}{t}}\onevartwo{\gamma-\pi_{(0,N)}\gamma}{s}{t}.
\label{eq:free_development_truncation}
\end{align}
They are special cases of \cref{ext:prop:free_dev_estimates}; its proof also gives the stronger estimates used later for an expected signature truncation comparison in the jump case (cf. \cref{sec:tail_truncation}).

\subsection{Free Lie groups}
The free Lie algebra over $V$ is embedded in $T(V)$ with commutator bracket $[\bx,\by]_\otimes:=\bx\otimes\by-\by\otimes\bx$, and we denote by $\mathrm{Lie}((V))\subset\TT_0$ the corresponding formal Lie series. Its truncation is the free step-$N$ nilpotent Lie algebra $\mathfrak g^N(V)$, and we set $G^N(V):=\exp_{\otimes_N}(\mathfrak g^N(V))$. The full group of group-like elements is denoted by $\mathcal G = \exp_\otimes(\mathrm{Lie}((V)))$.

\subsection{Stochastic signatures and inhomogeneous L\'evy processes}\label{sec:signatures_semimartingales}\label{sec:inhom_levy}

Let $(\Omega,\F,(\F_t),\P)$ be a filtered probability space satisfying the usual conditions and let $\gamma$ be a $\mathfrak g^N(V)$-valued semimartingale. Its signature $\Sig{\gamma}_{0,\cdot}$ is defined as the full Marcus development in $\TT_1$, i.e., as the solution to
\begin{align}\label{eq:marcus_signature}
\dd S_t=S_{t-}\otimes\diamond\dd\gamma_t,
\qquad S_0=1.
\end{align}
It takes values in $\mathcal G$, and $\pi_{(0,N)}\Sig{\gamma}$ is the corresponding $G^N(V)$-valued development. 
Denoting by $\langle \gamma^c \rangle$ the quadratic variation of the continuous martingale part, the It\^o form of \eqref{eq:marcus_signature} reads
\begin{equation*}
    S_t = 1 + \int_0^t S_{u-} \otimes \d{\gamma}_u + \frac{1}{2}\int_0^t S_{u} \otimes \d{\langle\gamma^c\rangle}_u + \sum_{0 < u \le t}S_{u-}(e^{\Delta \gamma_u} - 1 - \Delta \gamma_u), \qquad t \ge0.
\end{equation*}
Moreover, dilation commutes with stochastic development:
\begin{equation}\label{eq:stochastic_signature_dilation}
\delta_\lambda\Sig{\gamma}_{s,t}=\Sig{\delta_\lambda\gamma}_{s,t},
\qquad \lambda>0.
\end{equation}
The general c\`adl\`ag version of the Itô form and the stochastic logarithm are recalled in \cref{ext:sec:signatures_semimartingales}.
We now directly introduce the class of \emph{inhomogeneous L\'evy processes} used throughout by specifying their deterministic semimartingale characteristics (cf. \cite[Section~II.4.c]{jacod2003limit}). Fix a basis $(\mathfrak u_i)_{i=1}^{D_N}$ of $\mathfrak g^N(V)$ and set \begin{equation*}\mathfrak S_+^N(V):=\left\{\sum_{i,j=1}^{D_N}a^{ij}\mathfrak u_i\otimes\mathfrak u_j:(a^{ij})_{i,j}\text{ is positive semidefinite}\right\}\subset T^{2N}(V).\end{equation*} We consider c\`adl\`ag, stochastically continuous $\mathfrak g^N(V)$-valued processes with independent increments, starting at zero, whose characteristics relative to the truncation function $x\mapsto x\indic{\Norm{x}\le1}$ are absolutely continuous. Thus, $B_t=\int_0^tb(s)\dd s$, $\Sigma_t=\int_0^ta(s)\dd s$, and $\nu(\dd x,\dd t)=K_t(\dd x)\dd t$, where $b:[0,\infty)\to\mathfrak g^N(V)$ and $a:[0,\infty)\to\mathfrak S_+^N(V)$ are locally integrable and $(K_t)_{t\ge0}$ is a family of L\'evy measures satisfying
\begin{equation*}
\int_0^t\int_{\mathfrak g^N(V)}(\Norm{x}^2\wedge1)K_s(\dd x)\dd s<\infty,
\qquad t\ge0.
\end{equation*}

\begin{notation}\label{not:inhom_levy}
We write $\gamma\sim\mathscr L^N(b,a,K)$ for an inhomogeneous L\'evy process with the differential triplet $(b,a,K)$ above.
\end{notation}
The associated $G^N(V)$-valued L\'evy rough path is $\pi_{(0,N)}\Sig{\gamma}$. The converse description through the stochastic logarithm, together with the general Lie-group formulation, is given in \cref{ext:sec:inhom_levy}.
For $\lambda>0$, dilation preserves this class
\begin{align}
\delta_\lambda\gamma\sim\mathscr L^N(b^\lambda,a^\lambda,K^\lambda),\qquad
K_t^\lambda&=K_t\circ\delta_{\lambda^{-1}},
\qquad
a^\lambda=\delta_\lambda a,\label{eq:triplet_dilation}
\end{align}
and the drift $b^\lambda$ depends on the chosen truncation function. The proof and further details are given in \cref{ext:rem:triplet_scaling}. Finally, when $K=0$, the characteristic velocity introduced in \eqref{def:frak_eta} has only finitely many nonzero levels: $b(t)$ takes values in $T^N(V)$ and $a(t)$ in $T^{2N}(V)$.

\section{Expected signature kernels of inhomogeneous L\'evy processes}
This section contains our 
main analytical result, which is an extended PDE system for computing the expected signature kernel of (group-valued) inhomogeneous L\'evy processes. 
The most general version of this system, stated in \cref{thm:main}, takes the form of a linear hyperbolic PDE (Goursat-type) coupled with two linear ODEs.
In the presence of a nontrivial L\'evy measure, the system is inherently infinite-dimensional, yet well-posed.

We further present, in \cref{thm:truncated}, a truncated (finite-dimensional) system that, on the one hand, approximates the infinite-dimensional system in the presence of jumps—with explicit control of the truncation error—and, on the other hand, allows exact computation of the expected signature kernel for continuous inhomogeneous L\'evy processes.
In the deterministic case, we recover the system from \cite{lemercier2024high} for the signature kernel of smooth rough paths.

Finally, in the special case of continuous $d$-dimensional inhomogeneous L\'evy processes, possibly enhanced with deterministic area terms, \cref{cor:level2} expresses the system in an accessible form using only matrix multiplication.

We begin by recalling known results on the expected signature of inhomogeneous L\'evy processes and provide a sufficient moment condition on the L\'evy measure for the expected signature kernel to be well defined.

\subsection{Expected signatures and a moment condition}\label{sec:moment_condition}%
Recall \cref{not:inhom_levy} for $\mathfrak{g}^N(V)$-valued inhomogeneous L\'evy processes.
The expected signature of such $\lvyalg\sim\mathscr{L}^N(b, a, K)$ has already been fully characterized in the literature and is summarized by the following
\begin{proposition}[\cite{FHT22}]\label{prop:levy_esig}
    Let $\lvyalg \sim \mathscr{L}^N(b, a, K)$ and assume the moment condition
    \begin{align*}
        \int_0^t \int_{\mathfrak{g}^N(V)} \indic{\Norm{x} > 1}\Norm{x}^p K_s(\d{x}) \d{s} < \infty, \qquad t \ge 0, \quad p>1,
    \end{align*}
    holds. 
    Then the expected signature $\esig(t):= \E[\Sig{\lvyalg}_{0,t}] := \sum_{w\in\mathcal{W}}e_w\E[\Sig{\lvyalg}^{w}_{0,t}] \in \TT_1$ is well defined and the unique solution to
    \begin{align}\label{eq:levy_esig}
        \esig(t) = 1 + \int_0^t \esig(u)\otimes\mathfrak{y}(u)\d{u}, \qquad t \ge 0,
    \end{align}
    where the \emph{characteristic velocity} in $\TT_0$ is defined by
    \begin{align}\label{def:frak_eta}
        \mathfrak{y}(t) := b(t) + \frac{1}{2}a(t) + \int_{\mathfrak{g}^N(V)}(\exp_\otimes(x) - 1 - x \indic{|x|\le 1}) K_t(\d{x}).
    \end{align}
\end{proposition}
\begin{remark}
The above result is a L\'evy--Khintchine-type formula for expected signatures.
The homogeneous case $\mathfrak{y}(t) \equiv \mathfrak{y}$ goes back to \cite{friz2017general} where one obtains the signature cumulants $\log_\otimes\E[\Sig{\lvyalg}_{0,T}] = T \mathfrak{y}$.
In the homogeneous and continuous case this is known as Fawcett's formula \cite{fawcett2002problems}.
Returning to the inhomogeneous case, we also note that the signature cumulants are obtained from a Magnus expansion of $\mathfrak{y}$ in $\TT_0$ \cite[Section~6.2.2]{FHT22}.
\end{remark}

Note that in the terminology and notation of \cref{sec:free_developments} $\esig$ is the \emph{free development} of the path $\int_0^\cdot \mathfrak{y}(t)\dd{t}$, i.e., $\esig = \mathcal{S}(\int_0^\cdot \mathfrak{y}(t)\dd{t})$. 
This is precisely the fact that will allow us to %
derive an extended PDE system for the expected signature kernel of inhomogeneous L\'evy processes in the next subsection.
To this end, note that the requirements in \cref{prop:levy_esig} are in general not enough to conclude that $\esig$ takes values in $\TT^2$.
The following result yields a sufficient condition for $\esig$ to take values in $\TT^1$, therefore in $\TT^2$, as desired for the existence of the expected signature kernel in \eqref{eq:expected_signature_kernel}.

\begin{theorem}\label{cor:exponential_moments}
    Let $\lvyalg \sim \mathscr{L}^N(b, a, K)$ and denote by $\mathfrak{y}$ its characteristic velocity defined in \eqref{def:frak_eta}.
    For any $\lambda > 0$ and $t\ge0$ a sufficient condition for $\int_{0}^{t}\Vert\delta_\lambda\mathfrak{y}(u)\Vert_1\dd{u} < \infty$, hence $\delta_\lambda\E[\Sig{\lvyalg}_{0,t}] \in \TT^1$, is given by
    \begin{align}\label{eq:exponential_moments}
        \int_0^t\int_{\mathfrak{g}^N(V)} \indic{\Norm{x}> 1}(e^{ \Vert \delta_\lambda x \Vert_1} - 1) K_s(\dd x)\dd s < \infty.
    \end{align}
     If the above holds for all $\lambda > 0$, then $\lvyalg$ belongs to the Chevyrev--Lyons class (cf. \cite{chevyrev2016characteristic}); that is, the distribution of $\Sig{\lvyalg}_{0,t}$ is characterized by its expectation $\E[\Sig{\lvyalg}_{0,t}]\in \mathcal{R}^{\infty}$ (using the notation from \cref{sec:radius}).
\end{theorem}
\begin{remark}We note that the above is an exponential-moment condition on the distribution of the jumps of $\lvyalg$, which is trivially satisfied when $\lvyalg$ is continuous.
\end{remark}

\begin{proof}
    For any $\lambda>0$, \eqref{eq:stochastic_signature_dilation}, \eqref{eq:triplet_dilation} give $\delta_\lambda\Sig{\lvyalg}=\Sig{\delta_\lambda\lvyalg}$ and $K^\lambda=K\circ\delta_{\lambda^{-1}}$.
    Since $|1_{\Norm{x}>1}-1_{\Norm{\delta_\lambda x}>1}| \le C(1\wedge\Norm{x}^2)$ for a suitable $C>0$, we readily see that \eqref{eq:exponential_moments} is equivalent to
    \begin{equation*}\int_0^t\int_{\mathfrak{g}^N(V)} \indic{\Norm{x}> 1}(e^{ \Vert  x \Vert_1} - 1) K^{\lambda}_s(\dd x)\dd s < \infty.\end{equation*}
    Hence we can assume w.l.o.g. that $\lambda = 1$.
	By \eqref{eq:free_development_bound} it suffices to show that \eqref{eq:exponential_moments} implies that 
    $\int_0^t \Vert \mathfrak{y}(u) \Vert_1 \dd{u} < \infty$.
    To this end, note that also from \eqref{eq:free_development_bound} we have for any $\bx \in \TT_0$ the estimate 
    \begin{equation*}\Vert\exp_\otimes(\bx)-1\Vert_1 ~=~ \Vert\exp_\otimes(\bx)\Vert_1 - 1 ~\le~ e^{ \Vert \bx \Vert_1} - 1.\end{equation*}
    Using \eqref{eq:young_T1} we then also get
    \begin{equation*}
    \Vert\exp_\otimes(\bx)-1- \bx  \Vert_1 
    ~=~ \left\Vert\int_0^1 (\exp_\otimes(t \bx)-1)\otimes \bx \d{t} \right\Vert_1 %
    ~\le~  \int_0^1 (e^{t\Vert \bx \Vert_1}- 1) \Vert\bx\Vert_1 \d{t}, %
    \end{equation*}
    and the right-hand side again equals $e^{\Vert \bx\Vert_1} - 1 - \Vert\bx\Vert_1$.
    This allows us to estimate 
    \begin{align*}
        \Vert \mathfrak{y}(t) \Vert_1 
        &\le \Vert b(t) \Vert_1 +  \frac{1}{2}\Vert  a(t) \Vert_1 + \int_{\mathfrak{g}^N(V)} (e^{\Vert x \Vert_1} - 1 - \Vert x\Vert_1 \indic{\Norm{x}\le 1})  K_t(\d{x}).
    \end{align*}
    In view of assumption \eqref{eq:exponential_moments} and the integrability of $\Vert b\Vert_1 + \Vert a\Vert_1 \le C(\Norm{b} + \Norm{a})$ for suitable $C>0$ depending only on $d$ and $N$, it thus suffices to conclude that 
    \begin{align*}
        \int_0^t\int_{\mathfrak{g}^N(V)} \indic{\Norm{x}\le 1}(e^{\Vert x \Vert_1} - 1 - \Vert x\Vert_1 )  K_s(\d{x})\dd{s} < \infty, \qquad t \ge 0,
    \end{align*}
    which follows from the fact that $K_t(\dd{x})\dd{t}$ is assumed to be a L\'evy measure and the fact that 
    $|e^{\Vert x \Vert_1} - 1 - \Vert x\Vert_1| \le C \Norm{x}^2 $ for $x\in\mathfrak{g}^N(V)$ with $\Norm{x} \le 1$ and choosing $C$ sufficiently large.
\end{proof}

\subsection{An infinite-dimensional equation for the general case}\label{sec:infinite_sys}

We now present the extended PDE system for the expected signature kernel in the most general case of $G^N(V)$-valued inhomogeneous L\'evy processes.
Recall the adjoint left and right multiplications $\dualLeftOp$ and $\dualRightOp$ from \cref{sec:adjoint_multiplication}, in particular \eqref{eq:dual_multiplication}.
For more convenient notation, we define for all $\bx, \by \in \TT^1$:
\begin{equation*}\dualLeftZ{\bx}{\by} ~=~ \dualLeft{\bx}{\by} - \langle\bx, \by\rangle ~=~ \dualLeft{\bx}{\by} - \pi_0(\dualLeft{\bx}{\by})\; \in \TT^1_0,\end{equation*}
and analogously $\dualRightZ{\bx}{\by} = \dualRight{\bx}{\by} - \langle\bx, \by\rangle = \dualRight{\bx}{\by} - \pi_0(\dualRight{\bx}{\by}) \in \TT^1_0$.

\begin{proposition}\label{prop:extended_pde_representation}
    Let $\mathfrak{y}, \widetilde{\mathfrak{y}}: [0,\infty) \to \TT_0$ be measurable such that
    $\int_0^t \Vert \mathfrak{y}(u) \Vert_1 \dd{u}$, $ \int_0^t \Vert \widetilde{\mathfrak{y}}(u) \Vert_1 \dd{u} < \infty$ for all $t\ge0$
    and define their free developments by $\esig = \mathcal{S}(\int_0^\cdot \mathfrak{y}(t)\dd{t})$ and $\widetilde{\esig} = \mathcal{S}(\int_0^\cdot \widetilde{\mathfrak{y}}(t)\dd{t})$.
    Then the triplet $(u, \phi, \widetilde{\phi}): [0,\infty)^2 \to \R \times \TT^1_0 \times \TT^1_0$ defined by
    \begin{align}\label{eq:solution_infinte}
        u(s,t) = \langle \esig(s), \widetilde{\esig}(t) \rangle, \qquad \phi(s,t) := \dualLeftZ{\widetilde{\esig}(t)}{\esig(s)},\qquad
        \widetilde{\phi}(s,t) := \dualLeftZ{\esig(s)}{\widetilde{\esig}(t)},
    \end{align}
    is a continuous solution to the following system of integral equations for all $s,t\ge0$:
    \begin{equation}\label{eq:extended_pde}
    \left\{
    \begin{split}
        u(s,t) =&~ 1 +\int_0^s\int_0^t \Big\{ u(r,v) \langle \mathfrak{y}(r), \widetilde{\mathfrak{y}}(v)\rangle + \langle \phi(r,v), \dualRightZ{\mathfrak{y}(r)}{\widetilde{\mathfrak{y}}(v)}\rangle\\
        &\qquad\qquad\qquad\qquad \qquad\qquad\qquad\qquad  +  \langle \widetilde{\phi}(r,v), \dualRightZ{\widetilde{\mathfrak{y}}(v)}{\mathfrak{y}(r)}\rangle\Big\} \dd{v}\dd{r} \\
        \phi(s,t) =&~ \int_0^s \left\{ u(r,t) \mathfrak{y}(r) + \phi(r,t) \otimes \mathfrak{y}(r) + \dualLeftZ{\widetilde{\phi}(r,t)}{\mathfrak{y}(r)} \right\}\dd{r}\\
        \widetilde\phi(s,t) =&~ \int_0^t \left\{ u(s,r) \widetilde{\mathfrak{y}}(r) + \widetilde\phi(s,r) \otimes \widetilde{\mathfrak{y}}(r) + \dualLeftZ{{\phi}(s,r)}{\widetilde{\mathfrak{y}}(r)} \right\}\dd{r}
        \end{split}\right. .
    \end{equation}
\end{proposition}

The derivation of the above system follows by observing that the arguments used in the proof of \cite[Theorem 5.3]{lemercier2024high} do not rely on the group-like structure of signatures of smooth rough paths, but merely on the fact that these signatures solve an ODE in the extended tensor algebra.
The same arguments therefore extend to the free developments introduced in \cref{sec:free_developments}.
For completeness, we provide a full derivation of the system in \cref{sec:proof_derivation_pde}.

Applied to the expected signatures of inhomogeneous L\'evy processes, the above system provides an analytic characterization of the expected signature kernel once uniqueness of its solution has been established.
The following result establishes this well-posedness, with the required a priori estimate proved in \cref{sec:uniquness}.

\begin{theorem}\label{thm:main}
    Under the assumptions of \cref{prop:extended_pde_representation}, the system \eqref{eq:extended_pde} admits a unique continuous solution $(u, \phi, \widetilde{\phi}): [0,\infty)^2 \to \R \times \TT^1_0 \times \TT^1_0$, which is therefore given by \eqref{eq:solution_infinte}.

    In particular, for inhomogeneous L\'evy processes
    \begin{equation*}\lvyalg \sim \mathscr{L}^N(b, a, K) \quad\text{and}\quad \widetilde{\lvyalg} \sim \mathscr{L}^N(\widetilde{b}, \widetilde{a}, \widetilde{K}),\end{equation*}
    each satisfying the moment condition \eqref{eq:exponential_moments} with $\lambda=1$ for every $t\ge0$, and with characteristic velocities $\mathfrak{y}$ and $\widetilde{\mathfrak{y}}$ defined in \eqref{def:frak_eta}, the expected signature kernel is given by
    \begin{equation*}\big\langle \E[\Sig{\lvyalg}_{0,s}], \E[\Sig{\widetilde{\lvyalg}}_{0,t}] \big\rangle = u(s,t), \qquad s,t\ge0,\end{equation*}
    where $(u,\phi,\widetilde\phi)$ is the unique continuous solution to \eqref{eq:extended_pde}.
\end{theorem}

\begin{proof}
Existence follows from \cref{prop:extended_pde_representation}.
Uniqueness directly follows from the linearity of the system and the a priori estimate in \cref{lem:apriori}.
The final statement follows from \cref{cor:exponential_moments}, which yields the required integrability of the characteristic velocities, and \cref{prop:levy_esig}, which identifies the expected signatures as their free developments.
\end{proof}

\begin{remark}
    In the case of continuous characteristic velocities, the above system can be written in differential form, with the following PDE for the kernel:
    \begin{equation*}\frac{\partial^2}{\partial{s}\partial{t}} u(s,t) = u(s,t) \langle \mathfrak{y}(s), \widetilde{\mathfrak{y}}(t)\rangle + \langle \phi(s,t), \dualRightZ{\mathfrak{y}(s)}{\widetilde{\mathfrak{y}}(t)}\rangle +  \langle \widetilde{\phi}(s,t), \dualRightZ{\widetilde{\mathfrak{y}}(t)}{\mathfrak{y}(s)}\rangle,\end{equation*}
    with boundary conditions $u(0, \cdot) = u(\cdot, 0) = 1$ and corresponding coupled ODEs for $\phi$ and $\widetilde{\phi}$.
\end{remark}

\subsection{Truncated characteristics and the continuous case}
\label{sec:truncated_pde}

For continuous inhomogeneous L\'evy processes, the characteristic velocities have only finitely many nonzero levels. Consequently, the expected signature kernel can be computed exactly by solving a finite-dimensional system of integral equations. 

In the presence of nontrivial L\'evy measures the integral system from the previous section \eqref{eq:extended_pde} is inherently infinite-dimensional, since the characteristic velocities include all (multivariate) moments of the L\'evy measure in their levels.
In this section we provide an error estimate for the approximation of \eqref{eq:extended_pde} by a finite-dimensional system obtained through truncating the characteristic velocities.
Compared to direct truncation of the expected signature, truncating the characteristic velocity can yield provably faster approximation rates, as we discuss subsequently.

\newcommand{\n}{{N-1}}
\newcommand{\m}{{M-1}}
\newcommand{\nn}{{N}}
\newcommand{\mm}{{M}}
\newcommand{\nm}{{P-1}}
\newcommand{\nnmm}{{P}}
\newcommand{\nmm}{{P_{\diamond}}}
\newcommand{\nnm}{{P_{\blackdiamond}}}
\newcommand{\varvar}{{K}}
\newcommand{\var}{{K-1}}

\begin{theorem}\label{thm:truncated}
    Let $\mathfrak{y}, \widetilde{\mathfrak{y}}: [0,\infty) \to \TT_0$ be as in \cref{thm:main}.
    Then for any $\mm, \nn \in \mathbb{N}_{\ge 1}$ there is a unique continuous solution 
    \begin{equation*}(w, f, \widetilde{f}): [0,\infty)^2\to \R \times T^\n (V) \times T^\m(V),\end{equation*}
    to the system 
    \begin{equation}\label{eq:extended_pde_trunc}
    \left\{
    \begin{split}
        w(s,t) =&~ 1 +\int_0^s\int_0^t \Big\{ w(r,v) \langle \mathfrak{y}^{\nnmm}(r), \widetilde{\mathfrak{y}}^{\nnmm}(v)\rangle + \langle f(r,v), \dualRightZ{\mathfrak{y}^\nmm(r)}{\widetilde{\mathfrak{y}}^\nn(v)}\rangle \\
        &\qquad\qquad\qquad\qquad\qquad\qquad\qquad\qquad\quad+  \langle \widetilde{f}(r,v), \dualRightZ{\widetilde{\mathfrak{y}}^\nnm(v)}{\mathfrak{y}^\mm(r)}\rangle\Big\} \dd{v}\dd{r}, \\
        f(s,t) =& \int_0^s \left\{ w(r,t) \mathfrak{y}^\nmm(r) + f(r,t) \otimes_\n \mathfrak{y}^\nmm(r) + \pi_{(0,\n)}\dualLeftZ{\widetilde{f}(r,t)}{{\mathfrak{y}}^\mm(r)} \right\}\dd{r},\\
        \widetilde{f}(s,t) =& \int_0^t \left\{ w(s,r) \widetilde{\mathfrak{y}}^\nnm(r) + \widetilde{f}(s,r) \otimes_\m \widetilde{\mathfrak{y}}^\nnm(r) + \pi_{(0,\m)}\dualLeftZ{{f}(s,r)}{\widetilde{\mathfrak{y}}^\nn(r)} \right\}\dd{r},
    \end{split}\right.
\end{equation} 
where $\nnmm = \min(\mm,\nn)$, $\nmm = \min(\mm,\n)$, $\nnm = \min(\nn,\m)$, and here for any $\varvar\in\N$ we set $\mathfrak{y}^{\varvar} := \pi_{(0,\varvar)} \mathfrak{y}$, $\widetilde{\mathfrak{y}}^{\varvar} := \pi_{(0,\varvar)} \widetilde{\mathfrak{y}}$.
Denote the first component of this unique solution by $w^{\mm,\nn}:[0,\infty)^{2}\to \R$ and let $u:[0,\infty)^2\to\R$ be the solution to the augmented system \eqref{eq:extended_pde}.
Then the following estimate holds
    \begin{equation*}|u(s,t) - w^{\mm,\nn}(s,t)| \le C_{s}\widetilde{C}_{t}\left( \int_0^s\Vert \mathfrak{y}(r) - \pi_{(0,\mm)}\mathfrak{y}(r)\Vert_1\dd{r} +  \int_0^t\Vert \widetilde{\mathfrak{y}}(r) - \pi_{(0,\nn)}\widetilde{\mathfrak{y}}(r)\Vert_1\dd{r}\right),\end{equation*}
    for all $s,t\ge0$ and $\mm,\nn\in\N$ where $C_{s}= e^{\int_0^s \Vert \mathfrak{y}(r)\Vert_1\dd{r}}$ and $\widetilde{C}_{t}= e^{\int_0^t  \Vert\widetilde{\mathfrak{y}}(v)\Vert_1\dd{v}}$.
    
    In particular, when $\mathfrak{y}$ and $\widetilde{\mathfrak{y}}$ are given as characteristic velocities of \emph{continuous} inhomogeneous L\'evy processes $\lvyalg$ and $\widetilde{\lvyalg}$, then there exist minimal truncation levels $\mm_0, \nn_0 \in \N_{\ge 1}$ such that $\mathfrak{y} = \pi_{(0, \mm_0)}\mathfrak{y}$ and  $\widetilde{\mathfrak{y}} = \pi_{(0, \nn_0)}\widetilde{\mathfrak{y}}$. Consequently,
    \begin{equation*} \langle \E[\Sig{\lvyalg}_{0,s}], \E[\Sig{\widetilde{\lvyalg}}_{0,t}] \rangle = w^{\mm_0, \nn_0}(s,t), \qquad s,t\ge0.\end{equation*}
\end{theorem}
\begin{proof}
We first define the free developments of the truncated paths $\mathfrak{y}^{\mm}$ and $\widetilde{\mathfrak{y}}^{\nn}$ by
$
\pmb{\nu}=\mathcal{S}(\int_0^\cdot\mathfrak{y}^{\mm}(t)\dd t)$ and 
$\widetilde{\pmb{\nu}}=\mathcal{S}(\int_0^\cdot\widetilde{\mathfrak{y}}^{\nn}(t)\dd t)$.
Then it follows from \cref{thm:main} that
\begin{equation*}
(w(s,t),\phi(s,t),\widetilde\phi(s,t))
:=
\left(
\langle\pmb{\nu}(s),\widetilde{\pmb{\nu}}(t)\rangle,
\dualLeftZ{\widetilde{\pmb{\nu}}(t)}{\pmb{\nu}(s)},
\dualLeftZ{\pmb{\nu}(s)}{\widetilde{\pmb{\nu}}(t)}
\right)
\end{equation*}
is the unique solution to \eqref{eq:extended_pde} with underlying paths $\mathfrak{y}^{\mm}$ and $\widetilde{\mathfrak{y}}^{\nn}$. We verify that
\begin{equation*}
(w,f,\widetilde f):=(w,\pi_{(0,\n)}\phi,\pi_{(0,\m)}\widetilde\phi)
\end{equation*}
solves \eqref{eq:extended_pde_trunc}.
To this end, we first note that for any $\bx\in\TT_0$ and $\by\in T^\varvar(V)$,
\begin{align}\label{eq:dual_truncation_relation}
\dualLeftZ{\bx}{\by}
=
\dualLeftZ{\pi_{(0,\var)}\bx}{\by}
\in T^\var(V),
\qquad
\dualRightZ{\bx}{\by}
=
\dualRightZ{\pi_{(0,\var)}\bx}{\by}
\in T^\var(V).
\end{align}
Indeed, by \eqref{eq:adjoint_degree}--\eqref{eq:adjoint_equal_degree}, $\dualLeft{\bx^{(k)}}{\by^{(n)}}\in V^{\otimes(n-k)}$ for $n\ge k$ and vanishes for $n<k$, while the equal-degree term is scalar and is removed in $\dualLeftZ{}{}$. Hence
\begin{align*}
\dualLeftZ{\bx}{\by}
&=
\sum_{n=1}^{\varvar}\sum_{k=1}^n
\dualLeft{\bx^{(k)}}{\by^{(n)}}
-
\langle\pi_{(0,\varvar)}\bx,\by\rangle\\
&=
\sum_{n=1}^{\varvar}\sum_{k=1}^{n\wedge\var}
\dualLeft{\bx^{(k)}}{\by^{(n)}}
-
\langle\pi_{(0,\var)}\bx,\by\rangle
=
\dualLeftZ{\pi_{(0,\var)}\bx}{\by}.
\end{align*}
The right-sided identity follows analogously.
Using \eqref{eq:dual_truncation_relation}, the cross-terms in the equation for $w$ become
\begin{align*}
\langle\phi(s,t),\dualRightZ{\mathfrak{y}^{\mm}(s)}{\widetilde{\mathfrak{y}}^{\nn}(t)}\rangle
&=
\langle f(s,t),\dualRightZ{\mathfrak{y}^{\nmm}(s)}{\widetilde{\mathfrak{y}}^{\nn}(t)}\rangle,\\
\langle\widetilde\phi(s,t),\dualRightZ{\widetilde{\mathfrak{y}}^{\nn}(t)}{\mathfrak{y}^{\mm}(s)}\rangle
&=
\langle\widetilde f(s,t),\dualRightZ{\widetilde{\mathfrak{y}}^{\nnm}(t)}{\mathfrak{y}^{\mm}(s)}\rangle.
\end{align*}
Thus, the equation for $w$ in \eqref{eq:extended_pde} reduces to the corresponding equation in \eqref{eq:extended_pde_trunc}. Next, projecting the equation for $\phi$, using that $\pi_{(0,\n)}:(\TT,\otimes)\to(T^\n(V),\otimes_\n)$ is an algebra homomorphism, $\pi_{(0,\n)}\mathfrak{y}^{\mm}=\mathfrak{y}^{\nmm}$, and again \eqref{eq:dual_truncation_relation}, gives the equation for $f$. The equation for $\widetilde f$ follows analogously. This proves existence.

Uniqueness can be deduced from the uniqueness of the infinite-dimensional system in \cref{thm:main}. For completeness we provide a full proof in \cref{sec:proof_truncated_uniqueness}; compare also the direct argument for the $N=M=2$ case in \cite[Proof of Theorem~4.3]{lemercier2024high}.

Now denote $\eta=\int_0^\cdot\mathfrak{y}(t)\dd t$ and $\widetilde\eta=\int_0^\cdot\widetilde{\mathfrak{y}}(t)\dd t$. By construction and uniqueness,
\begin{equation*}
w^{\mm,\nn}(s,t)
=
\left\langle
\mathcal S(\pi_{(0,\mm)}\eta)_s,
\mathcal S(\pi_{(0,\nn)}\widetilde\eta)_t
\right\rangle.
\end{equation*}
Consequently, the triangle inequality gives
\begin{align*}
|u(s,t)-w^{\mm,\nn}(s,t)|
&\le
\Vert\mathcal S(\eta)_s\Vert_1
\Vert\mathcal S(\widetilde\eta)_t-
\mathcal S(\pi_{(0,\nn)}\widetilde\eta)_t\Vert_1\\
&\quad+
\Vert\mathcal S(\pi_{(0,\nn)}\widetilde\eta)_t\Vert_1
\Vert\mathcal S(\eta)_s-
\mathcal S(\pi_{(0,\mm)}\eta)_s\Vert_1.
\end{align*}
The asserted error estimate now follows from \eqref{eq:free_development_bound} and \eqref{eq:free_development_truncation}.

Finally, if $\mathfrak{y}$ is the characteristic velocity of a continuous inhomogeneous L\'evy process $\lvyalg\sim\mathscr L^M(b,a,0)$, then by \eqref{def:frak_eta} it holds $\mathfrak{y}=\pi_{(0,2M)}\mathfrak{y}$, and hence there exists a minimal $\mm_0\in\N_{\ge1}$ such that $\mathfrak{y}=\pi_{(0,\mm_0)}\mathfrak{y}$. The same holds for $\widetilde{\mathfrak{y}}$ with a minimal level $\nn_0$. The error estimate therefore vanishes for $\mm\ge\mm_0$ and $\nn\ge\nn_0$, and \cref{thm:main} yields
\begin{equation*}
\langle\E[\Sig{\lvyalg}_{0,s}],\E[\Sig{\widetilde{\lvyalg}}_{0,t}]\rangle
=
w^{\mm_0,\nn_0}(s,t),
\qquad s,t\ge0.
\end{equation*}
\end{proof}

\begin{example}
    Put forward in \cite{bellingeri2022smooth}, any smooth group-valued path $X:[0,\infty) \to G^N(V)$ is the development of a smooth path $\lvyalg: [0,\infty) \to \mathfrak{g}^N(V)$ and is called a \emph{smooth rough path} (see also \cref{sec:free_developments}).
    Then $\lvyalg \sim \mathscr{L}^N(\dot{\lvyalg}, 0, 0)$, where $\dot\lvyalg$ is the diagonal derivative of $X$.
    In this case, \cref{thm:truncated} recovers the PDE--ODE system for the signature kernel presented in \cite[Theorem~5.3]{lemercier2024high}.
\end{example}

The difference to direct truncation of the expected signature lies in the order of truncation and free development. On either side of the kernel, the two procedures correspond to
\begin{equation*}
\mathcal{S}\left(\int_0^\cdot\pi_{(0,M)}\mathfrak y(r)\dd r\right)
\qquad\text{vs.}\qquad
\pi_{(0,M)}\mathcal{S}\left(\int_0^\cdot\mathfrak y(r)\dd r\right),
\end{equation*}
respectively. 
We provide an extended theoretical discussion of the corresponding approximation errors in \cref{sec:tail_truncation}. There, upper bounds for truncation of the characteristic velocity are combined with lower bounds for direct expected-signature truncation in suitable one-dimensional compound Poisson examples, showing that the difference in rates is genuine. For Gaussian and compactly supported jump laws, truncation of the characteristic velocity has factorial-type decay, whereas direct truncation of the expected signature can decay substantially slower and, in particular, need not have factorial-type decay. For exponentially distributed jumps, truncation of the characteristic velocity has geometric decay, whereas direct truncation is slower by an additional subexponential factor.

\begin{example}\label{exmpl:circle_jumps}
We illustrate the difference between the two truncation procedures for the self-expected signature kernel of a homogeneous compound Poisson process $\lvyalg\sim\mathscr{L}^2(0,0,K)$ on $[0,1]$, with time-independent L\'evy measure
\begin{equation*}
K_t
=
\frac{\lambda}{2\pi}
\int_0^{2\pi}
\delta_{c+r(\cos\theta,\sin\theta)}
\dd{\theta},
\qquad
t\in[0,1].
\end{equation*}
We choose $\lambda=1$, $c=(0.45,-0.20)$ and $r=1.65$. The translation $c\neq0$ removes the exact symmetry cancellations of the centered circle law, while the bounded support places the example in the factorial-decay regime described above. Moreover, $r-|c|>1$, so the support of $K_t$ lies outside the unit ball.
For truncation levels $M\in\{1,\dots,13\}$, we compare the approximation
\begin{equation*}
C_M^{\mathrm{char}}
:=
w^{M,M}(1,1),
\end{equation*}
where $w^{M,M}$ denotes the first component of the truncated system in \cref{thm:truncated}, with the direct expected-signature truncation
\begin{align*}
C_M^{\mathrm{sig}}
&:=
\sum_{n=0}^M
\left|
\pi_n\E\big[\Sig{\lvyalg}_{0,1}\big]
\right|^2 =
\sum_{n=0}^M
\left|
\pi_n\exp_\otimes\left(
\lambda\left(
\frac{1}{2\pi}
\int_0^{2\pi}
\exp_\otimes\big(c+r(\cos\theta,\sin\theta)\big)
\dd{\theta}
-1
\right)
\right)
\right|^2.
\end{align*}
Here we have used \cref{prop:levy_esig} and the fact that the characteristic velocity is time-independent. The circle integral can be evaluated explicitly level by level.

Numerically, $w^{M,M}$ is computed on a dyadic grid with $2^7$ points. Since the exact value of the self-expected signature kernel is not available in closed form, we use $C_{\mathrm{ref}}:=C_{16}^{\mathrm{char}}$ as a high-order reference value. The left panel of \cref{fig:trunc_circle} shows the errors
\begin{equation*}
e_M^{\mathrm{char}}
:=
\left|C_{\mathrm{ref}}-C_M^{\mathrm{char}}\right|,
\qquad
e_M^{\mathrm{sig}}
:=
\left|C_{\mathrm{ref}}-C_M^{\mathrm{sig}}\right|,
\end{equation*}
while the right panel shows the corresponding rescaled ratios $M e_M^{\mathrm{char}}/e_{M-1}^{\mathrm{char}}$ and $M e_M^{\mathrm{sig}}/e_{M-1}^{\mathrm{sig}}$.

The characteristic-velocity truncation yields substantially smaller errors already at moderate truncation levels. Moreover, its rescaled ratios remain of moderate size, consistent with factorial-type decay, whereas the ratios for direct expected-signature truncation grow over the displayed range.
\end{example}

\begin{figure}[htbp]
    \centering
    \includegraphics[width=0.9\textwidth]{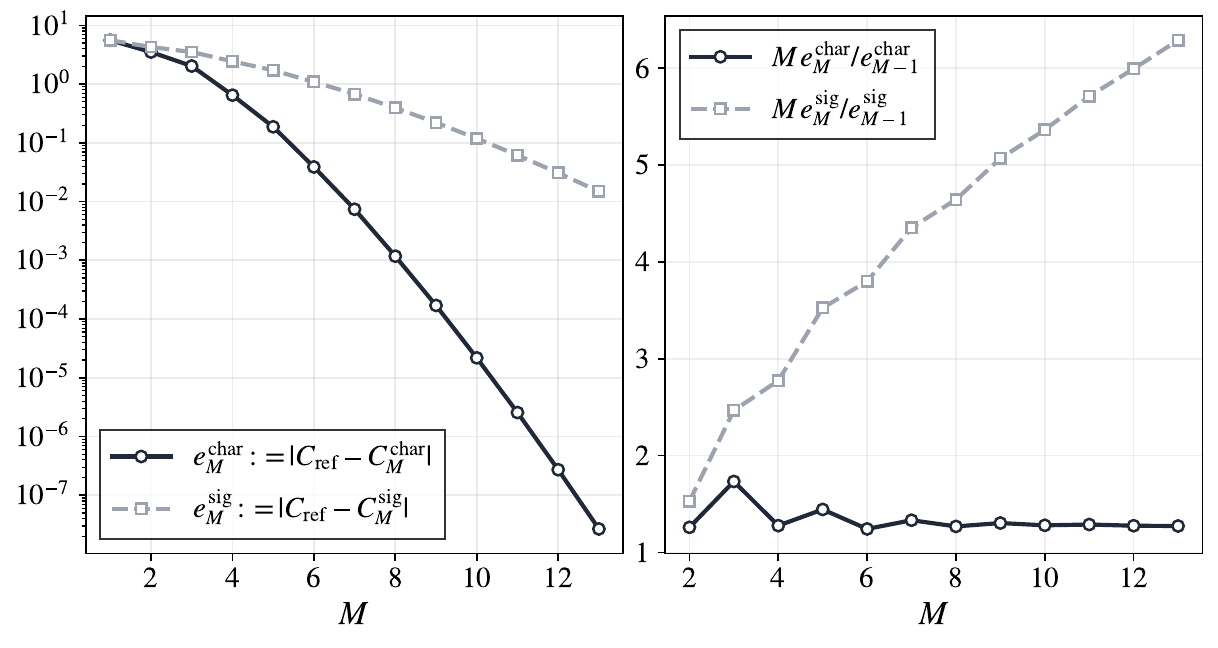}
    \caption{Truncation errors for the translated circle jump law. Left: absolute errors $e_M^{\mathrm{char}}$ and $e_M^{\mathrm{sig}}$. Right: the rescaled ratios $M e_M^{\mathrm{char}}/e_{M-1}^{\mathrm{char}}$ and $M e_M^{\mathrm{sig}}/e_{M-1}^{\mathrm{sig}}$.}
    \label{fig:trunc_circle}
\end{figure}

\subsection{Second level characteristics}\label{sec:second_level_char}

We now specialize to the case where the characteristic velocities are at most of second tensor level.
In this case, the adjoint multiplication can be expressed using matrix multiplication, resulting in remarkably accessible formulas.
For $x\in V$ and $\bx, \by\in V^{\otimes 2}$ we use the following notation for the matrix multiplication: 
\begin{align*}
\by.x := \sum_{i=1}^d e_i \sum_{j=1}^d \by^{ij} x^{j} \;\in V,\qquad
\by.\bx := \sum_{i,k=1}^d e_i\otimes e_k \sum_{j=1}^d \by^{ij}\bx^{jk} \;\in V^{\otimes 2}.
\end{align*}
Next recall that the free step-$2$ nilpotent Lie algebra decomposes as
\begin{equation*}\mathfrak{g}^2(V) = V \oplus \mathrm{Anti}(V),\end{equation*}
where $\mathrm{Anti}(V)\subset V^{\otimes 2}$ denotes the subspace of antisymmetric level-2 tensors (i.e., $\mathrm{Anti}(V)=\{\mathfrak{a}\in V^{\otimes 2}: \mathfrak{a}^\top=-\mathfrak{a}\}$).
Recall also from \cref{sec:inhom_levy} that $\mathfrak{S}_+^1(V) \subset V^{\otimes 2}$ denotes the set of positive semidefinite level-2 tensors.
A general inhomogeneous L\'evy process $\gamma$ with at most second-level characteristic velocity is necessarily continuous and of the form 
\begin{align}\label{eq:second_level_levy}
    \gamma = \int_0^\cdot (b(t) + \mathfrak{a}(t)) \dd{t} + \sum_{k=1}^m\int_0^\cdot \sigma_k(t) \dd B^k_t,
\end{align}
where $b: [0, \infty) \to V$ and $\mathfrak{a}: [0, \infty) \to \mathrm{Anti}(V)$ are integrable, $B$ is an $m$-dimensional Brownian motion and $\sigma_k:[0,\infty)\to V$ ($k=1,\dots,m$) are square-integrable.
The characteristic velocity is precisely \begin{equation*}\mathfrak{y}(t)=b(t)+\mathfrak{a}(t)+\tfrac12\,a(t) \;\in T^2(V), \qquad t\ge0,\end{equation*}
where $a = \sum_{k=1}^m\sigma_k \otimes\sigma_k: [0,\infty) \to \mathfrak{S}_+^1(V)$.
The ``area term'' $\mathfrak{a}$ plays a role in certain applications (see \cref{sec:application}).
Setting $\mathfrak{a} \equiv 0$ leads to the case of a general Brownian motion with time-dependent volatility matrix and drift; however, as seen below, this does not materially simplify the formulas.

\begin{corollary}\label{cor:level2}
    Let $\lvyalg \sim \mathscr{L}^2(b + \mathfrak{a}, a, 0)$ and $\widetilde{\lvyalg} \sim \mathscr{L}^2(\widetilde{b} + \widetilde{\mathfrak{a}}, \widetilde{a}, 0)$ as in \eqref{eq:second_level_levy} where $b, \widetilde{b}: [0,\infty)\to V$, $\mathfrak{a}, \widetilde{\mathfrak{a}}: [0,\infty) \to \mathrm{Anti}(V)$ and $a, \widetilde{a}: [0,\infty) \to \mathfrak{S}_+^1(V)$ are integrable.
    Then there exists a unique continuous solution $(u, f, \widetilde{f}): [0,\infty)^2 \to \R\times V\times V$ to the system
    \begin{equation}\label{eq:continuous_case}\left\{
    \begin{split}
            u(s,t) =&~ 1 +\int_0^s\int_0^t \bigg\{ u(r,v) \Big( \langle b(r), \widetilde{b}(v) \rangle + \langle \mathfrak{a}(r), \widetilde{\mathfrak{a}}(v) \rangle + \frac14\langle a(r), \widetilde{a}(v) \rangle\Big)\\
            &\qquad\qquad\qquad\qquad\qquad
            + \langle \widetilde{c}(v).f(r,v), b(r)\rangle 
            +  \langle c(r).\widetilde{f}(r,v), \widetilde{b}(v)\rangle\bigg\} \dd{v}\dd{r} \\
        f(s,t) =&~ \int_0^s \Big\{ u(r,t) b(r)  + c(r).\widetilde{f}(r,t)\Big\}\dd{r}\\
        \widetilde{f}(s,t) =&~ \int_0^t \Big\{ u(s,r) \widetilde{b}(r) +  \widetilde{c}(r).{f}(s,r)\Big\}\dd{r}
        \end{split},\right.
    \end{equation} 
    for all $s,t\ge0$, where $c:= \frac12 a - \mathfrak{a}$ and $\widetilde{c}:= \frac12 \widetilde{a} - \widetilde{\mathfrak{a}}$; and the 
    expected signature kernel is given by \begin{equation*}\langle \E[\Sig{\lvyalg}_{0,s}], \E[\Sig{\widetilde{\lvyalg}}_{0,t}] \rangle = u(s,t), \qquad s,t\ge 0.\end{equation*}
\end{corollary}
The proof is provided in \cref{sec:proof_second_level}.
\begin{remark}
    In the presence of jumps, say $\gamma \sim \mathscr{L}^2(b + \mathfrak{a}, a, K)$ where $(K_t)_{t\ge0}$ is a family of L\'evy measures on $V$, the truncated system \eqref{eq:extended_pde_trunc} at level $N=M=2$ takes the same form as \eqref{eq:continuous_case} with modified coefficients replacing $b$ with $b + \E_{\xi\sim K}[\xi 1_{|\xi|\ge 1}]$ and $a$ with $a + \E_{\xi \sim K}[\xi^{\otimes 2}]$.
\end{remark}

\begin{example}
    For the case of smooth $V$-valued paths $\lvyalg$ and $\widetilde{\lvyalg}$, i.e., $b=\dot{\gamma}$, $\widetilde{b}=\dot{\widetilde{\gamma}}$ and $\mathfrak{a} = \widetilde{\mathfrak{a}} = a = \widetilde{a} \equiv 0$, we immediately recover from the above the classical Goursat problem for the signature kernel \cite{salvi2021signature}:
    \begin{equation}\label{eq:goursat}
        \frac{\partial^2}{\partial{s}\partial{t}}u_{\alpha}(s,t) = u_{\alpha}(s,t)\alpha(s,t), \qquad u_{\alpha}(s, 0) = u_{\alpha}(0,t) = 1,
    \end{equation}
    for all $s,t\ge 0$ with $\alpha(s,t) = \langle \dot{\lvyalg}(s), \dot{\widetilde{\lvyalg}}(t)\rangle$.
\end{example}
\begin{example}\label{exmpl:time-dependent-bm}
    Consider two time-dependent Brownian motions $B \sim \mathscr{L}^1(0, a, 0)$ and $W \sim \mathscr{L}^1(0, \widetilde{a}, 0)$.
    Then, directly from the above corollary, the expected signature kernel solves the Goursat problem; more precisely,
    \begin{equation*}\langle \E[\Sig{B}_{0,s}], \E[\Sig{{W}}_{0,t}] \rangle = u_\alpha(s,t), \quad\text{with}\quad \alpha(s,t) = \frac{1}{4}\langle a(s), \widetilde{a}(t)\rangle, \quad s,t\ge0.\end{equation*}
    In the homogeneous case, where $\alpha(\cdot,\cdot) \equiv \alpha \in \R$, the solution has the explicit form $u(s,t) = I_0(2\sqrt{\alpha st})$ where $I_0$ is the modified Bessel function of order zero\footnote{For the definition and properties of the modified Bessel functions see \cite[Section~9.6]{abramowitz1972handbook}.} (see \cite{salvi2021signature}).
\end{example}

\begin{example}
    Considering a deterministic path $\gamma \sim \mathscr{L}^1(b, 0,0)$ and a time-dependent Brownian motion $\widetilde{\gamma} \sim \mathscr{L}^1(0, a,0)$ the system \eqref{eq:continuous_case} reduces to
    \begin{align*}
        \frac{\partial^2}{\partial{s}\partial{t}}u(s,t) =&~ \frac{1}{2}\int_0^s u(r,t) \langle b(r), {a}(t).b(s)\rangle \d{r}.
    \end{align*}
    For the homogeneous case $a(\cdot) \equiv a$ we arrive at an expression for $u$ in other ways. 
    Using Fawcett's formula $\widetilde{\esig}(t) = \exp_\otimes(\frac{1}{2}t a)$ one can obtain the Dyson-type expansion
    \begin{align*}
        u(s,t) = 1 + \sum_{n=1}^{\infty} \frac{t^{n}}{n! 2^{n}} \int_{0\le t_1 \le \dots \le t_{2n} \le s} \langle b(t_1),{a}.b(t_2)\rangle \cdots \langle b(t_{2n-1}),{a}.b(t_{2n})\rangle \d{t_1}\cdots\d{t_{2n}},
    \end{align*}
    and we immediately verify consistency.
    In \cite{cass2024weighted} the above formula for the standard Brownian case $a = \Id$ was related to the Cartan development of the path $\lvyalg$ in the isometry group of the hyperbolic space, thus allowing to express $u$ in terms of a contour integral over a parameterized family of matrix-valued ODE solutions.
\end{example}

\section{Signature maximum mean discrepancy for conditionally Gaussian martingales}\label{sec:application}
We consider the signature maximum mean discrepancy (signature-MMD, cf.~\eqref{eq:mmd_general_def})
\begin{align*}
    \dMMD(P,Q)=\left\Vert\E_{X\sim P}[\Sig{X}_{0,T}]-\E_{Y\sim Q}[\Sig{Y}_{0,T}]\right\Vert_2.
\end{align*}
The finite-sample power of MMD tests can be affected by the variance of the empirical MMD statistic \cite{gretton2012optimal,sutherland2017generative}; related finite-sample effects for the signature-MMD are documented in \cite{horvath2025signature}.

We show that, for samples from conditionally Gaussian martingale laws, part of the variance of the unbiased empirical squared signature-MMD statistic can be removed by conditioning on the quadratic variation. We first establish the resulting orthogonal decomposition, then derive the corresponding variance decomposition, and finally introduce finite-grid statistics and derive their computation.

\subsection{Conditional Gaussianity and variance reduction}

Denote by $\mathcal Q$ the set of probability measures $P$ on the continuous path space $\Omega=C([0,T];\mathbb R^d)$ such that $P$ is the distribution of a martingale $X$ admitting, on some filtered probability space, a representation
\begin{align}\label{eq:time-dependent-BM}
    X_t=\sum_{i=1}^m\int_0^t\sigma_s^i\dd B_s^i,\qquad t\in[0,T],
\end{align}
where $B$ is a standard $m$-dimensional Brownian motion and $\sigma^1,\dots,\sigma^m$ are predictable $V\cong\R^d$-valued processes satisfying
\begin{equation}\label{eq:moment_condition_cond_gaus}
    \E\left[\exp\bigl(2\,\mathrm{tr}\langle X\rangle_T\bigr)\right]<\infty, \qquad \mathrm{tr}\langle X\rangle_T := \int_0^T\sum_{i=1}^m|\sigma_s^i|^2\dd s.
\end{equation}
Let $\mathcal Q_G\subset\mathcal Q$ denote the subclass of laws admitting a representation \eqref{eq:time-dependent-BM} in which $(\sigma^1,\dots,\sigma^m)$ is independent of $B$. By \cite[Propositions~1-1 and 1-4]{jacod1997continuous}, these are precisely the laws in $\mathcal Q$ of conditionally Gaussian martingales with absolutely continuous quadratic variation. Indeed, writing
$a_t:=\sum_{i=1}^m\sigma_t^i\otimes\sigma_t^i$, so that $\langle X\rangle=\int_0^\cdot a_s\dd s$, for $P\in\mathcal Q_G$ it holds
\begin{align}\label{eq:conditional_gaussian_law}
    \mathrm{Law}\big(X\mid\langle X\rangle\big)(\omega)=\mathscr L^1\big(0,a(\omega),0\big)
\end{align}
for $P$-almost every $\omega$.
In accordance with the previous notation, we write $X\sim P\in\mathcal Q$ (resp. $\mathcal Q_G$) if there exists a filtered probability space $(\Omega,\F,(\F_t),\mathbb P)$ satisfying the usual conditions on which $X$ has the representation \eqref{eq:time-dependent-BM} and satisfies \eqref{eq:moment_condition_cond_gaus}.

Conditioning on the quadratic variation, we can thus compute the conditional expected signature of a conditionally Gaussian martingale explicitly from \cref{prop:levy_esig}.

\begin{lemma}\label{lem:conditional_gaussian_orthogonality}
Let $X\sim P\in\mathcal Q$. Then
$
\E[\Vert\Sig{X}_{0,T}\Vert_2]<\infty$ and
$
\E[\Vert\mathcal S(\tfrac12\langle X\rangle)_{0,T}\Vert^2_2]<\infty
$.
In particular,
\begin{equation*}
\esig_P:=\E_{X\sim P}[\Sig{X}_{0,T}]\in\TT^2,\qquad P\in\mathcal Q.
\end{equation*}
Moreover, if $P\in\mathcal Q_G$, then additionally
$
\E[\Vert\Sig{X}_{0,T}\Vert_2^2]<\infty$ and
\begin{align}\label{eq:conditional_expected_signature}
    \E\left[\Sig{X}_{0,T}\,\middle|\,\langle X\rangle\right]=\mathcal S\left(\tfrac12\langle X\rangle\right)_{0,T}.
\end{align}
In particular, the following yields an $L^2(P;\TT^2)$ orthogonal decomposition
\begin{align}\label{eq:signature_characteristic_residual_split}
    \Sig{X}_{0,T}-\esig_P={}&\left(\mathcal S\left(\tfrac12\langle X\rangle\right)_{0,T}-\esig_P\right)+\left(\Sig{X}_{0,T}-\mathcal S\left(\tfrac12\langle X\rangle\right)_{0,T}\right).
\end{align}
\end{lemma}

\begin{remark}
The first assertion concerning the integrability of $\Sig{X}$ for $X\sim P\in\mathcal Q$ follows from the more general estimate
\begin{equation*}
\E\left[\left\Vert\delta_\lambda\Sig{X}_{0,T}\right\Vert_2\right]
\le
\E\left[\exp\left(2\lambda^2\operatorname{tr}\langle X\rangle_T\right)\right]^{1/2},
\qquad \lambda>0.
\end{equation*}
This implies that $\esig_P\in\mathcal R^\infty$ whenever $\operatorname{tr}\langle X\rangle_T$ has exponential moments of all orders. The proof of this estimate, based on a direct It\^o argument applied to the squared norm, is postponed to \cref{lem:martingale_signature_integrability}.
For $X\sim P\in\mathcal Q_G$, the additional independence assumption yields the stronger estimate
\begin{equation*}
\E\left[\left\Vert\Sig{X}_{0,T}\right\Vert_2^2\,\middle|\,\langle X\rangle\right]
\le
\exp\left(2\operatorname{tr}\langle X\rangle_T\right),
\end{equation*}
so the same exponential-moment condition \eqref{eq:moment_condition_cond_gaus} yields square-integrability; see \cref{lem:conditional_gaussian_square_integrability} for the proof, which instead uses the coproduct estimate in \cite[Equation~(2.14)]{chevyrev2026orthogonal}.
\end{remark}

\begin{proof}
The first assertion concerning $\Sig{X}$, and hence $\esig_P\in\TT^2$, follows from \cref{lem:martingale_signature_integrability}. Suppose that $X\sim P\in\mathcal Q$.
Since $a_s=\sum_{i=1}^m\sigma_s^i\otimes\sigma_s^i$ is positive semidefinite, $|a_s|\le\operatorname{tr}(a_s)$, and \eqref{eq:free_development_bound} gives
\begin{equation*}
\left\Vert\mathcal S\left(\tfrac12\langle X\rangle\right)_{0,T}\right\Vert_2^2
\le
\left\Vert\mathcal S\left(\tfrac12\langle X\rangle\right)_{0,T}\right\Vert_1^2
\le
\exp\left(\operatorname{tr}\langle X\rangle_T\right).
\end{equation*}
Thus
$
\E_{X\sim P}[\Vert\mathcal S(\tfrac12\langle X\rangle)_{0,T}\Vert_2^2]<\infty
$
by \eqref{eq:moment_condition_cond_gaus}.
Suppose now that $X\sim P\in\mathcal Q_G$. Then \eqref{eq:conditional_gaussian_law} and \cref{prop:levy_esig} yield \eqref{eq:conditional_expected_signature} coordinatewise. Moreover, \cref{lem:conditional_gaussian_square_integrability} gives $\E_{X\sim P}[\Vert\Sig{X}_{0,T}\Vert_2^2]<\infty$.
Since both sides of the coordinatewise conditional expected-signature identity are therefore square-integrable and $\TT^2$-valued, it yields \eqref{eq:conditional_expected_signature} as a $\TT^2$-valued conditional-expectation identity. The first term on the right-hand side of \eqref{eq:signature_characteristic_residual_split} is $\sigma(\langle X\rangle)$-measurable, whereas the second has conditional expectation zero given $\langle X\rangle$. Hence the two terms are orthogonal in $L^2(P;\TT^2)$.
\end{proof}

We formulate the following results for the one-sample problem of testing $P=Q$ against a specified reference measure $Q\in\mathcal Q_G$. The same conditioning argument applies to two-sample tests by conditioning on the quadratic-variation paths in both samples.

The decomposition in \cref{lem:conditional_gaussian_orthogonality} suggests separating the contribution of the quadratic variation from the remaining signature fluctuations.
Let $X^1,\dots,X^\samples$ be independent with law $P\in\mathcal Q$, and let $\mathbf X^k = \pi_{(0,2)}\Sig{X^k}$ denote the canonical Stratonovich rough-path lift of $X^k$. We regard $(\mathbf X^k,\langle X^k\rangle)$ as the $k$th observation and set
$P^\samples=\frac1\samples\sum_{k=1}^\samples\delta_{(\mathbf X^k,\langle X^k\rangle)}$.
For $\samples\ge2$, we define
\begin{align}
    \dMMDhat(P^\samples,Q)
    &=
    \frac1{\samples(\samples-1)}
    \sum_{\substack{j,k=1\\j\ne k}}^\samples
    \left\langle
        \Sig{X^j}_{0,T}-\esig_Q,
        \Sig{X^k}_{0,T}-\esig_Q
    \right\rangle,
    \notag
    \\
    \dcharhat(P^\samples,Q)
    &=
    \frac1{\samples(\samples-1)}
    \sum_{\substack{j,k=1\\j\ne k}}^\samples
    \left\langle
        \mathcal S\left(\tfrac12\langle X^j\rangle\right)_{0,T}-\esig_Q,
        \mathcal S\left(\tfrac12\langle X^k\rangle\right)_{0,T}-\esig_Q
    \right\rangle.
    \notag
\end{align}
These statistics are unbiased estimators of squared discrepancies and need not be non-negative.

The following proposition shows that, within $\mathcal Q_G$, the signature and quadratic-variation statistics are unbiased estimators of the same squared signature-MMD, while the latter is the conditional expectation of the former given the sampled quadratic variations and therefore has no larger variance.

\begin{proposition}\label{prop:conditional_signature_decomposition}
Let $P\in\mathcal Q_G$ and $Q\in\mathcal Q$, let $X^1,\dots,X^\samples$ be independent with law $P$, and let
$
\mathcal G_\samples=\sigma(\langle X^1\rangle,\dots,\langle X^\samples\rangle).
$
Then
\begin{align}\label{eq:scalar_rao_blackwell}
    \E\left[
        \dMMDhat(P^\samples,Q)
        \,\middle|\,
        \mathcal G_\samples
    \right]
    =
    \dcharhat(P^\samples,Q).
\end{align}
In particular,
\begin{align}\label{eq:unbiased_mmd_estimators}
    \E\left[\dMMDhat(P^\samples,Q)\right]
    =
    \E\left[\dcharhat(P^\samples,Q)\right]
    =
    \dMMD(P,Q)^2,
\end{align}
and
\begin{align}\label{eq:mmd_variance_decomposition}
    \operatorname{Var}\left(
        \dMMDhat(P^\samples,Q)
    \right)
    =
    \operatorname{Var}\left(
        \dcharhat(P^\samples,Q)
    \right)
    +
    \E\left[
        \operatorname{Var}\left(
            \dMMDhat(P^\samples,Q)
            \,\middle|\,
            \mathcal G_\samples
        \right)
    \right].
\end{align}
\end{proposition}

\begin{proof}
For $j\ne k$, conditional independence together with \eqref{eq:conditional_expected_signature} gives
\begin{align*}
    &\E\left[
        \left\langle
            \Sig{X^j}_{0,T}-\esig_Q,
            \Sig{X^k}_{0,T}-\esig_Q
        \right\rangle
        \,\middle|\,
        \mathcal G_\samples
    \right] =
    \left\langle
        \mathcal S\left(\tfrac12\langle X^j\rangle\right)_{0,T}-\esig_Q,
        \mathcal S\left(\tfrac12\langle X^k\rangle\right)_{0,T}-\esig_Q
    \right\rangle.
\end{align*}
Summing over $j\ne k$ proves \eqref{eq:scalar_rao_blackwell}. Taking expectations and using independence gives
\eqref{eq:unbiased_mmd_estimators}, while
\eqref{eq:mmd_variance_decomposition} follows from the law of total
variance together with \eqref{eq:scalar_rao_blackwell}.
\end{proof}

\begin{remark}\label{rmk:rao_blackwell}
Consider the family of laws in $\mathcal Q_G$ parameterized by the distribution $\rho$ of the quadratic variation. Conditionally on the quadratic-variation paths, the sample law is the product of the corresponding Gaussian martingale laws and does not depend on $\rho$. Hence, the bracket sample is a sufficient statistic for $\rho$, and \eqref{eq:scalar_rao_blackwell} is the classical Rao--Blackwell construction \cite{blackwell1947conditional,rao1945information}.
\end{remark}

\begin{remark}\label{rmk:conditional_pii}
The conditioning argument also applies to processes with jumps. Following the conditional-PII framework of \cite[Section~2]{jacod2002processes}, let $\mathcal G$ be a conditioning $\sigma$-algebra such that
\begin{equation*}
    \mathrm{Law}(X\mid\mathcal G)(\omega)
    =
    \mathscr L^N\bigl(b(\omega),a(\omega),K(\omega)\bigr)
\end{equation*}
almost surely, where $(b,a,K)$ is a $\mathcal G$-measurable random time-varying L\'evy triplet. The characteristic velocity $\mathfrak y$ defined pathwise by \eqref{def:frak_eta} is then $\mathcal G$-measurable. Assume that \eqref{eq:exponential_moments} holds almost surely at $t=T$ with $\lambda=1$ and that $\E[\Vert\Sig{X}_{0,T}\Vert_2^2]<\infty$. Applying \cref{prop:levy_esig} conditionally gives
\begin{equation*}
    \E\left[\Sig{X}_{0,T}\,\middle|\,\mathcal G\right]
    =
    \mathcal S\left(\int_0^\cdot\mathfrak y(t)\dd t\right)_{0,T}.
\end{equation*}
For independent copies of this construction, replacing each sampled signature by its conditional expected signature yields, by the same argument as in \cref{prop:conditional_signature_decomposition}, an unbiased estimator of the same squared signature-MMD with no larger variance. The required inner products are computed by \cref{thm:main}, with finite-dimensional approximations provided by \cref{thm:truncated}.
Estimation of the finitely many required characteristic-velocity coefficients can draw on methods for volatility \cite{ait2014high}, L\'evy characteristics and jump moments \cite{comte2011estimation}, and time-changed L\'evy models \cite{bull2014timechanges,figueroa2009timechanged}, under the respective model and sampling assumptions.
\end{remark}

\begin{remark}\label{rmk:conditional_moment_residual}
For $X\sim P\in\mathcal Q_G$, \cref{lem:conditional_gaussian_orthogonality} gives the conditional moment restriction
\begin{equation*}
\E\left[\ResQV(X)\,\middle|\,\langle X\rangle\right]=0, \qquad \ResQV(X)=\Sig{X}_{0,T}-\mathcal S\left(\tfrac12\langle X\rangle\right)_{0,T}.
\end{equation*}
This suggests specification tests for conditional Gaussianity based on the kernel conditional-moment framework of \cite{muandet2020kernel}, here applied to Hilbert-valued moments, by embedding the restriction through
\begin{align*}
    \left\Vert
        \E\left[
            \ResQV(X)\otimes_2
            k\big(\cdot,\langle X\rangle\big)
        \right]
    \right\Vert_{\TT^2\widehat\otimes_2\mathcal H_k}
    =0,
\end{align*}
where $\otimes_2$ denotes an elementary tensor in the Hilbert tensor product. In the present setting,
\begin{equation*}
k\big(\cdot,\langle X\rangle\big)
=
\mathcal S\left(\tfrac12\langle X\rangle\right)_{0,T}
\end{equation*}
provides a natural conditioning feature \cite{song2009hilbert} under the canonical feature-map identification. Indeed, the inner products required for the corresponding conditional-moment criterion reduce to the path--path, path--quadratic-variation, and quadratic-variation--quadratic-variation terms computed by the PDE systems from \cref{sec:second_level_char}. The resulting test is naturally formulated for the composite null $P\in\mathcal Q_G$. Moreover, conditionally on the sampled quadratic-variation paths, the null law does not depend on their unknown distribution, allowing for conditional calibration. A systematic study of the choice of the conditioning feature, the calibration of the resulting tests, and their finite-sample properties is left for future work.
\end{remark}

\subsection{Finite-grid statistics and computation}

Consider a coarser and a finer time grid
\begin{equation*}
\pi=\{0=\tau_0<\tau_1<\dots<\tau_L=T\}\;\;\subset\;\;\widetilde\pi=\{0=t_0<t_1<\dots<t_N=T\}.
\end{equation*}
For a path $x\in C([0,T],\mathbb R^d)$ observed on the fine grid $\widetilde\pi$, set $\delta_j^{\widetilde\pi}x:=x(t_j)-x(t_{j-1})$ and define
\begin{align*}
    X^\pi_{\tau_\ell}&:=\sum_{t_j\le\tau_\ell}\delta_j^{\widetilde\pi}x,&
    \mathbb A^\pi_{\tau_\ell}&:=\frac12\sum_{r=1}^\ell\sum_{\tau_{r-1}<t_i<t_j\le\tau_r}[\delta_i^{\widetilde\pi}x,\delta_j^{\widetilde\pi}x]_{\otimes},&
    \langle X\rangle^\pi_{\tau_\ell}&:=\sum_{t_j\le\tau_\ell}\delta_j^{\widetilde\pi}x\otimes\delta_j^{\widetilde\pi}x,
\end{align*}
and interpolate each component linearly on $\pi$. We write $\log \mathbf X^\pi:=(X^\pi,\mathbb A^\pi)$.
The construction may be viewed as a compression of the fine observations onto the coarser grid $\pi$ that retains their second-order information through the L\'evy area and realized quadratic variation. In the context of signature kernels, \cite{lemercier2024high} shows that L\'evy-area augmentation can provide more tractable approximations for highly oscillatory paths. The following results also apply without L\'evy-area augmentation, in particular by taking $\pi=\widetilde\pi$.

Let $a:[0,T]\to\mathfrak{S}_+^1(V)$ be integrable and set
\begin{equation*}
Q_a:=\mathscr L^1(0,a,0), \qquad \text{with} \qquad \esig_{Q_a}=\mathcal S\left(\frac12\int_0^\cdot a(t)\dd t\right)_{0,T}.
\end{equation*}
For simplicity, we restrict below to reference laws with deterministic quadratic variation. The extension to a prescribed random quadratic variation follows by conditioning and linearity. In particular, for a finite mixture $Q_\rho=\sum_{\ell=1}^L p_\ell Q_{a_\ell}$, one has
\begin{equation*}
\esig_{Q_\rho}=\sum_{\ell=1}^L p_\ell\mathcal S\left(\frac12\int_0^\cdot a_\ell(t)\dd t\right)_{0,T},
\end{equation*}
and all reference--reference and path--reference terms are obtained as the corresponding weighted sums of the PDE solutions below.

For the observations $x_1,\dots,x_\samples$, let $\mathbf X_k^\pi$ and $\langle X_k\rangle^\pi$ denote the paths constructed above and set
$
P^\samples=\frac1\samples\sum_{k=1}^\samples\delta_{(\mathbf X_k^\pi,\langle X_k\rangle^\pi)}.
$
For $\samples\ge2$, define the corresponding finite-grid statistics by
\begin{align}
    \dMMDhatpi(P^\samples,Q_a)
    &=
    \frac1{\samples(\samples-1)}
    \sum_{\substack{j,k=1\\j\ne k}}^\samples
    \left\langle
        \mathcal S(\log\mathbf X_j^\pi)_{0,T}-\esig_{Q_a},
        \mathcal S(\log\mathbf X_k^\pi)_{0,T}-\esig_{Q_a}
    \right\rangle,
    \label{eq:finite_grid_mmd}
    \\
    \dcharhatpi(P^\samples,Q_a)
    &=
    \frac1{\samples(\samples-1)}
    \sum_{\substack{j,k=1\\j\ne k}}^\samples
    \left\langle
        \mathcal S\left(\tfrac12\langle X_j\rangle^\pi\right)_{0,T}-\esig_{Q_a},
        \mathcal S\left(\tfrac12\langle X_k\rangle^\pi\right)_{0,T}-\esig_{Q_a}
    \right\rangle.
    \notag
\end{align}

\begin{theorem}\label{thm:dMMD_formula}
Let $(\log \mathbf X_k^\pi, \langle X_k\rangle^\pi )$ be constructed from sample paths $x_1, \dots, x_M$ as above. Write $\log \mathbf X_k^\pi=\int_0^\cdot\big(b_k(t)+\mathfrak a_k(t)\big)\dd t$, $
    \langle X_k\rangle^\pi=\int_0^\cdot a_k(t)\dd t$
and set $a_0:=a$. For $j,k=0,\dots,\samples$, let $u_{j,k}$ be the unique solution to the scalar Goursat problem \eqref{eq:goursat} with 
    $\alpha_{j,k}(s,t):=\frac14\langle a_j(s),a_k(t)\rangle.$ 
For $j,k=1,\dots,\samples$, let $v_{j,0}$ and $w_{j,k}$ be the scalar components of the unique solutions to the PDE system \eqref{eq:continuous_case} for the following choices of characteristics; see \cref{cor:level2}:
\begin{align}\label{eq:kernel_characteristics}
\begin{array}{c|c|c}
    & (b,\mathfrak a,a) & (\widetilde b,\widetilde{\mathfrak a},\widetilde a)\\
    \hline
    v_{j,0} & (b_j,\mathfrak a_j,0) & (0,0,a)\\
    w_{j,k} & (b_j,\mathfrak a_j,0) & (b_k,\mathfrak a_k,0).
\end{array}
\end{align}
Then
\begin{align}
    \dMMDhatpi(P^\samples,Q_a)
    &=
    u_{0,0}(T,T)
    -\frac2\samples\sum_{k=1}^\samples v_{k,0}(T,T)
    +\frac1{\samples(\samples-1)}
    \sum_{\substack{j,k=1\\j\ne k}}^\samples w_{j,k}(T,T),
    \label{eq:mmd_formula}
    \\
    \dcharhatpi(P^\samples,Q_a)
    &=
    u_{0,0}(T,T)
    -\frac2\samples\sum_{k=1}^\samples u_{k,0}(T,T)
    +\frac1{\samples(\samples-1)}
    \sum_{\substack{j,k=1\\j\ne k}}^\samples u_{j,k}(T,T).
    \notag
\end{align}
\end{theorem}

\begin{proof}
Expanding the inner products in the definitions of
$\dMMDhatpi$ and $\dcharhatpi$ shows that it remains to identify the Gaussian--Gaussian, path--Gaussian, and path--path kernel terms. To this end, set
\begin{align*}
    \overline v_{j,0}(s,t)
    =
    \bigl\langle
        \mathcal S(\mathbf X_j^\pi)_{0,s},
        \E_{Y\sim Q_a}\left[\Sig{Y}_{0,t}\right]
    \bigr\rangle\Big.,\qquad
    \overline w_{j,k}(s,t)
    =
    \bigl\langle
        \mathcal S(\mathbf X_j^\pi)_{0,s},
        \mathcal S(\mathbf X_k^\pi)_{0,t}
    \bigr\rangle\Big..
\end{align*}
\Cref{exmpl:time-dependent-bm} identifies the Gaussian--Gaussian terms with $u_{j,k}$. Moreover, \cref{cor:level2} shows that $\overline v_{j,0}$ and $\overline w_{j,k}$ are the $u$-components of solutions to \eqref{eq:continuous_case} for the choices of characteristics listed in \eqref{eq:kernel_characteristics}. By uniqueness, they coincide with $v_{j,0}$ and $w_{j,k}$, respectively. Substituting these identities into the definitions of $\dMMDhatpi$ and $\dcharhatpi$ gives the asserted formulas.
\end{proof}

The preceding formulas allow us to evaluate both statistics through the PDE systems of \cref{sec:second_level_char}, without sampling from the reference law. For the specified null $P=Q_a$, critical values may therefore be obtained by standard Monte Carlo calibration under $Q_a$, with the simulated paths processed on the same observation and compression grids as the data. The formulas eliminate the auxiliary reference sample that would otherwise be required in each calibration replication.

\subsection{A goodness-of-fit example}\label{sec:numerics_gof}

We illustrate the preceding results in a goodness-of-fit experiment comparing $\dMMDhatpi$ and $\dcharhatpi$. Since their continuous-observation counterparts estimate the same population discrepancy within the conditionally Gaussian family considered below, we also compare their empirical variances to assess the variance reduction in \cref{prop:conditional_signature_decomposition}.
On $[0,T]=[0,1]$, set
\begin{equation*}
a_\pm(t)
=
I_2
+
\frac12
\left[
    \cos(2\pi t)H_1
    \pm
    \sin(2\pi t)H_2
\right],
\qquad
H_1
=
\begin{pmatrix}
    1 & 0\\
    0 & -1
\end{pmatrix},
\qquad
H_2
=
\begin{pmatrix}
    0 & 1\\
    1 & 0
\end{pmatrix},
\end{equation*}
and write $Q_\pm:=\mathscr L^1(0,a_\pm,0)$. The paths $a_+$ and $a_-$ have the same image in $\mathfrak{S}_+^1(\R^2)$, traversed in opposite directions. Moreover,
$
\int_0^1a_\pm(t)\dd t=I_2,
$
so $Q_+$ and $Q_-$ have the same terminal marginal law, while their expected signatures first differ at level four.
For $\delta\in[0,1]$, define
\begin{equation*}
P_\delta^{\mathrm{or}}
:=
(1-\delta)Q_+
+
\delta Q_-.
\end{equation*}
The null is $P_0^{\mathrm{or}}=Q_+$, while $P_1^{\mathrm{or}}=Q_-$ reverses the covariance-loop orientation completely. Since $P_\delta^{\mathrm{or}}\in\mathcal Q_G$, \cref{prop:conditional_signature_decomposition} shows that the ideal signature and quadratic-variation statistics are unbiased estimators of the same squared signature-MMD. By linearity,
\begin{equation*}
\dMMD(P_\delta^{\mathrm{or}},Q_+)^2
=
\delta^2\dMMD(Q_-,Q_+)^2,
\end{equation*}
where \cref{cor:level2} gives $
\dMMD(Q_-,Q_+)^2
\approx
0.01534$. 

\smallskip
\noindent\emph{Observation scheme and calibration.}
All paths are observed on the common regular grid
\begin{equation*}
\pi_{48}:=\left\{ j/{48}:j=0,\dots,48\right\}.
\end{equation*}
We take $\pi=\widetilde\pi=\pi_{48}$ in the construction above and do not use L\'evy-area augmentation. For each $\samples\in\{16,32,64\}$, both statistics are calibrated at level $0.05$ using $B=10\,000$ independent Monte Carlo datasets under $Q_+$, processed on the same grid as the data. The PDE formulas evaluate all reference-dependent kernel terms directly, so no auxiliary reference sample is required. Empirical size is estimated from $2\,000$ independent datasets, while empirical power is estimated from $1\,000$ independent datasets for each value of $\delta$ and $\samples$.

\smallskip
\noindent\emph{Variance reduction.}
Under $Q_+$, the exact quadratic variation is deterministic, and hence the ideal statistic $\dcharhat(P^\samples,Q_+)$ vanishes. The remaining variability of $\dcharhatpi$ is therefore introduced by the finite-grid estimation of the quadratic variation. \Cref{tab:gof_size_variance} reports the empirical size of both statistics, their null standard deviations, and the ratio of their empirical variances, which quantifies the variance reduction in \eqref{eq:mmd_variance_decomposition}.

\begin{table}[h]
    \centering
    \begin{tabular}{c|cc|cc|c}
        & \multicolumn{2}{c|}{Empirical size}
        & \multicolumn{2}{c|}{Null standard deviation}
        & Variance ratio\\
        $\samples$
        & $\dMMDhatpi$
        & $\dcharhatpi$
        & $\dMMDhatpi$
        & $\dcharhatpi$
        & $\tfrac{\widehat{\operatorname{Var}}(\dcharhatpi)}{\widehat{\operatorname{Var}}(\dMMDhatpi)}$\\
        \hline
        $16$ & $0.048\;(0.005)$ & $0.050\;(0.005)$
        & $2.49\cdot10^{-1}$ & $2.68\cdot10^{-3}$ & $1.16\cdot10^{-4}$\\
        $32$ & $0.051\;(0.005)$ & $0.047\;(0.005)$
        & $1.25\cdot10^{-1}$ & $1.36\cdot10^{-3}$ & $1.20\cdot10^{-4}$\\
        $64$ & $0.041\;(0.004)$ & $0.048\;(0.005)$
        & $6.10\cdot10^{-2}$ & $6.87\cdot10^{-4}$ & $1.27\cdot10^{-4}$\\
    \end{tabular}
    \caption{Empirical rejection probabilities under $Q_+$ at nominal level $0.05$, with Monte Carlo standard errors in parentheses, empirical null standard deviations of the two statistics, and the ratio of their empirical variances.}
    \label{tab:gof_size_variance}
\end{table}

At the primary grid resolution, the empirical variance ratio is approximately $1.2\cdot10^{-4}$, so the variance reduction from \cref{prop:conditional_signature_decomposition} is retained almost entirely by the finite-grid implementation. As a grid-refinement check, the experiment is repeated for $\samples=32$ on the grid $\pi_{96}$, using $2\,000$ calibration and $500$ evaluation datasets. Increasing the grid resolution from $48$ to $96$ leaves the null standard deviation of $\dMMDhatpi$ at $1.25\cdot10^{-1}$ and lowers that of $\dcharhatpi$ from $1.36\cdot10^{-3}$ to $6.49\cdot10^{-4}$, consistent with the remaining variability of $\dcharhatpi$ being caused by the finite-grid estimation of the quadratic variation.

\smallskip
\noindent\emph{Power.}
For
$
\delta\in\{0,0.25,0.5,1\}
$, 
we estimate rejection probabilities using the critical values obtained under $Q_+$. The results are shown in \cref{fig:gof_power} and \cref{tab:gof_power}.

\begin{figure}[h]
    \centering
    \includegraphics[width=\textwidth]{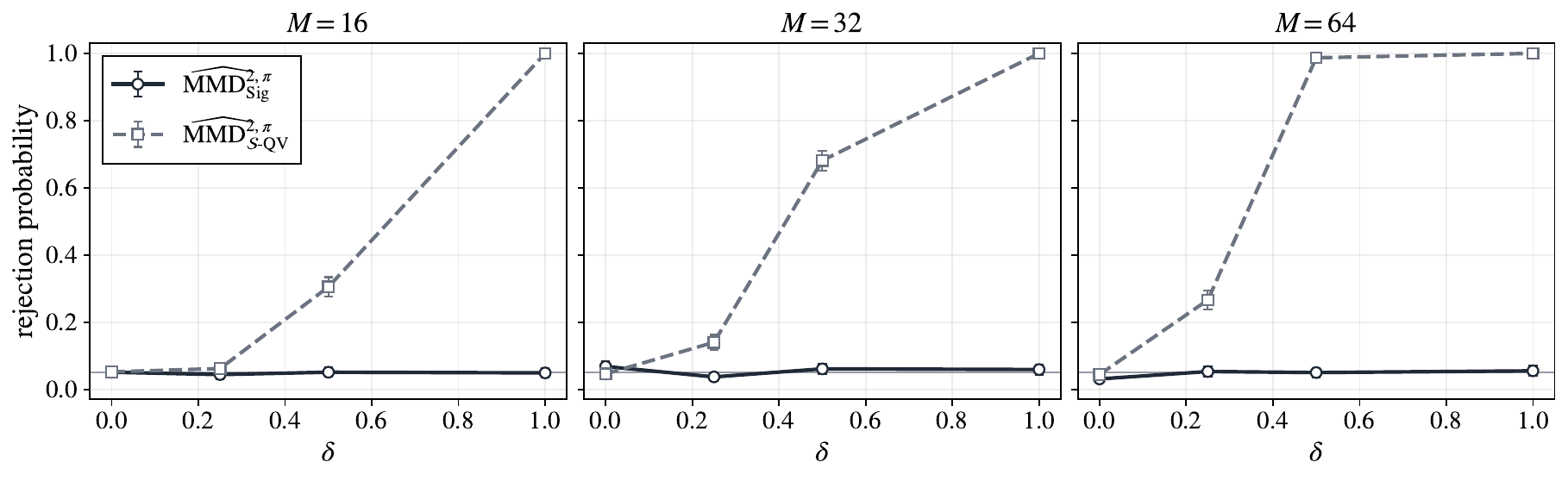}
    \caption{Empirical rejection probabilities at nominal level $0.05$ for $\dMMDhatpi$ and $\dcharhatpi$ under the orientation alternatives $P_\delta^{\mathrm{or}}$ against $Q_+$. The panels correspond to $\samples=16$, $32$, and $64$, and error bars show two binomial Monte Carlo standard errors.}
    \label{fig:gof_power}
\end{figure}

\begin{table}[h]
    \centering
    \begin{tabular}{c|ccc|ccc}
        & \multicolumn{3}{c|}{$\dMMDhatpi$}
        & \multicolumn{3}{c}{$\dcharhatpi$}\\
        $\delta$
        & $\samples=16$ & $32$ & $64$
        & $\samples=16$ & $32$ & $64$\\
        \hline
        $0$    & $0.051$ & $0.068$ & $0.031$ & $0.052$ & $0.046$ & $0.044$\\
        $0.25$ & $0.044$ & $0.037$ & $0.053$ & $0.062$ & $0.140$ & $0.266$\\
        $0.5$  & $0.051$ & $0.061$ & $0.050$ & $0.305$ & $0.681$ & $0.987$\\
        $1$    & $0.049$ & $0.059$ & $0.055$ & $1.000$ & $1.000$ & $1.000$\\
    \end{tabular}
    \caption{Empirical rejection probabilities at nominal level $0.05$ under the orientation alternatives $P_\delta^{\mathrm{or}}$ against $Q_+$. Each entry is based on $1\,000$ evaluation datasets, so the Monte Carlo standard errors are at most $0.016$. The row $\delta=0$ is an independent estimate of the size reported in \cref{tab:gof_size_variance}.}
    \label{tab:gof_power}
\end{table}

Although both statistics estimate the same population discrepancy, the power of $\dcharhatpi$ increases in both $\delta$ and $\samples$, reaching one under the complete orientation reversal $\delta=1$, whereas $\dMMDhatpi$ remains at the nominal level throughout. This difference is consistent with the variance reduction predicted by \cref{prop:conditional_signature_decomposition}: at the continuous-observation level, conditioning changes neither the population target nor the alternative signal, while removing the conditional signature fluctuations.

\newcommand{\etalchar}[1]{$^{#1}$}

\appendix

\section*{Contents of the Appendices}
\startcontents[appendices]
\printcontents[appendices]{app}{1}[2]{}

\section{Extended tensor-algebraic and stochastic framework}\label{sec:extended_preliminaries}

In \cref{sec:preliminaries} we introduced, in condensed form, the tensor-algebraic and stochastic objects used throughout the main article. We retain throughout the tensor-algebra notation fixed there and do not repeat its basic construction. We now give the corresponding results in greater generality, together with the proofs and additional background. In particular, we discuss the general $\TT^p$ estimates and adjoint multiplications, the componentwise construction and further estimates for free developments, stochastic developments in free Lie groups, and the description of inhomogeneous L\'evy processes through their differential characteristics.

\subsection{General $\TT^p$ estimates and adjoint multiplications}\label{ext:sec:Tpspaces}
Retaining the notation of \cref{sec:Tpspaces}, let $\TT^p$ be the space defined in \eqref{eq:def_Tp}. For $1\le p\le q$ it holds $\Vert\bx\Vert_p\ge\Vert\bx\Vert_q$ and thus $\TT^p\subseteq\TT^q$. Furthermore, $\TT^p$ is a Banach space with Schauder basis $\{e_w\}_{w\in\mathcal W}$, ordered lexicographically. For $p,q\ge1$ with $\frac1p+\frac1q=1$, H\"older's inequality gives
\begin{equation*}
\sum_{k=0}^\infty|\langle\bx^{(k)},\by^{(k)}\rangle|
\le\Vert\bx\Vert_p\Vert\by\Vert_q,
\end{equation*}
so that the levelwise inner product from \cref{sec:Tpspaces} defines the usual dual pairing between $\TT^p$ and $\TT^q$. In particular, $(\TT^2,\langle\cdot,\cdot\rangle)$ is a Hilbert space.

For $\bx,\by\in\TT$, compatibility of the tensor norms gives
\begin{align}\label{ext:eq:lp_convolution}
\Vert\bx\otimes\by\Vert_p^p
&=\sum_{n=0}^\infty\left|\sum_{k=0}^n\bx^{(k)}\otimes\by^{(n-k)}\right|^p
\le\sum_{n=0}^\infty\left(\sum_{k=0}^n|\bx^{(k)}||\by^{(n-k)}|\right)^p.
\end{align}
Thus, properties of convolution products of $\ell^p$-sequences extend directly to tensor series.

In particular \emph{Young's convolution inequality} (see e.g. \cite[Theorem~4.5.1-2]{hoernmander2003analysis} or the $\ell^p$-formulation in \cite[Theorem 12.1.2]{cheng2020function}) directly extends to the following
\begin{proposition}\label{ext:prop:young}
Let $p,q,r\ge1$ satisfy $\frac1p+\frac1q=1+\frac1r$. Then
\begin{equation*}
\Vert\bx\otimes\by\Vert_r\le\Vert\bx\Vert_p\Vert\by\Vert_q,
\qquad \bx\in\TT^p,\ \by\in\TT^q.
\end{equation*}
\end{proposition}
\begin{proof}
Set $\overline\bx=(|\bx^{(n)}|)_{n\ge0}$ and $\overline\by=(|\by^{(n)}|)_{n\ge0}$. By \eqref{ext:eq:lp_convolution} and Young's convolution inequality for $\ell^p$-sequences,
\begin{equation*}
\Vert\bx\otimes\by\Vert_r
\le\Vert\overline\bx\ast\overline\by\Vert_{\ell^r}
\le\Vert\overline\bx\Vert_{\ell^p}\Vert\overline\by\Vert_{\ell^q}
=\Vert\bx\Vert_p\Vert\by\Vert_q.
\end{equation*}
\end{proof}

When $\frac{1}{p}+\frac{1}{q} < 1 + \frac{1}{r}$ we can find tensor series $\bx \in \TT^p$ and $\by\in\TT^q$ such that $\Vert \bx \otimes \by \Vert_r = \infty$.
Indeed, we construct an example for $p=q=r=2$, which can easily be generalized: Consider the series $\bx = \sum_{n\ge1}  n^{-\beta} e_1^{\otimes n}$ with  $\beta \in (\frac{1}{2}, \frac{3}{4}]$, for which we clearly have $\bx\in\TT^2$. Then $|\pi_n( \bx^{\otimes 2})| = \sum_{k=1}^n k^{-\beta}(n-k)^{-\beta} \ge n (\frac{n}{2})^{-2\beta} \ge (\frac{n}{2})^{-2\beta + 1}$. Thus $\Vert \bx^{\otimes 2} \Vert_{2} = \infty$.
This demonstrates that $\TT^2$ is not a subalgebra of $\TT$.
For $p=1$ we can instead directly deduce from Young's inequality the following
\begin{corollary}
    $(\TT^1, +, \otimes)$ is a Banach algebra.
\end{corollary}
Following \cite{lemercier2024high}, we will also need the operations that are adjoint to tensor multiplication. Young's inequality allows us to identify their appropriate domains and codomains.
\begin{proposition}\label{ext:prop:adjoint_mul}
    Let $p,q,r\ge1$ such that $\frac1p + \frac{1}{q} = 1 + \frac{1}{r}$.
    There exist continuous bilinear maps 
    \begin{align*}
        \dualLeft{}{}: \TT^{p}\times \TT^{q} \to \TT^{r},\quad (\bx, \bz)\mapsto \dualLeft{\bx}{\bz} \\
        \dualRight{}{}:  \TT^{p}\times \TT^{q} \to \TT^{r},\quad (\bx, \bz)\mapsto \dualRight{\bx}{\bz}
    \end{align*}
    such that 
    \begin{align}\label{ext:eq:dual_multiplication}
    \langle \bz, \bx\otimes \by \rangle = \langle \dualLeft{\bx}{\bz}, \by \rangle = \langle \dualRight{\by}{\bz}, \bx \rangle
    \end{align}
    for all $\bx\in\TT^{p}, \bz\in\TT^{q}$ and $\by\in\TT^{r^\prime}$ with $\frac{1}{r^\prime} = 1 - \frac1r$.
\end{proposition}
\begin{proof}[Proof of \Cref{ext:prop:adjoint_mul}]
    Let $\bx \in \TT^{p}$ and $\bz\in \TT^{q}$.
    For any $w\in\mathcal{W}$ we clearly have $\bx\otimes e_w \in \TT^{p}$ and since $\frac1p +\frac1q > 1$ it holds
    \begin{align*}
    \langle \bz, \bx\otimes e_w \rangle = \sum_{v\in\mathcal{W}} \bz^{vw}\bx^{v} < \infty.
    \end{align*}
    We then define $\dualLeft{\bx}{\bz} := \sum_{w\in\mathcal{W}}e_w\langle \bz, \bx\otimes e_w \rangle \in \TT$.
    Clearly, $\dualLeft{}{}$ is bilinear and it holds
    \begin{equation}\label{ext:eq:dual_proof}
        \langle \bz, \bx\otimes\by\rangle = \langle \dualLeft{\bx}{\bz}, \by\rangle, \qquad \by\in T(V).
    \end{equation}
    Let $\frac{1}{q^{\prime}} = 1 - \frac{1}{q}$, $\frac{1}{r^{\prime}} = 1 - \frac{1}{r}$ and note that
    $1+\frac{1}{q^{\prime}} = 
    \frac{1}{p} + \frac{1}{r^{\prime}}.$
    Then by H\"older's and Young's inequality (\cref{ext:prop:young}) we have
    \begin{equation*} \vert \langle \dualLeft{\bx}{\bz}, \by\rangle \vert =  \vert \langle \bz, \bx\otimes\by\rangle  \vert
    \le \Vert \bz \Vert_q \Vert \bx \otimes \by \Vert_{q^{\prime}} \le \Vert \bz \Vert_q\Vert \bx \Vert_p \Vert \by \Vert_{r^{\prime}}.\end{equation*}
    By continuity the identity \eqref{ext:eq:dual_proof} extends to all $\by\in \TT^{r^{\prime}}$ and in particular $\dualLeft{\bx}{\bz} \in (\TT^{r^{\prime}})^{\ast} = \TT^r$.
    An entirely analogous discussion applies to the dual of the right-tensor multiplication map.
\end{proof}

\begin{remark}For $p=q=1$ the above proposition specializes to show that the adjoint left-~and right-tensor multiplication actually define non-associative products on $\TT^1$.
\end{remark}

The following algebraic properties of the adjoint multiplications are adapted from \cite{lemercier2024high} to the $\TT^1$ setting. 
\begin{proposition}\label{ext:prop:left-right-technical}
    Let $a \in V^{\otimes n}$ and $b \in V^{\otimes k}$. Then  \begin{equation*}\dualLeft{b}{a},\; \dualRight{b}{a} \in V^{\otimes n-k}\quad\text{for }n \ge k, \qquad \text{ and }\qquad \dualLeft{b}{a},\, \dualRight{b}{a} =0,\quad\text{for } n < k.\end{equation*}
    Furthermore, for $n\ge k$ it holds
\begin{align*}
\big\langle \bx\otimes a, \by\otimes b\big\rangle = \big\langle \dualLeft{\bx}{\by}, \dualRight{b}{a} \big\rangle, \qquad \bx, \by \in \TT^1.
\end{align*}
\end{proposition}
\begin{proof}[Proof of \Cref{ext:prop:left-right-technical}]
To see the first claim we simply project to words $w \in \mathcal{W}$ to obtain
\begin{equation*}\langle \dualLeft{b}{a}, e_w \rangle = \langle a, b \otimes e_w \rangle = \sum_{v\in \mathcal{W}}b^va^{vw} = \sum_{\substack{v\in \mathcal{W}\\|v|=k}}b^va^{vw},\end{equation*}
where in the last equality we have used that $b \in V^{\otimes k}$. Further, using that $a \in V^{\otimes n}$ we see that the right-hand side is zero whenever $|vw| \neq n$, i.e., whenever $|w| \neq n - k$, from which we readily see the first claim.
For the adjoint right-multiplication the argument is entirely analogous.
For the second claim, we note that by H\"older's inequality it holds $\sum_{k=0}^\infty |\langle (\bx\otimes a)^{(k)}, (\by\otimes b)^{(k)} \rangle|<\infty$ and so the following rearrangements of absolutely convergent sums are justified:
\begin{align*}
\big\langle \bx\otimes a, \by\otimes b\big\rangle 
&= \sum_{\substack{w\in\mathcal{W}\\|w|\ge n}}(\bx\otimes a)^w (\by\otimes b)^w \\
&= \sum_{w_1\in\mathcal{W}}\sum_{\substack{w_2\in\mathcal{W}\\|w_2|=n-k}}\sum_{\substack{w_3\in\mathcal{W}\\|w_3|=k}}\bx^{w_1}a^{w_2w_3} \by^{w_1w_2} b^{w_3} \\
&= \sum_{w_1\in\mathcal{W}}\sum_{\substack{w_2\in\mathcal{W}\\|w_2|=n-k}}\bx^{w_1} \by^{w_1w_2}\sum_{\substack{w_3\in\mathcal{W}\\|w_3|=k}}a^{w_2w_3} b^{w_3} \\
&= \sum_{\substack{w_2\in\mathcal{W}\\|w_2|=n-k}}\left(\sum_{w_1\in\mathcal{W}}\bx^{w_1} \by^{w_1w_2}\right)(\dualRight{b}{a})^{w_2}\\
&= \sum_{\substack{w_2\in\mathcal{W}\\|w_2|=n-k}}\dualLeft{\bx}{\by}^{w_2} (\dualRight{b}{a})^{w_2}\\
&= \big\langle \pi_{n-k}\dualLeft{\bx}{\by}, \pi_{n-k}\dualRight{b}{a} \big\rangle.
\end{align*}
The result now follows from noting that $\dualRight{b}{a} \in V^{\otimes n-k}$.
\end{proof}

\subsection{Dilation and convergence radius}\label{ext:sec:radius}
With $\delta_\lambda$ and $\mathcal R^\lambda$ as in \cref{sec:radius}, we record for completeness that $\mathcal R^\lambda$ is a subalgebra of $\TT$ for every $\lambda\in(0,\infty]$.
\begin{proof}
For $\bx,\by\in\mathcal R^\lambda$ and $\alpha\in(0,\lambda)$, \eqref{eq:dilation_multiplication} and \cref{ext:prop:young} give
\begin{equation*}
\Vert\delta_\alpha(\bx\otimes\by)\Vert_1
=\Vert\delta_\alpha\bx\otimes\delta_\alpha\by\Vert_1
\le\Vert\delta_\alpha\bx\Vert_1\Vert\delta_\alpha\by\Vert_1<\infty.
\end{equation*}
Hence $\bx\otimes\by\in\mathcal R^\lambda$.
\end{proof}

\subsection{Free Lie algebras and groups}
The free Lie algebra over $\{e_1, \dots, e_d\}$ is embedded in $T(V)$ with commutator bracket $[\bx, \by]_\otimes := \bx\otimes \by - \by\otimes \bx$ and is denoted by $\mathfrak{g}(V)$.
Following \cite{reutenauer2003free}, a graded basis $\{\mathfrak{u}_w\}_{w\in \mathcal{I}}$ of $\mathfrak{g}(V)$ is suitably indexed by a subset of words $\mathcal{I} \subset \mathcal{W}$.
The grading of $\mathcal{I}$ is by word length and we define $\mathcal{I}_N := \{ w \in \mathcal{I} \;\vert\; |w|\le N\}$.
Similar to the extended tensor algebra we define the space of formal Lie series by $\mathrm{Lie}((V)) = \{ \sum_{w\in\mathcal{I}} \mathfrak{u}_w \bx^w\;\vert\; \bx^w \in \R\}$, which forms a sub-Lie algebra of $\TT_0$.
The image under the truncation map
\begin{equation*}\mathfrak{g}^N(V) := \pi_{(0,N)}\mathrm{Lie}((V)) =  \mathrm{span}\{ \mathfrak{u}_w \;\vert\; w \in \mathcal{I}_N\}\end{equation*}
yields the free step-$N$ nilpotent Lie algebra with the truncated commutator bracket $[\cdot, \cdot]_{\otimes_N}$.
The exponential image $\mathcal{G} = \exp_\otimes(\mathrm{Lie}((V)))$ yields a subgroup of $\TT_1$, which, borrowing from Hopf algebra terminology, is called the group of \emph{group-like elements}.
The exponential image $G^N(V) = \pi_{(0,N)}\exp_{\otimes}(\mathfrak{g}^N(V))$ forms a group under the truncated tensor multiplication and equipped with the subspace topology of $(T^N(V), \Norm{\cdot})$ it forms the \textit{free step-$N$ nilpotent Lie group}.

\subsection{Free developments}\label{ext:sec:free_developments}

The classical Cartan development of a smooth path $\gamma: [0,\infty)\to \mathfrak{g}^N(V)$ into the Lie group $G^N(V)$ is obtained by solving the differential equation
\begin{align}\label{ext:eq:cartan}
    \dot{X}(t) = X(t) \otimes_N \dot\gamma(t), \qquad X_0 = \mathbf{1}\in\ G^N(V).
\end{align}
While such developments can be understood from a general Lie perspective, here, as explained below, we understand and solve the above equation simply by projecting to tensor levels/components in hierarchical order.
The solution is a smooth path in the step-$N$ group $X: [0,\infty) \to G^N(V)$.
Conversely, any \emph{smooth rough path} \cite{bellingeri2022smooth}, i.e., any smooth path $X: [0,\infty) \to G^N(V)$ is the development of a smooth path $\gamma: [0,\infty)\to \mathfrak{g}^N(V)$.
The \emph{minimal extension} of the smooth rough path $X$ is obtained by fully developing $\gamma$ into $\TT_1$:
\begin{align}\label{ext:eq:free_development}
    S(t) = 1 + \int_{0}^t S(u) \otimes \dot\gamma(u)\dd{u}.
\end{align}
For further reference, we will state how equations of the above form are understood precisely.
The integral of a path $\mathfrak{y}: [0,\infty) \to \TT$ is understood componentwise, i.e.,
\begin{equation*}\int_0^t \mathfrak{y}(s)\dd{s} := \sum_{w\in\mathcal{W}}e_w\left(\int_0^{t}\mathfrak{y}(s)^{w}\dd{s}\right) \in \TT, \quad t\ge0,\end{equation*}
thus well defined whenever $\mathfrak{y}(\cdot)^{w}$ is locally integrable for all $w\in\mathcal{W}$.
We understand \eqref{ext:eq:free_development} as an equality between paths taking values in $\TT$.
In the specific case of \eqref{ext:eq:free_development}, we solve by first projecting to $w = \varnothing$, leading to $S(\cdot)^\varnothing \equiv 1$, and then inductively for $w\in\mathcal{W}\setminus\{\varnothing\}$ we get
\begin{align*} S(t)^{w} &= \sum_{w_1 w_2 = w} \int_0^t S(s)^{w_2} \dot\gamma(s)^{w_1} \dd{s}  \\
&= \sum_{k = 1}^{|w|}\sum_{w_k \cdots w_1= w}\int_0^t\int_0^{s_1}\cdots\int_0^{s_{k-1}} \dot\gamma(s_k)^{w_k}\dd{s_k} \cdots \dot\gamma(s_1)^{w_1} \dd{s_1}, \quad t \ge0, \;w\in\mathcal{W},
\end{align*}
where the inner summation is over all deconcatenations ${w_k\cdots w_1= w}$ of the word $w$ into $k$ non-empty words.
It can then be verified that $S$ takes values in $\mathcal{G}$ and is therefore called the \emph{signature} of $\gamma = \int_0^\cdot \dot{\gamma}(s) \dd{s}$ \cite[Theorem 2.8]{bellingeri2022smooth}.

Note that equation \eqref{ext:eq:free_development} makes perfect sense for any path $\dot\gamma:[0,\infty) \to \TT_0$ that is componentwise integrable.
The solution $S$ is constructed in the exact same way, but will generally take values in the larger group $\TT_1$.
With the spirit of the \textit{free noncommutative algebra} in mind, we will call the solution the ``free development''.

For consistency with the signature of $V$-valued paths, we define this development in terms of the underlying path $\gamma = \int_0^\cdot \dot{\gamma}(u) \dd{u}$.
Our motivation being the treatment of differential characteristics (see \cref{ext:sec:inhom_levy} and \cref{sec:moment_condition}), we will restrict our attention to componentwise absolutely continuous paths; the generalization to continuous $1$-variation paths is straightforward.

\begin{definition}\label{ext:def:free_development}
    Let $\gamma:[0,\infty) \to \TT_0$ be componentwise absolutely continuous.
    We define the \textit{free development} of ${\gamma}$ starting in $s\ge0$, \begin{equation*}\mathcal{S}(\gamma)_{s,\cdot}: [0,\infty) \to \TT_1, \quad t\mapsto S(\gamma)_{s,t},\end{equation*}
    as the unique solution to \begin{align}\label{ext:eq:free_development_start}
    \mathcal{S}(\gamma)_{s,t} = 1 + \int_{s}^t \mathcal{S}(\gamma)_{s,u} \otimes \dot\gamma(u)\dd{u}, \qquad t \ge s.
    \end{align}
    We will also write $\mathcal{S}(\gamma) = \mathcal{S}(\gamma)_{0, \cdot}$.
\end{definition}

We collect a few basic properties of the free development in the following two propositions. 
\begin{proposition}%
\label{ext:lem:gronwall} Let $\gamma:[0,\infty) \to \TT_0$ be a componentwise absolutely continuous path in $\TT_0$.
Then for its free development $\mathcal{S}(\gamma)$ it holds:
\begin{enumerate}[label=\arabic*.)]
    \item $\mathcal{S}(\gamma)_{s,t} = \mathcal{S}(\gamma)_{s,u}\otimes \mathcal{S}(\gamma)_{u,t},$ for all $t \ge u \ge s \ge 0$;
    \item $\delta_{\lambda}\mathcal{S}(\gamma) = \mathcal{S}(\delta_{\lambda}\gamma)$, for all $\lambda  \in \R$;
    \item $\mathcal{S}(\gamma)_{s,t} \in \mathcal{G}$ for all $t \ge s \ge0$ if and only if $\dot\gamma(t) \in \mathrm{Lie}((V))$ for a.e. $t\ge0$. 
    In this case we write \begin{equation*}\mathrm{Sig}(\gamma)_{s,t} := \mathcal{S}(\gamma)_{s,t}, \qquad  t\ge s \ge0;\end{equation*}
    \item If $\dot\gamma(\cdot) \equiv \bx \in \TT_0$ then $\mathcal{S}(\gamma)_{s,t} = \exp_\otimes((t-s) \bx)$ for all $t\ge s \ge 0$.
\end{enumerate}
\end{proposition}
\begin{proof}[Proof of \Cref{ext:lem:gronwall}]
    \begin{enumerate}[label=\arabic*.)] 
    \item This simply follows from the flow property of the linear differential equation defining $\mathcal{S}(\gamma)$:
    Let $u\ge s\ge0$ be arbitrarily fixed. Since $\mathcal{S}(\gamma)_{s,u} \in \TT_1$ is invertible we can define $S(t):= \mathcal{S}(\gamma)_{s,u}^{-1}\mathcal{S}(\gamma)_{s,t}$ for all $t \ge u$.
    Then it holds for all $t \ge u$
    \begin{align*}
        S(t) &= \mathcal{S}(\gamma)_{s,u}^{-1}\mathcal{S}(\gamma)_{s,t} \\
        &= \mathcal{S}(\gamma)_{s,u}^{-1}\left( 1 + \int_s^t\mathcal{S}(\gamma)_{s,r} \dot{\gamma}(r) \dd{r} \right) \\
        &= \mathcal{S}(\gamma)_{s,u}^{-1}\left( 1 + \int_s^u\mathcal{S}(\gamma)_{s,r} \dot{\gamma}(r) \dd{r} + \int_u^t\mathcal{S}(\gamma)_{s,r} \dot{\gamma}(r) \dd{r} \right)\\
        &= \mathcal{S}(\gamma)_{s,u}^{-1}\left(\mathcal{S}(\gamma)_{s, u} + \int_u^t\mathcal{S}(\gamma)_{s,r} \dot{\gamma}(r) \dd{r} \right)\\
        &= 1 + \int_u^t\mathcal{S}(\gamma)_{s,u}^{-1}\mathcal{S}(\gamma)_{s,r} \dot{\gamma}(r) \dd{r} \\
        &= 1 + \int_u^t S(r) \dot{\gamma}(r) \dd{r},
    \end{align*}
    which confirms that $S(t) = \mathcal{S}(\gamma)_{u,t}.$
    \item This follows directly by recalling that $\delta_\lambda: \TT \to \TT$ is an algebra morphism.

        \item Note that for any $\bx\in\TT_1$ we have $\pi_{(0,N)}\log_\otimes(\bx) = \log_{\otimes_N}(\pi_{(0,N)}\bx).$ Hence, it suffices to conclude that for the truncated development $X := \pi_{(0,N)} \mathcal{S}(\gamma)$ with $X = \int_0^\cdot X_t \otimes_N \pi_{(0,N)}\dot\gamma(t)\dd{t}$ we have $X(t) \in G^N(V)$ for all $t\ge0$ if and only if $\pi_{(0,N)}\dot\gamma(t) \in \mathfrak{g}^N(V)$ for a.e. $t\ge0$.
    Here, the `if'-direction follows from the well-posedness of the differential equation \eqref{ext:eq:cartan} (in the geometric sense) defining the Cartan development (see  \cite{iserles1999solution}). 
    The `only if'-direction follows by observing that for the diagonal derivative we have $\pi_{(0,N)}\dot\gamma(t) = \lim_{h\to0}\frac{1}{h}X(t)^{-1}X(t+h) \in \mathfrak{g}^N(V)$ (see \cite[Proposition 2.6]{bellingeri2022smooth}).
    Alternatively, the `if'-direction is proven directly using the Magnus expansion of $\log_\otimes(\mathcal{S}(\gamma)_t)$ in $\TT_0$ (see e.g. \cite[Theorem~5.1.(ii)]{FHT22}).
    \item By definition of the tensor exponential it holds
    \begin{align*}
        1 +\int_s^t \exp_\otimes(\bx (u-s))\bx\dd{u} &= 1 +  \int_s^t \sum_{k=0}^\infty\frac{1}{k!}\bx^{\otimes(k+1)} (u-s)^{k}\dd{u}\\
        &= 1 + \sum_{k=0}^\infty\frac{1}{k!}\bx^{\otimes(k+1)} \int_s^t  (u-s)^{k}\dd{u} \\
        &= \exp_\otimes(\bx (t-s)),
    \end{align*}
    for all $t\ge s\ge 0$, stressing that, when projected to tensor levels, the second equality is simply an interchange of a finite summation and integration.
\end{enumerate}
\end{proof}

The next proposition provides estimates for the free development with respect to the norm $\Vert \cdot \Vert_1$ on the $1$-summable tensor series $\TT^1$ defined in \eqref{eq:def_Tp}.
For ease of notation we also define %
for any $\beta: [0,\infty) \to \TT$ componentwise absolutely continuous  %
\begin{equation*}\onevartwo{\beta}{s}{t} = \int_{s}^t \Vert\dot\beta(u)\Vert_1 \dd{u}, \qquad t \ge s \ge 0.\end{equation*}

As it is used frequently, we present the following special case of Gr\"onwall's inequality in a form convenient for later reference.
\begin{lemma}\label{lem:special_gronwal}
    Let $C\in \R$, let $\beta: [0,\infty) \to \R_+$ and $\phi:[0,\infty) \to \R$ be (locally) integrable, and let $\varphi:[0,\infty) \to \R$ be continuous such that
    \begin{align*}
    \varphi(t) \le&~ C + \int_0^t \left(\varphi(u) + \phi(u) \right)\beta(u)\dd{u}, \qquad t\ge0.
    \end{align*}
    Then it holds
    \begin{align*}
        \varphi(t) \le&~ Ce^{\int_0^t \beta(r)\dd{r}} + \int_0^t \phi(u) e^{\int_u^t \beta(r)\dd{r}}\beta(u)\dd{u}, \qquad t\ge0.
    \end{align*}
\end{lemma}
\begin{proof}
    Using integration by parts, one verifies that
    \begin{align*}
        \rho(t) = Ce^{\int_0^t \beta(r)\dd{r}} + \int_0^t \phi(u) e^{\int_u^t \beta(r)\dd{r}}\beta(u)\dd{u}, \qquad t\ge0,
    \end{align*}
    satisfies the integral inequality with equality. Hence by linearity
    \begin{align*}
        \varphi(t) - \rho(t) \le&~\int_0^t \left(\varphi(u) - \rho(u) \right)\beta(u)\dd{u}, \qquad t\ge0.
    \end{align*}
    Hence, $\varphi - \rho \le 0$ (as implied from a standard version of Gr\"onwall's inequality, see e.g. \cite[Proposition 6.1.4]{applebaum2009levy}) and the claim follows.
\end{proof}

\begin{proposition}\label{ext:prop:free_dev_estimates}
 Let $\gamma = \int_0^\cdot \dot{\gamma}(s) \dd{s}$ be a componentwise absolutely continuous path in $\TT_0$ with $\onevartwo{\gamma}{s}{t} < \infty$ for some $t \ge s \ge0$.
 Then the following estimates hold for the free development $\mathcal{S}(\gamma)$:
\begin{enumerate}[label=\arabic*.)]
    \item\label{ext:itm:bell_estimate} For all $n\in\N$:
    \begin{align*}
    |\pi_n \mathcal{S}(\gamma)_{s,t}| \le \frac{1}{n!} B_n\Big(1! \onevartwo{\pi_1\gamma}{s}{t},\, \dots,\, n! \onevartwo{\pi_n\gamma}{s}{t}\Big),
    \end{align*}
    where 
    $B_n(y_1, \dots, y_n) =  \frac{\dd^n}{\dd\lambda^n} \exp(\sum_{k=1}^n y_k \frac1{k!}\lambda^k)\vert_{\lambda=0}$ denotes the $n$th complete exponential Bell polynomial\footnote{\cite[Section~3.3]{comtet1974advanced}.}.
    \item \label{ext:itm:grnwld_estimate}
    \begin{equation*}\Vert \mathcal{S}(\gamma)_{s,t} \Vert_{1} \le  e^{\onevartwo{\gamma}{s}{t}}.\end{equation*}
    In particular $\mathcal{S}(\gamma): [0,\infty) \to \TT^1$ is absolutely continuous with 
    \begin{equation*}\onevartwo{\mathcal{S}(\gamma)}{s}{t}  \le  e^{\onevartwo{\gamma}{s}{t}} - 1.\end{equation*}
    
    \item\label{ext:itm:lpschtz_estimate} For any $\beta = \int_0^\cdot\dot{\beta}(u) \dd{u}: [0,\infty) \to \TT_0$  %
    with $\onevartwo{\beta}{0}{t} < \infty$:%
    \begin{align*}
    \Vert\mathcal{S}(\gamma)_{s,t} - \mathcal{S}(\beta)_{s,t}\Vert_1 
    \le&~ \int_s^t e^{\onevartwo{\gamma}{s}{u}+ \onevartwo{\beta}{u}{t}} \Vert \dot{\gamma}(u) - \dot{\beta}(u)\Vert_1\dd{u} \\
    \le&~ \bigg.e^{\onevartwo{\gamma}{s}{t} + \onevartwo{\beta}{s}{t}} \onevartwo{\gamma - \beta}{s}{t}.
\end{align*}
    \item\label{ext:itm:inner_truncation} For all $N\in\mathbb{N}$:
    \begin{equation*}\Vert \mathcal{S}(\gamma)_{s,t} - \mathcal{S}(\pi_{(0,N)}\gamma)_{s,t}\Vert_1 \le e^{\onevartwo{\gamma}{s}{t}} \onevartwo{\gamma - \pi_{(0,N)}\gamma}{s}{t}.\end{equation*}
    \item\label{ext:itm:outer_truncation}
    For all $M, N \in \N$ with $M\ge N$:
    \begin{align*}  
        \Vert \mathcal{S}(\pi_{(0,N)}\gamma)_{s,t} - \pi_{(0,M-1)}&\mathcal{S}(\pi_{(0,N)}\gamma)_{s,t}\Vert_1 \\
        &\le e^{\onevar{\pi_{(0,N)}\gamma}{t}} \frac{(\onevar{\pi_{(0,N)}\gamma}{t})^{\ceil{M /N}}}{\ceil{M /N}!}.
    \end{align*}
\end{enumerate}    
\end{proposition}
\begin{proof}[Proof of \Cref{ext:prop:free_dev_estimates}] 
Without loss of generality we prove the claims only for $s=0$. 
The general case simply follows by shifting the path $\gamma(\cdot + s)$.
\begin{enumerate}[label=\arabic*.)]
    \item Define $x_n(u) := \onevar{\pi_n\gamma}{u} =  \int_0^u |\pi_n \dot{\gamma}(r)|\dd{r}$ for all $u\le t$.
Then using 
\begin{equation*}\frac\partial{\partial{y_i}}B_n(y_1, \dots, y_n) = \binom{n}{i} B_{n-i}(y_1, \dots, y_{n-i}), \qquad n \ge i > 0,\end{equation*}
the claim follows inductively with $B_0 \equiv 1$ and the induction step
\begin{align*}
    |\pi_{n} \mathcal{S}(\gamma)_u| 
    &\le  \sum_{k=1}^{n}\int_0^u |\pi_{n-k} \mathcal{S}(\gamma)_r| |\pi_{k}\dot \gamma(r)| \dd{r} \\
    &\le \sum_{k=1}^{n}\int_0^u \frac{1}{(n-k)!} B_{n-k}(1 x_1(r), \dots, (n-k)! x_{n-k}(r)) \dot{x}_{k}(r) \dd{r}\\
    &\quad= \frac{1}{n!}\sum_{k=1}^{n}\int_0^u \binom{n}{k} B_{n-k}(1 x_1(r), \dots, (n-k)! x_{n-k}(r)) k!\dot{x}_{k}(r) \dd{r} \\
    &\quad= \frac{1}{n!}B_n(1!x_1(u), \dots, n!x_n(u)).
\end{align*}
    \item The first inequality follows from summation of the bound in \ref{ext:itm:bell_estimate} and the definition of the Bell polynomials. Alternatively, using the compatibility of the norm $|\cdot|$ we obtain for any $N\in\mathbb{N}$ the estimate
\begin{align*}
    \Vert \pi_{(0,N)}\mathcal{S}(\gamma)_t\Vert_1 &= \bigg\Vert 1 + \int_0^t \pi_{(0,N)}(\mathcal{S}(\gamma)_s\otimes {\dot\gamma}(s))\dd{s}\bigg\Vert_1 \\
    &\le1  + \sum_{n=0}^N\sum_{k=1}^n\int_0^t \vert\pi_{n-k}\mathcal{S}(\gamma)_s\vert \vert \pi_k\dot{\gamma}(s)\vert\d{s}, \\
    &\le1  + \int_0^t \Vert\pi_{(0,N)}\mathcal{S}(\gamma)_s\Vert_1 \Vert \pi_{(0,N)}{\dot\gamma}(s)\Vert_1\d{s},\quad t\ge0.
\end{align*}
The statement now follows from Gr\"onwall's inequality and then letting $N\to\infty$.
Now suppose that $\int_0^t\Vert{\dot\gamma}(s)\Vert_1\dd{s} <\infty$. 
By definition $\mathcal{S}(\gamma)$ is componentwise absolutely continuous.
Then applying \cref{ext:prop:young} and the just proven estimate we obtain
\begin{align*}
    \onevartwo{\mathcal{S}(\gamma)}{0}{t} &\le  \int_0^t \left\Vert \mathcal{S}(\gamma)_s\otimes {\dot\gamma}(s)\right\Vert_1\dd{s} \\
    &\le  \int_0^t \left\Vert \mathcal{S}(\gamma)_s\Vert_1 \Vert {\dot\gamma}(s)\right\Vert_1\dd{s} \\
    &\le  \int_0^t e^{\int_0^s \Vert {\dot\gamma}(r)\Vert_1\dd{r}} \Vert {\dot\gamma}(s) \Vert_1\dd{s}
    = e^{\onevartwo{\gamma}{0}{t}}-1.
\end{align*}
\item By definition of the free development we have for all $u\ge0$:
\begin{align*}
    \mathcal{S}(\gamma)_u - \mathcal{S}(\beta)_u = \int_0^u (\mathcal{S}(\gamma)_r - \mathcal{S}(\beta)_r)\otimes\dot{\beta}(r)\dd{r} + \int_0^u \mathcal{S}(\gamma)_r \otimes\big( \dot{\gamma}(r) - \dot{\beta}(r)\big)\dd{r}.
\end{align*}
Hence, by the triangle inequality, \cref{ext:prop:young}, Gr\"onwall's inequality (more precisely \cref{lem:special_gronwal}) and \ref{ext:itm:grnwld_estimate} from above we can estimate for all $u \le t$:
\begin{align*}
    \Vert\mathcal{S}(\gamma)_u - &\mathcal{S}(\beta)_u\Vert_1 \\
    \le&~ \int_0^u \Vert\mathcal{S}(\gamma)_r - \mathcal{S}(\beta)_r\Vert_1 \Vert \dot{\beta}(r)\Vert_1\dd{r} + \int_0^u \Vert \mathcal{S}(\gamma)_r \Vert_1 \Vert \dot{\gamma}(r) - \dot{\beta}(r)\Vert_1\dd{r}\\
    \le&~ \int_0^u e^{\onevartwo{\beta}{r}{u}}\Vert \mathcal{S}(\gamma)_r \Vert_1 \Vert \dot{\gamma}(r) - \dot{\beta}(r)\Vert_1\dd{r}\\
    \le&~ \int_0^u e^{\onevartwo{\beta}{r}{u} + \onevar{\gamma}{r}}  \Vert \dot{\gamma}(r) - \dot{\beta}(r)\Vert_1\dd{r}.
\end{align*}

\item Applying \ref{ext:itm:lpschtz_estimate} for the case $\beta = \pi_{(0,N)}\gamma$ and using that \begin{equation*}\onevar{\gamma}{t} = \onevar{\pi_{(0,N)}\gamma}{t} + \onevar{\gamma - \pi_{(0,N)}\gamma}{t}\end{equation*}
we have
\begin{align*}
    &\Vert\mathcal{S}(\gamma)_s - \mathcal{S}(\pi_{(0,N)}\gamma)_s\Vert_1  \\
    &\quad\le~ \int_0^s e^{\onevartwo{\pi_{(0,N)}\gamma}{u}{s} + \onevar{\gamma}{u}}  \Vert \dot{\gamma}(u) - \pi_{(0,N)}\dot\gamma(u)\Vert_1\dd{u}\\
    &\qquad=~ e^{\onevar{\pi_{(0,N)}\gamma}{s}}\int_0^s e^{\onevar{\gamma - \pi_{(0,N)}\gamma}{u}}  \Vert \dot{\gamma}(u) - \pi_{(0,N)}\dot\gamma(u)\Vert_1\dd{u}\\
    &\qquad=~ e^{\onevar{\pi_{(0,N)}\gamma}{s}} \left(e^{\onevar{\gamma - \pi_{(0,N)}\gamma}{s}}  -1\right)\\
     &\qquad=~ e^{\onevar{\gamma}{s}} \left(1 -e^{-\onevar{\gamma - \pi_{(0,N)}\gamma}{s}}\right)\\
     &\quad\qquad\le~ e^{\onevar{\gamma}{s}} \onevar{\gamma - \pi_{(0,N)}\gamma}{s}.
\end{align*}

\item Recall the following combinatorial expression for the Bell polynomials (see \cite[Section~3.3]{comtet1974advanced})
\begin{equation*}B_n(y_1, \dots, y_n) = n! \sum_{\ell \in L^n_n} \prod_{i=1}^n \frac{(y_i)^{\ell_i}}{(i!)^{\ell_i} \ell_i!},\qquad n > 0,\end{equation*}
where for $m\le n$ we set $L_{n}^m := \{ \ell \in\mathbb{N}^m \;\vert\; \ell_1 + 2\ell_2 + \dots + m\ell_m = n\}$.
Using this and \ref{ext:itm:bell_estimate} we have
\begin{align*}
    \Vert \mathcal{S}(\pi_{(0,N)}\gamma)_t - \pi_{(0,M-1)}\mathcal{S}&(\pi_{(0,N)}\gamma)_t\Vert_1 \\
    &\le\sum_{n=M}^\infty \frac{1}{n!} B_n(1! x_1(t), \dots, N!x_N(t), 0, \dots)  \\
    &\quad =\sum_{n=M}^\infty \;\sum_{\ell \in L^N_n} \;\prod_{i=1}^N \frac{(x_i(t))^{\ell_i}}{\ell_i!}.
\end{align*}
We further notice that
\begin{align*}
    \dot{\bigcup}_{n\ge M}L^N_n 
    =& \left\{\ell \in \N^N \;\Big\vert\; \ell_1 + 2\ell_2 + \dots + N\ell_N \ge M \right\} \\
    \subseteq& \left\{\ell \in \N^N \;\Big\vert\; \ell_1 + \ell_2 + \dots + \ell_N \ge \frac{M}{N} \right\} \\
    \subseteq&~\dot{\bigcup}_{k \ge \ceil{\frac{M}{N}}}\left\{\ell \in \N^N \;\Big\vert\; \ell_1 + \ell_2 + \dots + \ell_N = k \right\}.
\end{align*}
Hence, defining $K= \ceil{\frac{M}{N}} \ge 1$ and returning to the estimate, we have
\begin{align*}
    \Vert \mathcal{S}(\pi_{(0,N)}\gamma)_t - \pi_{(0,M-1)}\mathcal{S}(\pi_{(0,N)}\gamma)_t\Vert_1 
     & \le \sum_{k \ge K} \sum_{\ell_1 + \dots + \ell_N = k} \;\prod_{i=1}^N \frac{(x_i(t))^{\ell_i}}{\ell_i!} \\
     & = \sum_{k \ge K} \frac{1}{k!} (x_1(t) + \cdots + x_N(t))^k\\
     & = \frac{\onevar{\pi_{(0,N)}\gamma}{t}^K}{K!}\sum_{k \ge K} \frac{K!}{k!} \onevar{\pi_{(0,N)}\gamma}{t}^{k-K} \\
     &\le e^{\onevar{\pi_{(0,N)}\gamma}{t}}\frac{\onevar{\pi_{(0,N)}\gamma}{t}^K}{K!}.
\end{align*}
\end{enumerate}
\end{proof}

\begin{remark}\label{ext:rmk:free_dev_tensor_norms}
    The proof of \cref{ext:prop:free_dev_estimates} only uses the compatibility of the tensor norms $|\cdot|$ on $V^{\otimes n}$ with tensor multiplication. Hence, the estimates remain valid when $|\cdot|$ is replaced by any other compatible family of tensor norms.
\end{remark}

\begin{remark}\label{ext:rmk:grnwl_estimate}
Note that it is not possible to adapt the general estimate in \ref{ext:itm:grnwld_estimate} by using a weaker $\TT^p$-norm on the right-hand side, for any choice of $\TT^q$-norm on the left-hand side.
Indeed, similarly to the example provided after \cref{ext:prop:young}, for any $p,q>1$ one finds a tensor series $\bx \in \TT^p$ such that $\Vert\exp_\otimes(\bx)\Vert_q = \infty$.
\end{remark}

\subsection{Lie group developments of semimartingales}\label{ext:sec:signatures_semimartingales}

We recall some considerations from \cite{estrade1992exponentielle,hakim1986exponentielle} for the case of free Lie groups.
Let $(\Omega, \F, (\F_t), \P)$ be a filtered probability space satisfying the usual conditions.
A measurable map \begin{equation*}\gamma: [0,\infty) \times \Omega \to \mathfrak{g}^N(V), \qquad (t,\omega)\mapsto \gamma_t(\omega),\end{equation*} is called a $\mathfrak{g}^N(V)$-valued semimartingale if $(\langle \mathfrak{u}_w, \gamma_t \rangle)_{t\ge0}$ defines a semimartingale for all $w\in\mathcal{I}$.
In what follows, we will always refer to the càdlàg version of a semimartingale without explicitly mentioning it, i.e., we assume $\lim_{s\downarrow t} (\gamma_s(\omega) - \gamma_t(\omega)) = 0$ for all $\omega\in\Omega$ and $t\ge0$. 
Further, we use the usual notation $\gamma_{t-} := \lim_{s\uparrow t,s<t} \gamma_s$ and $\Delta\gamma_t = \gamma_t - \gamma_{t-}$.
We consider the \emph{stochastic Lie-exponential} of $\gamma$, defined as the solution to the linear Marcus-type stochastic differential equation \cite{applebaum2009levy,friz2017general,kurtz1995strtonovich,marcus1978modeling,marcus1981modeling}
\begin{align}\label{ext:eq:GN-development}
    \d X_t =  X_{t} \otimes_N \diamond\,\d{\gamma}_t, \quad t \ge 0, \quad X_0 = 1 \in G^N(V).
\end{align}
The embedding into the extended tensor algebra allows us to construct the solution to this equation similarly to the previous section, without any geometric arguments.
Specifically, we define the \textit{signature} $\Sig{\gamma}_{0,\cdot}$ as the full geometric development of $\gamma$ in $\TT$, i.e., the solution to
\begin{align}\label{ext:eq:marcus_signature}
    \d{S_t} = S_{t} \otimes \diamond\,\d{\gamma}_t, \qquad t \ge0, \qquad S_0 = 1 \in \TT,
\end{align}
which in Itô form reads
\begin{align}\label{ext:eq:ito_signature}
    S_t = 1 + \int_0^t S_{u-} \otimes \d{\gamma}_u + \frac{1}{2}\int_0^t S_{u} \otimes \d{\langle\gamma^c\rangle}_u + \sum_{0 < u \le t}S_{u-}(e^{\Delta \gamma_u} - 1 - \Delta \gamma_u), \qquad t \ge0.
\end{align}
Once again, the graded structure of $\TT$ allows us to solve the above equation by iterated integration.
The fact that $\Sig{\gamma}$ takes values in $\mathcal{G}$ can be seen either by the definition of the Marcus integration (see \cite{friz2017general,kurtz1995strtonovich}) or by formulating the Magnus expansion of the tensor-logarithm of $\Sig{\gamma}$, which can be done using only Itô calculus (see \cite{FHT22}).
The truncated signature $X := \pi_{(0,N)}\Sig{\gamma}_{0,\cdot}$ then solves the equation \eqref{ext:eq:GN-development}.
Interestingly, there is a corresponding \textit{stochastic (Lie-)logarithm}:
\begin{proposition}[\cite{estrade1992exponentielle,hakim1986exponentielle}]\label{ext:prop:lie_logarithm}Let $X$ be a $G^N(V)$-valued semimartingale, i.e., for any smooth function $f: G^N(V) \to \R$ the process $f(X)$ is a semimartingale.
Then there exists a unique $\mathfrak{g}^N(V)$-valued semimartingale $\gamma$ such that $X = \pi_{(0,N)}\Sig{\gamma}_{0,\cdot}$.
\end{proposition}

\subsection{Inhomogeneous L\'evy processes in free Lie groups}\label{ext:sec:inhom_levy}

While here we specialize to the free step-$N$ nilpotent Lie group (as previously studied in the homogeneous case in \cite{friz2017general} and \cite{chevyref2018random}), let us mention that the definition of L\'evy processes in general Lie groups is standard and their properties are well understood:
the interested reader may refer to \cite{liao2004levy} for an introduction to the topic.
The terminology of inhomogeneous L\'evy processes is less standard. For the purposes of this exposition, we introduce a definition below. For the reader's convenience, we gradually transition from the definition for processes with independent increments to the definition of L\'evy processes.

\begin{definition}
Let $(\Omega, \F, (\F_t), \P)$ be a filtered probability space and let $(G, \odot)$ be a Lie group
with identity element $1_G$. An adapted process $X: [0,\infty) \times \Omega \to G$, $(t,\omega) \mapsto X_t(\omega)$, is called a \emph{process with independent increments} ($(G,\odot)$-PII) if
    \begin{enumerate}[label=(\roman*)]
        \item For all $t\ge s\ge0$ it holds that $X_{s,t}:=(X_s)^{-1} \odot X_t$ is independent of $\F_s$.
    \end{enumerate}
    The process $X$ is called \emph{semimartingale $(G,\odot)$-PII} if additionally
    \begin{enumerate}[resume,label=(\roman*)]
        \item $X$ is càdlàg, i.e., $\lim_{s\downarrow t} X_{s} = X_t$ and $\lim_{s\uparrow t} X_{t} \in G$ for $t\ge0$,
        \item $f(X)$ is a semimartingale for any smooth $f:G \to \R$.
    \end{enumerate}
    The process $X$ is called an \emph{inhomogeneous L\'evy process in $(G,\odot)$} if additionally
    \begin{enumerate}[resume,label=(\roman*)]
        \item $X_0 = 1_G$,
        \item for all $t> 0$ it holds $\lim_{\varepsilon \downarrow 0}X_{t-\varepsilon,t} = 1_G$ in probability.
    \end{enumerate}
    Finally, $X$ is called a \emph{L\'evy process in $(G,\odot)$} if additionally
    \begin{enumerate}[resume,label=(\roman*)]
        \item $X$ has stationary increments.
    \end{enumerate}
\end{definition}

The theory developed by G.~A.~Hunt \cite{hunt1956semi} provides a characterization of L\'evy processes in Lie groups via their generators, which are expressed in terms of left-invariant vector fields on the group.
Although extensions to the inhomogeneous case are available in \cite{liao2014inhomogeneous}, we instead choose to work with the stochastic logarithm, which allows us to use classical semimartingale characteristics in finite-dimensional vector spaces.
As a direct consequence of \cref{ext:prop:lie_logarithm} we get
\begin{corollary}\label{ext:cor:anitdev_levy} A process $X$ is an inhomogeneous L\'evy process in $(G^N(V), \otimes_N)$ if and only if it is the development of an inhomogeneous L\'evy process $\lvyalg$ in $(\mathfrak{g}^N(V), +)$, i.e., if $X = \pi_{(0,N)}\Sig{\lvyalg}_{0,\cdot}$.
\end{corollary}
\begin{remark}
Trivially, $(\mathfrak{g}^N(V), +)$ is a Lie group with identity $0\in\mathfrak{g}^N(V)$. The above corollary could be stated in a general Lie setting \cite{applebaum1993levy}.
\end{remark}

Classical results for processes with independent increments, now in finite-dimensional vector spaces \cite[Section~II.4.c]{jacod2003limit}, imply that the law of an inhomogeneous L\'evy process $\lvyalg$ in $(\mathfrak{g}^N(V), +)$ is determined by a \emph{characteristic triplet}.
For the state-space $\mathfrak{g}^N(V)$ the triplet $(B, \Sigma, \nu)$ is given by a continuous function $B: [0,\infty) \to \mathfrak{g}^N(V)$, a continuous function $\Sigma:[0,\infty) \to \mathfrak{S}^N_+(V) \subset T^{2N}(V)$ where \begin{equation*}\mathfrak{S}^N_+(V) :=\bigg\{ \sum \mathop{}_{v,w\in \mathcal{I}_N} \mathfrak{u}_v \otimes \mathfrak{u}_w \Sigma^{v,w} \;\Big\vert\; (\Sigma^{w,v})\in\R^{|\mathcal{I}_N|\times |\mathcal{I}_N|} \text{ positive semidefinite}\bigg\},\end{equation*}
such that $\Sigma_0 = 0$, $\Sigma_t - \Sigma_s \in \mathfrak{S}^N_+(V)$ for all $0\le s \le t$, and a measure $\nu$ on $\mathfrak{g}^N(V) \times [0,\infty)$ satisfying
\begin{equation*}\int_{0}^t \int_{\mathfrak{g}^N(V)}(\Norm{x}^2 \wedge 1)\nu(\d{x}, \d{s}) < \infty, \qquad t \ge 0.\end{equation*}
In this work, we only consider inhomogeneous L\'evy processes $\lvyalg$ that have a \emph{differential triplet} %
$(b, a, K)$, i.e., such that
\begin{equation*}B = \int_0^\cdot b(t) \d{t}, \quad \Sigma = \int_0^\cdot a(t) \d{t}, \quad \nu(\d{x}, \d{t}) = K_t(\d{x})\d{t},\end{equation*}
for integrable functions $b: [0,\infty) \to \mathfrak{g}^N(V)$, $a: [0,\infty) \to \mathfrak{S}^N(V)$ and a family of L\'evy measures $(K_t)_{t\ge0}$ on $\mathfrak{g}^N(V)$.
\begin{notation}\label{ext:not:inhom_levy}
We write $\lvyalg \sim \mathscr{L}^N(b, a, K)$ if there exists some filtered probability space $(\Omega, \F, (\F_t), \P)$ satisfying the usual conditions such that $\lvyalg: [0,\infty) \times \Omega \to \mathfrak{g}^N(V)$ is an inhomogeneous L\'evy process in $(\mathfrak{g}^N(V), +)$ with differential triplet $(b, a, K)$.
\end{notation}

In order to give a more direct definition, we use the martingale characterization from \cite{stroock1975diffusion}:
$\lvyalg \sim \mathscr{L}^N(b, a, K)$
if and only if $\lvyalg_0 = 0$ and for all $f\in C^2_c(\mathfrak{g}^N(V);\R)$,
\begin{align}\label{ext:eq:mrtgl_problem}
    M^f_t := f(\lvyalg_t) - \int_0^t \mathcal{L}_sf(\lvyalg_s)\dd{s}, \qquad t\ge0,
\end{align}
defines a martingale, where
\begin{equation}\label{ext:eq:generator}
    \begin{split}
        \mathcal{L}_tf(x) :=&~ \langle b(t), \nabla f(x)\rangle + \frac12 \langle a(t), \nabla^2 f(x)\rangle \\&~+ \int_{\mathfrak{g}^N(V)}\big(f(x+y) - f(x) -\langle \nabla f(x), y\rangle 1_{|y|\le 1}\big)K_t(\dd{y}),
    \end{split}
\end{equation}
for all $x\in \mathfrak{g}^N(V)$ and $t\ge0$ with $\nabla f$ and $\nabla^2 f$ denoting the gradient and Hessian of $f$ in the coordinates $(\mathfrak{u}_w)_{w\in\mathcal{I}}$.
Specifically, 
\begin{equation*}\nabla f(x) = \sum_{w\in\mathcal{I}^N}\mathfrak{u}_w \frac{\dd}{\dd h}f(x+ h\mathfrak{u}_w) \in \mathfrak{g}^N(V), \qquad \nabla^2 f(x) = \sum_{v\in\mathcal{I}^N}\mathfrak{u}_v \otimes \frac{\dd}{\dd h} \nabla f(x + h\mathfrak{u}_v) \in \mathfrak{S}^N(V),\end{equation*}
where $\mathfrak{S}^N(V)\subset T^{2N}(V)$ is defined as $\mathfrak{S}^N_+(V)$ above, but requiring only symmetry.

Note that the characteristic triplet depends on the choice of \emph{truncation function} \cite[Definition~II.2.3]{jacod2003limit} that separates small and large jumps, which we have implicitly fixed to the standard $x\mapsto x \indic{\Norm{x} \le 1}$.
When dilating the process, the truncation function only influences the scaling of the drift component.

\begin{lemma}\label{ext:rem:triplet_scaling}
     Let $\lvyalg \sim\mathscr{L}^N(b, a, K)$ and $\lambda > 0$. Then $\delta_\lambda \lvyalg \sim\mathscr{L}^N(b^{\lambda}, a^{\lambda}, K^{\lambda})$ where $K^{\lambda}(\cdot) = K(\delta_{\lambda^{-1}}(\cdot))$, $a^{\lambda} = \delta_{\lambda} a$ and \begin{equation*}b^{\lambda} = \delta_\lambda \left(b +\int_0^\cdot \int_{\mathfrak{g}^N(V)} (\indic{\Norm{x}\le 1} - \indic{\Norm{\delta_\lambda x}\le 1})K_t(\dd x)\dd{t}\right).\end{equation*}
\end{lemma}
\begin{proof}[Proof of \Cref{ext:rem:triplet_scaling}]
For any $f\in C^2_c(\mathfrak{g}^N(V);\R)$ we set $f^\lambda(x) := f(\delta_\lambda x)$.
Then it holds \begin{equation*}\nabla f^\lambda(x) = \sum_{w\in\mathcal{I}^N}\mathfrak{u}_w \frac{\dd}{\dd h}f(\delta_\lambda x+ h\delta_\lambda\mathfrak{u}_w) = \sum_{w\in\mathcal{I}^N}\mathfrak{u}_w \lambda^{|w|}\frac{\dd}{\dd h}f(\delta_\lambda x+ h\mathfrak{u}_w) =   \delta_\lambda \nabla f(\delta_\lambda x),\end{equation*}
and one similarly verifies, using that $\delta_\lambda$ is an algebra morphism, that 
\begin{align*}
    \nabla^2 f^\lambda (x) %
    =&\sum_{w\in\mathcal{I}^N}\delta_\lambda\mathfrak{u}_v \otimes  \delta_\lambda\left(\frac{\dd}{\dd h} \nabla f(\delta_\lambda x + h \mathfrak{u}_v)\right)
    =~\delta_\lambda \nabla^2 f(\delta_\lambda x).
\end{align*}
Thus, defining the scaled operator
\begin{equation}\label{ext:eq:modified_generator}
\begin{split}
    \mathcal{L}^\lambda_t f(x) :=&~ \langle \delta_\lambda b(t), \nabla f(x)\rangle + \frac12 \langle \delta_\lambda a(t), \nabla^2 f(x)\rangle \\&~+ \int_{\mathfrak{g}^N(V)}\big(f(x+\delta_{\lambda}y) - f(x) -\langle \nabla f(x), \delta_{\lambda} y\rangle 1_{|y|\le 1}\big)K_t(\dd{y}),
\end{split}
\end{equation}
we readily see that
\begin{align*}
    f(\delta_\lambda \lvyalg) - \int_0^t \mathcal{L}^{\lambda}_s f(\delta_\lambda \lvyalg_s) \dd{s}= f^\lambda(\lvyalg) - \int_0^t \mathcal{L}_s f^\lambda(\lvyalg_s) \dd{s}, \qquad t\ge 0.
\end{align*}
It follows from the martingale characterization of $\lvyalg$ that the above defines a martingale for all $f\in C^2(\mathfrak{g}^N(V); \R)$.
To conclude it thus suffices to argue that $\mathcal{L}^\lambda$ is of the form \eqref{ext:eq:generator} with modified characteristics $(b^\lambda, a^\lambda, K^\lambda)$. 
This is immediately visible for the second-order part. 
For the remaining parts it simply follows by transforming the integration variable in \eqref{ext:eq:modified_generator} with $\delta_{\lambda}$, and moving the displaced indicator terms into the first-order part. 
\end{proof}

\begin{remark}
    If $\lvyalg \sim \mathscr{L}^1(b,a,K)$, i.e., $\lvyalg$ is a $d$-dimensional inhomogeneous L\'evy process, then $\delta_\lambda \lvyalg = \lambda \lvyalg$ and we recover the classical scaling with $a^\lambda = \lambda^2 a$.
\end{remark}

\section{Technical proofs and auxiliary lemmas}\label{sec:technical_proofs}

This section contains the derivation of the infinite-dimensional system, the uniqueness argument for the truncated system, and the proof of the second-level specialization.

\subsection{Proofs for \Cref{sec:infinite_sys} -- The infinite-dimensional system} \label{sec:proof_derivation_pde}

\begin{proof}[Proof of \Cref{prop:extended_pde_representation}] 
By the assumptions on $\mathfrak{y}$ and $\widetilde{\mathfrak{y}}$, \cref{ext:prop:free_dev_estimates}.\ref{ext:itm:grnwld_estimate} shows that  
${\esig}(s), \widetilde{\esig}(t) \in \TT^1$ and thus $u(s,t) = \langle \esig(s), \widetilde{\esig}(t)\rangle$ is well defined for all $s,t \ge 0$.
By \cref{ext:prop:young} and \cref{ext:prop:free_dev_estimates}.\ref{ext:itm:grnwld_estimate}, we have for all $s\ge0$ the estimate
\begin{align*}
\onevar{\esig}{s} = \int_0^s\Vert \esig(r) \otimes \mathfrak{y}(r)\Vert_1 \dd{r} ~\le~ e^{\int_0^s\Vert \mathfrak{y}(r)\Vert_1 \dd{r}} - 1.
\end{align*}
An analogous estimate applies to the integral of $\widetilde{\esig}\otimes\widetilde{\mathfrak{y}}$.
Then by H\"older's inequality, for any $s, t \ge0$ we have
\begin{multline*}
\int_0^s\int_0^t\sum_{k=0}^\infty \big\vert\big\langle \pi_k\big(\esig(r)\otimes\mathfrak{y}(r)\big), \pi_k\big(\widetilde{\esig}(v)\otimes\widetilde{\mathfrak{y}}(v)\big)\big\rangle\big\vert\dd{v}\dd{r}
    \le\Big(e^{\int_0^s\Vert \mathfrak{y}(r)\Vert_1 \dd{r}} - 1\Big)\Big(e^{\int_0^t\Vert \widetilde{\mathfrak{y}}(v)\Vert_1 \dd{v}} - 1\Big) < \infty,
\end{multline*}
which allows us to apply Fubini's theorem in the following calculation
\begin{align*}
    u(s,t) =& \Big\langle 1+\int_0^s\esig(r)\otimes\mathfrak{y}(r)\dd{r}, 1+\int_0^t\widetilde{\esig}(v)\otimes\widetilde{\mathfrak{y}}(v)\dd{v}\Big\rangle \\
    =& 1 + \int_0^s\int_0^t\big\langle \esig(r)\otimes\mathfrak{y}(r), \widetilde{\esig}(v)\otimes\widetilde{\mathfrak{y}}(v)\big\rangle\dd{v}\dd{r}.
\end{align*}
For the remaining arguments to derive equation \eqref{eq:extended_pde}, we follow \cite[Proof of Theorem 5.3]{lemercier2024high}, providing full details for the reader's convenience.
We first decompose the inner product term in the above integral as follows
\begin{align}\nonumber
   \big\langle \esig(s)\otimes\mathfrak{y}(s), \widetilde{\esig}(t)\otimes\widetilde{\mathfrak{y}}(t)\big\rangle 
   =& \sum_{n=1}^\infty \sum_{k=1}^\infty  \big\langle \esig(s)\otimes\pi_n\mathfrak{y}(s), \widetilde{\esig}(t)\otimes\pi_k\widetilde{\mathfrak{y}}(t)\big\rangle \\\label{eq:proof_3_sums_1}
   =& \sum_{n=1}^\infty  \big\langle \esig(s)\otimes\pi_n\mathfrak{y}(s), \widetilde{\esig}(t)\otimes\pi_n\widetilde{\mathfrak{y}}(t)\big\rangle \\\label{eq:proof_3_sums_2}
   &+\sum_{n=1}^\infty \sum_{k=1}^{n-1}  \big\langle \esig(s)\otimes\pi_n\mathfrak{y}(s), \widetilde{\esig}(t)\otimes\pi_k\widetilde{\mathfrak{y}}(t)\big\rangle \\\label{eq:proof_3_sums_3}
    &+\sum_{n=1}^\infty \sum_{k=n+1}^{\infty}  \big\langle \esig(s)\otimes\pi_n\mathfrak{y}(s), \widetilde{\esig}(t)\otimes\pi_k\widetilde{\mathfrak{y}}(t)\big\rangle,
\end{align}
for almost all $(s,t)\in[0,\infty)^2$. Henceforth, unless otherwise stated, $s$ and $t$ are assumed to lie outside a null set where $\Vert\mathfrak{y}(s)\Vert_1$ or $\Vert\widetilde{\mathfrak{y}}(t)\Vert_1$ may be infinite.
The fact that each of the sums appearing above is absolutely convergent follows from H\"older's inequality and \cref{ext:prop:young}.
For the sum in \eqref{eq:proof_3_sums_1} we have
\begin{align*}
    \sum_{n=1}^\infty  \big\langle \esig(s)\otimes\pi_n\mathfrak{y}(s), \widetilde{\esig}(t)\otimes\pi_n\widetilde{\mathfrak{y}}(t)\big\rangle 
    &= \sum_{n=1}^\infty  \big\langle \esig(s), \widetilde{\esig}(t)\big\rangle \big\langle \pi_n\mathfrak{y}(s), \pi_n\widetilde{\mathfrak{y}}(t)\big\rangle \\
    &= u(s,t) \big\langle \mathfrak{y}(s), \widetilde{\mathfrak{y}}(t)\big\rangle.
\end{align*}
For the sum in \eqref{eq:proof_3_sums_2} we apply the identity in \cref{ext:prop:left-right-technical} to each term and obtain
\begin{align*}
\big\langle \esig(s)\otimes\pi_n\mathfrak{y}(s), \widetilde{\esig}(t)\otimes\pi_k\widetilde{\mathfrak{y}}(t)\big\rangle 
&= \big\langle \dualLeft{\esig(s)}{\widetilde{\esig}(t)}, \dualRight{\pi_k\widetilde{\mathfrak{y}}(t)}{\pi_n\mathfrak{y}(s)} \big\rangle \\&= \big\langle \dualLeftZ{\esig(s)}{\widetilde{\esig}(t)}, \dualRightZ{\pi_k\widetilde{\mathfrak{y}}(t)}{\pi_n\mathfrak{y}(s)} \big\rangle,
\end{align*}
where in the second equality we specifically use that $k<n$.
Using further that $\dualRight{\pi_k\widetilde{\mathfrak{y}}(t)}{\pi_n{\mathfrak{y}}(s)}= 0$ for $k > n$ and that $\dualRight{}{}:\TT^1\times\TT^1\to\TT^1$ is bilinear and continuous, we can conclude for the sum in \eqref{eq:proof_3_sums_2}
\begin{align*}
    \sum_{n=1}^\infty \sum_{k=1}^{n-1}  \big\langle \esig(s)\otimes\pi_n\mathfrak{y}(s), \widetilde{\esig}(t)\otimes\pi_k\widetilde{\mathfrak{y}}(t)\big\rangle = \big\langle \widetilde{\phi}(s,t), \dualRightZ{\widetilde{\mathfrak{y}}(t)}{\mathfrak{y}(s)} \big\rangle.
\end{align*}
An analogous argument for the sum in \eqref{eq:proof_3_sums_3} proves
\begin{align*}
    \sum_{n=1}^\infty \sum_{k=n+1}^{\infty}  \big\langle \esig(s)\otimes\pi_n\mathfrak{y}(s), \widetilde{\esig}(t)\otimes\pi_k\widetilde{\mathfrak{y}}(t)\big\rangle = \big\langle {\phi}(s,t), \dualRightZ{\mathfrak{y}(s)}{\widetilde{\mathfrak{y}}(t)} \big\rangle.
\end{align*}
It remains to derive the corresponding equations for $\phi$ and $\widetilde{\phi}$.
To this end, let $w\in\mathcal{W}\setminus\{\varnothing\}$
\begin{align*}
    \big\langle \phi(s,t), e_w \big\rangle = \big\langle \esig(s), \widetilde{\esig}(t)\otimes e_w \big\rangle = \int_0^s\big\langle \esig(r)\otimes\mathfrak{y}(r), \widetilde{\esig}(t)\otimes e_w \big\rangle\dd{r}.
\end{align*}
Then using \cref{ext:prop:left-right-technical} once again we see that
\begin{align*}
    \big\langle \esig(s)\otimes\mathfrak{y}(s), \widetilde{\esig}(t)\otimes e_w \big\rangle = u(s,t)\big\langle \mathfrak{y}(s), e_w \big\rangle + \big\langle \widetilde{\phi}(s,t), \dualRightZ{e_w}{\mathfrak{y}(s)} \big\rangle + \big\langle {\phi}(s,t), \dualRightZ{\mathfrak{y}(s)}{e_w} \big\rangle \\
    = u(s,t)\big\langle \mathfrak{y}(s), e_w \big\rangle + \big\langle \dualLeftZ{\widetilde{\phi}(s,t)}{\mathfrak{y}(s)}, e_w\big\rangle + \big\langle {\phi}(s,t)\otimes\mathfrak{y}(s), e_w \big\rangle.
\end{align*}
Finally, recalling that $\langle \phi, e_{\varnothing} \rangle \equiv 0$ and writing $\phi(s,t) = \sum_{w\in\mathcal{W}}\big\langle \phi(s,t), e_w\big\rangle e_w$ we readily arrive at the equation for $\phi$.
The derivation of the corresponding equation for $\widetilde{\phi}$ is entirely analogous.
\end{proof}

\subsection{Uniqueness of the truncated system}\label{sec:proof_truncated_uniqueness}

For completeness, we give the finite-dimensional uniqueness argument used in \cref{thm:truncated}.

\begin{proof}[Uniqueness in \Cref{thm:truncated}]
Let $(w,f,\widetilde f)$ be a continuous solution to \eqref{eq:extended_pde_trunc}. Consider the Picard iteration
\begin{align*}
\phi_{k+1}(s,t)
&:=
\int_0^s\left\{
 w(r,t)\mathfrak y^{\mm}(r)
 +\phi_k(r,t)\otimes\mathfrak y^{\mm}(r)
 +\dualLeftZ{\widetilde\phi_k(r,t)}{\mathfrak y^{\mm}(r)}
\right\}\dd r,\\
\widetilde\phi_{k+1}(s,t)
&:=
\int_0^t\left\{
 w(s,r)\widetilde{\mathfrak y}^{\nn}(r)
 +\widetilde\phi_k(s,r)\otimes\widetilde{\mathfrak y}^{\nn}(r)
 +\dualLeftZ{\phi_k(s,r)}{\widetilde{\mathfrak y}^{\nn}(r)}
\right\}\dd r,
\end{align*}
starting from $(\phi_0,\widetilde\phi_0)=(f,\widetilde f)$. By the projection identity \eqref{eq:dual_truncation_relation} and the equations for $(f,\widetilde f)$, the iteration leaves the relevant levels invariant:
\begin{equation*}
\pi_{(0,\n)}\phi_k=f,
\qquad
\pi_{(0,\m)}\widetilde\phi_k=\widetilde f,
\qquad k\in\N.
\end{equation*}
Using linearity together with \cref{ext:prop:young} and \cref{ext:prop:adjoint_mul}, we obtain for $k\ge1$
\begin{align*}
&\Vert\phi_{k+1}(s,t)-\phi_k(s,t)\Vert_1
+\Vert\widetilde\phi_{k+1}(s,t)-\widetilde\phi_k(s,t)\Vert_1\\
&\le
\int_0^s
\Big(
\Vert\phi_k(r,t)-\phi_{k-1}(r,t)\Vert_1
+\Vert\widetilde\phi_k(r,t)-\widetilde\phi_{k-1}(r,t)\Vert_1
\Big)
\Vert\mathfrak y^{\mm}(r)\Vert_1\dd r\\
&\quad+
\int_0^t
\Big(
\Vert\phi_k(s,v)-\phi_{k-1}(s,v)\Vert_1
+\Vert\widetilde\phi_k(s,v)-\widetilde\phi_{k-1}(s,v)\Vert_1
\Big)
\Vert\widetilde{\mathfrak y}^{\nn}(v)\Vert_1\dd v.
\end{align*}
Consequently, for every $S,T>0$, a direct induction gives
\begin{multline*}
\Vert\phi_k(s,t)-\phi_{k-1}(s,t)\Vert_1
+\Vert\widetilde\phi_k(s,t)-\widetilde\phi_{k-1}(s,t)\Vert_1\\
\le
\sup_{(r,v)\in[0,S]\times[0,T]}
\Big(
\Vert\phi_1(r,v)-f(r,v)\Vert_1
+\Vert\widetilde\phi_1(r,v)-\widetilde f(r,v)\Vert_1
\Big)\\
\cdot
\frac{2^{k-1}}{(k-1)!}
\left(
\int_0^s\Vert\mathfrak y^{\mm}(r)\Vert_1\dd r
+
\int_0^t\Vert\widetilde{\mathfrak y}^{\nn}(v)\Vert_1\dd v
\right)^{k-1},
\end{multline*}
for $(s,t)\in[0,S]\times[0,T]$. Hence $(\phi_k,\widetilde\phi_k)$ converges locally uniformly to a continuous fixed point $(\phi,\widetilde\phi)$ satisfying
\begin{equation*}
\pi_{(0,\n)}\phi=f,
\qquad
\pi_{(0,\m)}\widetilde\phi=\widetilde f.
\end{equation*}
Passing to the limit gives the equations for $\phi$ and $\widetilde\phi$ in \eqref{eq:extended_pde} with underlying paths $\mathfrak y^{\mm}$ and $\widetilde{\mathfrak y}^{\nn}$. The equation for $w$ follows from the projection identities established in the proof of \cref{thm:truncated}. By uniqueness of the infinite-dimensional system \eqref{eq:extended_pde}, this determines $(w,\phi,\widetilde\phi)$ uniquely, and hence also $(w,f,\widetilde f)$.
\end{proof}

\subsection{Proofs for \Cref{sec:second_level_char}  -- Second level characteristics}\label{sec:proof_second_level}

\begin{proof}[Proof of \Cref{cor:level2}]
    By assumption, the characteristic velocities $\mathfrak{y}$ and $\widetilde{\mathfrak{y}}$ respectively of $\gamma$ and $\widetilde{\gamma}$ take values in $T^2(V)$, hence the minimal truncation levels in the sense of \cref{thm:truncated} are $\nn_0=\mm_0=2$.
    Hence, denoting by $(u, f, \widetilde{f})$ the unique solution to \eqref{eq:extended_pde_trunc} with truncation level $\nn = \mm = 2$ it follows by the aforementioned theorem that \begin{equation*}\langle \E[\Sig{\lvyalg}_{0,s}], \E[\Sig{\widetilde{\lvyalg}}_{0,t}] \rangle = u(s,t), \qquad s,t\ge0.\end{equation*}

    It is thus left to show that the system \eqref{eq:extended_pde_trunc} with truncation level $\nn = \mm = 2$ is identical to the system \eqref{eq:continuous_case} above.
    This is directly verified using the symmetry properties of $\mathfrak{a}, \widetilde{\mathfrak{a}}, a$ and $\widetilde{a}$ (in particular $\langle a, \mathfrak{a}\rangle = \langle \widetilde{a}, \widetilde{\mathfrak{a}}\rangle \equiv 0$) and the following properties of $\dualLeftOpZ, \dualRightOpZ$ implied by their definition and \cref{ext:prop:left-right-technical}:  %
    For $x, \widetilde{x} \in V$ and $\by, \widetilde{\by} \in V^{\otimes 2}$ it holds \begin{equation*} \dualLeftZ{x}{\widetilde{x}} = \dualRightZ{x}{\widetilde{x}} = \dualLeftZ{\by}{\widetilde{\by}}= \dualRightZ{\by}{\widetilde{\by}} = \dualLeftZ{\by}{x} = \dualRightZ{\by}{x}=0,\end{equation*}
    and 
    \begin{align*}
    \dualRightZ{x}{\by} &= \dualRight{x}{\by} =  \sum_{i=1}^d e_i\sum_{j=1}^d \by^{ij}x^{j} = \by.x,\qquad
    \dualLeftZ{x}{\by} = \dualLeft{x}{\by} = \sum_{i=1}^d e_i\sum_{j=1}^d \by^{ji}x^{j} = \by^{\top}.x.
    \end{align*}
\end{proof}

\subsection{A comparison principle for two-parameter Volterra equations}

The a priori estimate of \cref{lem:apriori} rests on a comparison principle that is a special case of classical Gr\"onwall inequalities in two independent variables; see \cite{snow1972gronwall}. General well-posedness results for Goursat problems are given in \cite{lees1960goursat}. For completeness, we give a self-contained proof of the precise scalar statement needed there.

\begin{lemma}\label{lem:goursat_comparison}
Let $c\ge0$ and let $f,g:[0,T]\to\R_+$ be integrable. If $h:[0,T]^2\to\R_+$ is continuous and satisfies
\begin{align*}
    h(s,t)
    \le&~
    c+\int_0^s\int_0^t f(r)g(v)h(r,v)\dd{v}\dd{r},
\end{align*}
then
\begin{align*}
    h(s,t)
    \le&~
    c I_0\left(
    2\sqrt{
    \int_0^s f(r)\dd{r}
    \int_0^t g(v)\dd{v}}
    \right),
    \qquad s,t\in[0,T].
\end{align*}
Moreover, the right-hand side is the unique continuous solution of the corresponding integral equation with equality.
\end{lemma}
\begin{proof}
Iterating the assumed inequality gives, for every $n\in\N_{\ge1}$,
\begin{align*}
    h(s,t)
    \le&~
    c\sum_{k=0}^{n-1}
    \frac{
    \left(\int_0^s f(r)\dd{r}\right)^k
    \left(\int_0^t g(v)\dd{v}\right)^k
    }{(k!)^2}+
    \Vert h\Vert_\infty
    \frac{
    \left(\int_0^s f(r)\dd{r}\right)^n
    \left(\int_0^t g(v)\dd{v}\right)^n
    }{(n!)^2}.
\end{align*}
The last term converges to zero as $n\to\infty$. Using the power-series representation
\begin{equation*}
I_0\left(2\sqrt{xy}\right)
=
\sum_{k=0}^\infty\frac{x^ky^k}{(k!)^2},
\qquad x,y\ge0,
\end{equation*}
see \cite[9.6.10]{abramowitz1972handbook}, yields the claimed estimate.

The same representation and monotone convergence show that
\begin{equation*}
c I_0\left(
2\sqrt{
\int_0^\cdot f(r)\dd{r}
\int_0^\cdot g(v)\dd{v}}
\right)
\end{equation*}
solves the integral equation with equality.
Finally, if $\varphi_1$ and $\varphi_2$ are two continuous solutions, then
\begin{align*}
    |\varphi_1(s,t)-\varphi_2(s,t)|
    \le&~
    \int_0^s\int_0^t
    f(r)g(v)|\varphi_1(r,v)-\varphi_2(r,v)|
    \dd{v}\dd{r}.
\end{align*}
The first part applied with $c=0$ proves uniqueness.
\end{proof}

\section{An a priori estimate}\label{sec:uniquness}

The following lemma provides an a priori estimate for the infinite-dimensional PDE--ODE system \eqref{eq:extended_pde} and is thus the main step in the proof of \cref{thm:main}.

\begin{lemma}\label{lem:apriori}
Let $(u, \phi, \widetilde{\phi}):[0,T]^2 \to \R \times \TT^1 \times \TT^1$ be a continuous solution to the integral equation
    \begin{equation*}
        \left\{
    \begin{split}
        u(s,t) =&~ u(s,0) + u(0,t) -u(0,0) \\&~+\int_0^s\int_0^t \Big\{ u(r,v) \langle \mathfrak{y}(r), \widetilde{\mathfrak{y}}(v)\rangle + \langle \phi(r,v), \dualRightZ{\mathfrak{y}(r)}{\widetilde{\mathfrak{y}}(v)}\rangle+  \langle \widetilde{\phi}(r,v), \dualRightZ{\widetilde{\mathfrak{y}}(v)}{\mathfrak{y}(r)}\rangle\Big\} \dd{v}\dd{r} \\
        \phi(s,t) =&~ \int_0^s \left\{ u(r,t) \mathfrak{y}(r) + \phi(r,t) \otimes \mathfrak{y}(r) + \dualLeftZ{\widetilde{\phi}(r,t)}{\mathfrak{y}(r)} \right\}\dd{r}\\
        \widetilde\phi(s,t) =&~ \int_0^t \left\{ u(s,r) \widetilde{\mathfrak{y}}(r) + \widetilde\phi(s,r) \otimes \widetilde{\mathfrak{y}}(r) + \dualLeftZ{{\phi}(s,r)}{\widetilde{\mathfrak{y}}(r)} \right\}\dd{r}
        \end{split},\right.
    \end{equation*}
    for all $s,t \in[0,T]$ where $\mathfrak{y}, \widetilde{\mathfrak{y}}: [0,T]\to \TT_0$ with
\begin{equation*}C_t :=  \int_0^t \Vert {\mathfrak{y}}(r)\Vert_1  \dd{r} <\infty , \qquad \widetilde{C}_t := \int_0^t \Vert \widetilde{\mathfrak{y}}(r)\Vert_1  \dd{r} <\infty, \qquad t\in[0,T].\end{equation*} 
Then the following estimate holds:
 \begin{align*}
 	|u(s,t)| + \Vert \phi(s,t) \Vert_1 \le \psi(C_s, \widetilde{C}_t) \sup_{(r,v)\in[0,s]\times[0,t]} \big\vert u(r,0) +  u(0,v) - u(0,0)\big\vert,
 \end{align*}
 and similarly
  \begin{align*}
 	|u(s,t)| + \Vert \widetilde\phi(s,t) \Vert_1 \le \psi(C_s, \widetilde{C}_t) \sup_{(r,v)\in[0,s]\times[0,t]} \big\vert u(r,0) +  u(0,v) - u(0,0)\big\vert,
 \end{align*}
 for all $s,t \in[0,T]$, 
 where $\psi(x,y) =e^{x + y} I_0(2 \sqrt{xy})$
 and $I_0$ is the zero-order modified Bessel function of the first kind.
 \end{lemma}
\begin{proof}
First note that suitably integrating the equation for $\widetilde{\phi}$ and comparing with the equation for $u$ we see that
\begin{align*}
    \int_0^s \langle \widetilde{\phi}(r,t), \mathfrak{y}(r)\rangle \dd{r}  &= \int_0^s \int_0^t \left\langle u(r,v) \widetilde{\mathfrak{y}}(v) + \widetilde\phi(r,v) \otimes \widetilde{\mathfrak{y}}(v) + \dualLeftZ{{\phi}(r,v)}{\widetilde{\mathfrak{y}}(v)}, \mathfrak{y}(r) \right\rangle\dd{v}\dd{r} \\
    &= \int_0^s \int_0^t \left\{ u(r,v) \langle \widetilde{\mathfrak{y}}(v), \mathfrak{y}(r)\rangle + \langle \widetilde\phi(r,v) \otimes \widetilde{\mathfrak{y}}(v), \mathfrak{y}(r)\rangle + \langle \dualLeftZ{{\phi}(r,v)}{\widetilde{\mathfrak{y}}(v)}, \mathfrak{y}(r)\rangle \right\}\dd{v}\dd{r}\\
    &= \int_0^s \int_0^t \Big\{ u(r,v) \langle \mathfrak{y}(r), \widetilde{\mathfrak{y}}(v)\rangle +  \langle \widetilde{\phi}(r,v), \dualRightZ{\widetilde{\mathfrak{y}}(v)}{\mathfrak{y}(r)}\rangle + \langle \phi(r,v), \dualRightZ{\mathfrak{y}(r)}{\widetilde{\mathfrak{y}}(v)}\rangle\Big\}\dd{v}\dd{r} \\
    &= \bigg. u(s,t) - u(s,0) - u(0,t) + u(0,0),
\end{align*}
where in the third equality we repeatedly use that $\langle \dualLeft{a}{b}, c\rangle = \langle \dualLeftZ{a}{b}, c\rangle$ and $\langle \dualRight{a}{b}, c\rangle = \langle \dualRightZ{a}{b}, c\rangle$ for $a, b, c\in \TT^1$ with $\pi_0 c = 0$.
Next, in the equation for $\phi$ we split \begin{equation*}\int_0^s\dualLeftZ{\widetilde{\phi}(r,t)}{\mathfrak{y}(r)} \dd{r} = \int_0^s\dualLeft{\widetilde{\phi}(r,t)}{\mathfrak{y}(r)}\dd{r} - \int_0^s\langle \widetilde{\phi}(r,t), \mathfrak{y}(r)\rangle\dd{r},\end{equation*}
and insert the above derived identity to finally obtain
\begin{align*}
    [\phi + u](s,t) =&~u(s,0) + u(0,t) - u(0,0) + \int_0^s \left\{[\phi+u](r,t) \otimes \mathfrak{y}(r) + \dualLeft{\widetilde{\phi}(r,t)}{\mathfrak{y}(r)} \right\}\dd{r},
\end{align*}
for all $s,t \in[0,T]$.
Now let $s_0,t_0 \in[0,T]$ be arbitrary but fixed and define
\begin{equation*}U:=\sup_{(r,v)\in[0,s_0]\times[0,t_0]} \big\vert u(r,0) +  u(0,v) - u(0,0)\big\vert.\end{equation*}
By \eqref{eq:adjoint_bounds}, $\Vert \dualLeft{a}{b} \Vert_1 \le \Vert a \Vert_1 \Vert b\Vert_1$, and so
\begin{align*}
    \Vert \phi(s,t) + u(s,t) \Vert_1 \le&~ U +  \int_0^s \left(\Vert \phi(r,t) + u(r,t) \Vert_1 + \Vert \widetilde{\phi}(r,t) \Vert_1\right)\Vert\mathfrak{y}(r)\Vert_1\dd{r}.
\end{align*}
An application of Gr\"onwall's inequality, more precisely \cref{lem:special_gronwal}, then yields
\begin{align*}
      \Vert {\phi}(s,t) + u(s,t) \Vert_1 
       \le&~  e^{C_{s}}U +  e^{C_{s}}\int_0^s \Vert \widetilde{\phi}(r,t) \Vert_1  \Vert \mathfrak{y}(r)\Vert_1 e^{-C_r}\dd{r},
\end{align*} 
for all $s\in[0,s_0]$ and $t \in [0,t_0]$.
In the equation for $\widetilde{\phi}$ we directly estimate again using Gr\"onwall's inequality as follows 
\begin{align*}
    \Vert \widetilde\phi(s,t)\Vert_1 \le&~ \int_0^t \left(  \Vert \phi(s,v) + u(s,v)\Vert_1  + \Vert \widetilde{\phi}(s,v)\Vert_1 \right)\Vert \widetilde{\mathfrak{y}}(v)\Vert_1\dd{v} \\
    \le&~ e^{\widetilde{C}_{t}}\int_0^t \Vert {\phi}(s,v) + u(s,v) \Vert_1  \Vert \widetilde{\mathfrak{y}}(v)\Vert_1 e^{-\widetilde{C}_v}\dd{v}.
\end{align*}
Plugging this estimate into the previous one we get
\begin{multline*}
    \Vert {\phi}(s,t) + u(s,t) \Vert_1 
       \le  e^{C_{s}+\widetilde{C}_{t}}\left( U + \int_0^s  \int_0^t \Vert {\phi}(r,v) + u(r,v) \Vert_1 e^{-\widetilde{C}_v-C_r} \Vert \widetilde{\mathfrak{y}}(v)\Vert_1  \Vert \mathfrak{y}(r)\Vert_1\dd{v}  \dd{r} \right).
\end{multline*} 
Hence, the classical comparison principle, which for completeness we include as
\cref{lem:goursat_comparison} in \cref{sec:technical_proofs}, applied to
$e^{-C_s-\widetilde{C}_t}\Vert\phi(s,t)+u(s,t)\Vert_1$, yields
\begin{align*}
    \Vert\phi(s,t)+u(s,t)\Vert_1
    \le&~
    Ue^{C_s+\widetilde{C}_t}
    I_0\left(2\sqrt{C_s\widetilde{C}_t}\right)
    =
    U\psi(C_s,\widetilde{C}_t).
\end{align*}
The corresponding estimate for $\Vert\widetilde{\phi}(s,t)+u(s,t)\Vert_1$ follows analogously.
Finally, $\pi_0\phi=\pi_0\widetilde{\phi}=0$, and hence
\begin{equation*}
\Vert\phi(s,t)+u(s,t)\Vert_1
=
|u(s,t)|+\Vert\phi(s,t)\Vert_1,
\end{equation*}
and analogously for $\widetilde{\phi}$.
Choosing $s_0=s$ and $t_0=t$ concludes the proof.
\end{proof}

\section{Truncation of characteristic velocities and expected signatures}\label{sec:tail_truncation}

In this section we compare truncation before and after free development. Applied to the expected signature equation, these correspond respectively to truncating the characteristic velocity as in \cref{thm:truncated} and to truncating the expected signature directly. We first discuss the two truncation procedures for general free developments and then apply the resulting estimates to compound Poisson processes with Gaussian, compactly supported and exponentially distributed jumps. In the latter cases, upper bounds for truncation of the characteristic velocity are combined with lower bounds for direct expected-signature truncation in suitable one-dimensional examples, showing that the difference in rates is genuine.

\subsection{Truncation before and after free development}

Let $\gamma=\int_0^\cdot\dot\gamma(r)\dd r$ be a componentwise absolutely continuous path in $\TT_0$ with $\onevartwo{\gamma}{s}{t}<\infty$. At truncation level $M$, the two procedures considered in this paper are
\begin{equation*}
\mathcal S\bigl(\pi_{(0,M)}\gamma\bigr)_{s,t}
\qquad\text{and}\qquad
\pi_{(0,M)}\mathcal S(\gamma)_{s,t}.
\end{equation*}
By the triangular structure of the free-development equation, it holds
\begin{equation*}
\pi_{(0,M)}\mathcal S\bigl(\pi_{(0,M)}\gamma\bigr)_{s,t}
=
\pi_{(0,M)}\mathcal S(\gamma)_{s,t}.
\end{equation*}
Thus, the two procedures agree through level $M$, but truncation before free development additionally retains all higher levels generated by the first $M$ levels of $\gamma$.

The error introduced by truncating before free development is controlled directly by the omitted levels of $\gamma$. Indeed, \cref{ext:prop:free_dev_estimates}.\ref{ext:itm:inner_truncation} gives
\begin{equation*}
\Vert \mathcal S(\gamma)_{s,t}-\mathcal S\bigl(\pi_{(0,M)}\gamma\bigr)_{s,t}\Vert_1
\le
e^{\onevartwo{\gamma}{s}{t}}
\onevartwo{\gamma-\pi_{(0,M)}\gamma}{s}{t}.
\end{equation*}
For direct truncation after free development, \cref{ext:prop:free_dev_estimates}.\ref{ext:itm:bell_estimate} gives
\begin{align}\label{eq:direct_truncation_bell}
\Vert \mathcal S(\gamma)_{s,t}-\pi_{(0,M)}\mathcal S(\gamma)_{s,t}\Vert_1
\le
\sum_{n\ge M+1}\frac1{n!}
B_n\Bigl(
1!\onevartwo{\pi_1\gamma}{s}{t},\,\dots,\,
n!\onevartwo{\pi_n\gamma}{s}{t}
\Bigr).
\end{align}
Equality holds in \eqref{eq:direct_truncation_bell} in the one-dimensional case when $\dot\gamma\ge0$ componentwise. The right-hand side is precisely the remainder of the Taylor expansion of
\begin{equation*}
\lambda\mapsto
\exp\left(
\sum_{n=1}^\infty
\lambda^n\onevartwo{\pi_n\gamma}{s}{t}
\right)
\end{equation*}
around $\lambda=0$, evaluated at $\lambda=1$.

For completeness, one can also combine \cref{ext:prop:free_dev_estimates}.\ref{ext:itm:inner_truncation} and \ref{ext:itm:outer_truncation} to obtain the general upper bound
\begin{align}\label{eq:direct_truncation_general_bound}
\Vert \mathcal S(\gamma)_{s,t}-\pi_{(0,M)}\mathcal S(\gamma)_{s,t}\Vert_1
\le
e^{\onevartwo{\gamma}{s}{t}}
\min_{N=1,\dots,M+1}
\left(
\onevartwo{\gamma-\pi_{(0,N)}\gamma}{s}{t}
+
\frac{\onevartwo{\pi_{(0,N)}\gamma}{s}{t}^{\ceil{(M+1)/N}}}
{\ceil{(M+1)/N}!}
\right).
\end{align}
In case $\gamma$ is of a higher homogeneous level, say $\gamma=\pi_{N_0}\gamma$, this gives
\begin{equation*}
\Vert \mathcal S(\gamma)_{s,t}-\pi_{(0,M)}\mathcal S(\gamma)_{s,t}\Vert_1
\le
e^{\onevartwo{\gamma}{s}{t}}
\begin{dcases}
\onevartwo{\gamma}{s}{t},&M<N_0,\\
\dfrac{\onevartwo{\gamma}{s}{t}^{\ceil{(M+1)/N_0}}}
{\ceil{(M+1)/N_0}!},&M\ge N_0.
\end{dcases}
\end{equation*}
One easily observes in the case $d=1$ with non-negative derivative $\dot\gamma^{(N_0)}\ge0$ that this estimate is asymptotically sharp as $M\to\infty$.

\subsection{Remainder estimates}

We next provide two technical lemmas used below. The first concerns the Taylor remainder in \eqref{eq:direct_truncation_bell} for factorially and geometrically decaying levels and relies on known results from asymptotic analysis. The second establishes a factorial-decay bound for the Gaussian moments appearing in the characteristic velocity. For functions $f,g:\N_{\ge1}\to(0,\infty)$, we write $f(m)\sim g(m)$ if and only if $\lim_{m\to\infty}f(m)/g(m)=1$.

Recall that the complete exponential Bell polynomials are characterized by
\begin{equation*}
\exp\left(\sum_{n=1}^\infty a_n z^n\right)
=
\sum_{n=0}^\infty
\frac{z^n}{n!}
B_n\bigl(1!a_1,\,2!a_2,\,\dots,\,n!a_n\bigr).
\end{equation*}

\begin{lemma}\label{lem:truncation_decay}
For a sequence $(a_n)_{n\in\N_{\ge1}}\subset\R_+$ define
\begin{equation*}
r_m
:=
\sum_{n=m}^\infty
\frac1{n!}
B_n\bigl(1!a_1,\,2!a_2,\,\dots,\,n!a_n\bigr),
\qquad m\in\N_{\ge1}.
\end{equation*}
\begin{enumerate}[label=\arabic*.)]
    \item If $a_n=\frac{\rho^n}{n!}$ for all $n\in\N_{\ge1}$ and some $\rho>0$, then
    \begin{equation*}
    r_m\sim\rho^m\frac{B_m}{m!},
    \end{equation*}
    where $B_n=B_n(1,\dots,1)$ are the Bell numbers. In particular, there are constants $c,\varsigma,C_\rho,\zeta>0$ such that
    \begin{equation*}
    \frac{c}{e\sqrt{2\pi}}
    \frac1{\sqrt{m\log(m)}}
    \left(\frac{\varsigma\rho}{\log(m)}\right)^m
    \le
    r_m
    \le
    \frac{C_\rho}{e\sqrt{2\pi}}
    \frac1{\sqrt{m\log(m)}}
    \left(\frac{\zeta\rho}{\log(m)}\right)^m.
    \end{equation*}
    Moreover, for $m\ge m_0\in\N$, all constants $c,\varsigma,C_\rho$ and $\zeta$ can be chosen arbitrarily close to $1$ by choosing $m_0$ sufficiently large.
    \item If $a_n=\rho^n$ for all $n\in\N_{\ge1}$ and some $\rho\in(0,1)$, then
    \begin{equation*}
    r_m
    \sim
    \frac{e^{2\sqrt m}}{2m^{\frac34}\sqrt{\pi e}}
    \frac{\rho^m}{1-\rho}.
    \end{equation*}
\end{enumerate}
\end{lemma}

\begin{proof}
\begin{enumerate}[label=\arabic*.)]
    \item From the asymptotics in \cite[Proposition~VIII.3]{flajolet2009analytic} (see also \cite[Eq.~6.2.6~ff.]{bruijn1958asymptotic}) we obtain
    \begin{equation*}
    \frac{B_n}{n!}
    \sim
    \frac1{e\sqrt{2\pi}}
    \frac1{\sqrt{nW(n+1)}}
    \left(
    \frac{e^{\frac1{W(n+1)}}}{W(n+1)}
    \right)^n,
    \end{equation*}
    where $W$ denotes the Lambert $W$ function, defined by $W(n)e^{W(n)}=n$, and satisfying $W(n)\sim\log n$ as $n\to\infty$.

    It remains to show that $r_m\sim\rho^m\frac{B_m}{m!}$. To this end, note that
    \begin{equation*}
    \alpha_n
    :=
    \frac1{n!}B_n(\rho,\ldots,\rho^n)
    =
    \frac{\rho^n}{n!}B_n.
    \end{equation*}
    In particular, we directly obtain the lower bound
    \begin{equation*}
    r_m\ge\frac{\rho^m}{m!}B_m.
    \end{equation*}
    On the other hand, using the asymptotic formula above, one verifies that the ratio $\alpha_{n+1}/\alpha_n$ converges to $0$ as $n\to\infty$. Hence, for any $\varepsilon>0$, we can find $m_0\in\N$ sufficiently large such that
    \begin{equation*}
    r_m
    \le
    \frac{\rho^m}{m!}B_m\frac1{1-\varepsilon},
    \qquad m\ge m_0.
    \end{equation*}
    The remaining estimates follow directly from the displayed asymptotic formula for the Bell numbers and $W(m)\sim\log m$.

    \item In this case, we use \cite[Theorem~3.3.B]{comtet1974advanced} to obtain
    \begin{equation*}
    \alpha_n
    :=
    \frac1{n!}B_n(1!\rho,\ldots,n!\rho^n)
    =
    \beta_n\rho^n,
    \qquad
    \beta_n
    :=
    \sum_{k=1}^n\frac1{k!}\binom{n-1}{k-1}.
    \end{equation*}
    The coefficients $(\beta_n)_{n\ge1}$ have the known asymptotic behavior \cite[Proposition~VIII.4]{flajolet2009analytic}
    \begin{equation*}
    \beta_n
    \sim
    \frac{e^{2\sqrt n}}{2n^{\frac34}\sqrt{\pi e}}.
    \end{equation*}
    Moreover, the ratio $\alpha_{n+1}/\alpha_n$ converges to $\rho\in(0,1)$. The conclusion follows by summing the resulting asymptotically geometric tail.
\end{enumerate}
\end{proof}

Combining \cref{lem:truncation_decay} with \cref{ext:prop:free_dev_estimates}.\ref{ext:itm:inner_truncation} makes the difference between the two truncations explicit. In the one-dimensional case with $\dot\gamma\ge0$ componentwise and
\begin{equation*}
\onevartwo{\pi_n\gamma}{s}{t}=\frac{\rho^n}{n!},
\qquad n\ge1,
\end{equation*}
truncation before free development satisfies
\begin{equation*}
\Vert\mathcal S(\gamma)_{s,t}-\mathcal S\bigl(\pi_{(0,M)}\gamma\bigr)_{s,t}\Vert_1
\le
e^{e^\rho-1+\rho}
\frac{\rho^{M+1}}{(M+1)!},
\end{equation*}
whereas direct truncation satisfies
\begin{equation*}
\Vert\mathcal S(\gamma)_{s,t}-\pi_{(0,M)}\mathcal S(\gamma)_{s,t}\Vert_1
\sim
\rho^{M+1}\frac{B_{M+1}}{(M+1)!}.
\end{equation*}
In particular, direct truncation after free development does not have factorial-type decay.
Similarly, if
\begin{equation*}
\onevartwo{\pi_n\gamma}{s}{t}=\rho^n,
\qquad n\ge1,
\end{equation*}
for some $\rho\in(0,1)$, then truncation before free development satisfies
\begin{equation*}
\Vert\mathcal S(\gamma)_{s,t}-\mathcal S\bigl(\pi_{(0,M)}\gamma\bigr)_{s,t}\Vert_1
\le
e^{\frac{\rho}{1-\rho}}
\frac{\rho^{M+1}}{1-\rho},
\end{equation*}
while
\begin{equation*}
\Vert\mathcal S(\gamma)_{s,t}-\pi_{(0,M)}\mathcal S(\gamma)_{s,t}\Vert_1
\sim
\frac{e^{2\sqrt{M+1}}}{2(M+1)^{\frac34}\sqrt{\pi e}}
\frac{\rho^{M+1}}{1-\rho}.
\end{equation*}
Hence, direct truncation is slower by an additional subexponential factor. Since equality holds in \eqref{eq:direct_truncation_bell} in this one-dimensional non-negative setting, these are lower bounds for direct truncation and not only consequences of an upper estimate.

\begin{lemma}\label{lem:gauss_mgf}
Let $\Sigma\in\mathfrak S^1_+(\R^d)$ be a positive semidefinite matrix and denote by $\lambda_{\max}(\Sigma)$ its largest eigenvalue. It holds for all $M\in\N_{\ge1}$ that
\begin{equation*}
\E_{\xi\sim\mathcal N(0,\Sigma)}\bigl[e^{|\xi|}|\xi|^{2M}\bigr]
\le
e^{\lambda_{\max}(\Sigma)}
\lambda_{\max}(\Sigma)^M
\E_{\xi\sim\mathcal N(0,I_d)}\bigl[e^{\frac14|\xi|^2}|\xi|^{2M}\bigr].
\end{equation*}
Furthermore,
\begin{equation*}
\frac1{(2M)!}
\E_{\xi\sim\mathcal N(0,I_d)}\bigl[e^{\frac14|\xi|^2}|\xi|^{2M}\bigr]
\sim
\frac{2^{\frac d2}\sqrt\pi}{\Gamma(\frac d2)}
\frac{M^{\frac{d-1}{2}}}{M!},
\end{equation*}
where $\Gamma$ denotes the Gamma function \cite[Chapter~6]{abramowitz1972handbook}.
\end{lemma}

\begin{proof}
Set $\lambda=\lambda_{\max}(\Sigma)$. The first estimate follows simply from
\begin{equation*}
|\Sigma^{\frac12}x|
\le
\sqrt\lambda|x|
\le
\frac14|x|^2+\lambda,
\qquad x\in\R^d.
\end{equation*}
For the second part, first note that
\begin{align*}
\E_{\xi\sim\mathcal N(0,I_d)}\bigl[e^{\frac14|\xi|^2}|\xi|^{2M}\bigr]
&=
\frac{\dd^M}{(\dd\theta)^M}
\E_{\xi\sim\mathcal N(0,I_d)}\bigl[e^{\theta|\xi|^2}\bigr]
\Big\vert_{\theta=\frac14}\\
&=
\frac{\dd^M}{(\dd\theta)^M}
\frac1{(1-2\theta)^{\frac d2}}
\Big\vert_{\theta=\frac14}\\
&=
4^M2^{\frac d2}
\prod_{k=0}^{M-1}\left(\frac d2+k\right)\\
&=
2^{2M+\frac d2}
\frac{\Gamma(\frac d2+M)}{\Gamma(\frac d2)}.
\end{align*}
Finally, using the duplication formula \cite[6.1.18]{abramowitz1972handbook},
\begin{equation*}
2^{2M+\frac d2}
\frac{\Gamma(\frac d2+M)}{\Gamma(\frac d2)(2M)!}
=
\frac{2^{\frac d2}\sqrt\pi}{\Gamma(\frac d2)}
\frac{\Gamma(\frac d2+M)}{\Gamma(\frac12+M)}
\frac1{M!},
\end{equation*}
and the conclusion follows from the well-known asymptotic behavior of ratios of Gamma functions \cite[6.1.46]{abramowitz1972handbook}.
\end{proof}

\subsection{Compound Poisson jump laws}

We now apply the preceding estimates to the expected signature of compound Poisson processes. Recall from \cref{prop:levy_esig} that
\begin{equation*}
\E\bigl[\Sig{\lvyalg}_{0,t}\bigr]
=
\mathcal S\left(\int_0^\cdot\mathfrak y(r)\dd r\right)_t.
\end{equation*}
Consequently, \cref{thm:truncated} and \cref{ext:prop:free_dev_estimates}.\ref{ext:itm:inner_truncation} control truncation at the level of the characteristic velocity directly by
\begin{equation*}
\int_0^t
\Vert\mathfrak y(r)-\pi_{(0,M)}\mathfrak y(r)\Vert_1
\dd r.
\end{equation*}
The examples below show that, for the jump laws considered in the main text, this tail has factorial or geometric decay, while direct truncation of the expected signature can have the slower rates from \cref{lem:truncation_decay}.

\begin{example}[Gaussian jumps]\label{exmpl:gauss_jumps}
Consider a compound Poisson process with centered Gaussian jumps, i.e., $K_t(\cdot)=\lambda_t\mu_t(\cdot)$, where $\lambda:[0,\infty)\to\R_+$ is locally integrable and $\mu_t=\mathcal N(0,\Sigma_t)$ with $\Sigma_t\in\mathfrak S^1_+(V)$. By symmetry, only the even levels of the characteristic velocity are nonzero. Set
\begin{equation*}
\sigma_t^2
:=
\sup_{0\le s\le t}\lambda_{\max}(\Sigma_s)
<\infty.
\end{equation*}
Then, for all $M\in\N_{\ge1}$,
\begin{align*}
\int_0^t
\Vert\mathfrak y(s)-\pi_{(0,2M)}\mathfrak y(s)\Vert_1
\dd s
&\le
\int_0^t\int_V
\Vert\exp_\otimes(x)-\pi_{(0,2M)}\exp_\otimes(x)\Vert_1
K_s(\dd x)\dd s\\
&\le
\int_0^t
\frac1{(2M)!}
\E_{\xi\sim\mathcal N(0,\Sigma_s)}
\bigl[e^{|\xi|}|\xi|^{2M}\bigr]
\lambda_s\dd s\\
&\le
C e^{\sigma_t^2}\sigma_t^{2M}
\frac{M^{\frac{d-1}{2}}}{M!}
\int_0^t\lambda_s\dd s,
\end{align*}
where $C>0$ depends only on $d$ and the last inequality follows from \cref{lem:gauss_mgf}. Hence, truncation at the level of the characteristic velocity has factorial-type decay.

To see that direct expected-signature truncation can be strictly slower, consider the one-dimensional homogeneous case with unit intensity and jump law $\mathcal N(0,2\rho)$ for some $\rho>0$. Then
\begin{equation*}
\pi_{2n}\mathfrak y
=
\frac{\rho^n}{n!},
\qquad n\in\N_{\ge1},
\qquad
\pi_{2n+1}\mathfrak y=0,
\qquad n\in\N.
\end{equation*}
At time $t=1$, it therefore holds
\begin{equation*}
\left\Vert
\E\bigl[\Sig{\lvyalg}_{0,1}\bigr]
-
\pi_{(0,2M)}\E\bigl[\Sig{\lvyalg}_{0,1}\bigr]
\right\Vert_1
=
\sum_{n\ge M+1}
\frac{\rho^nB_n}{n!}
\sim
\frac{\rho^{M+1}B_{M+1}}{(M+1)!}.
\end{equation*}
Thus, for this Gaussian compound Poisson process, direct truncation of the expected signature does not have factorial-type decay.
\end{example}

\begin{example}[Compactly supported jumps]\label{exmpl:compact_jumps}
Let $K_t(\cdot)=\lambda_t\mu_t(\cdot)$, where $\lambda:[0,\infty)\to\R_+$ is locally integrable and the jump laws are uniformly supported in a ball of radius $R>0$, i.e., $\mu_t(\{x\in V:|x|\le R\})=1$. Then, for all $M\in\N_{\ge1}$,
\begin{align*}
\int_0^t
\Vert\mathfrak y(s)-\pi_{(0,M)}\mathfrak y(s)\Vert_1
\dd s
&\le
\int_0^t\lambda_s
\int_V
\sum_{n\ge M+1}\frac{|x|^n}{n!}
\mu_s(\dd x)\dd s\\
&\le
e^R\frac{R^{M+1}}{(M+1)!}
\int_0^t\lambda_s\dd s.
\end{align*}
Thus, truncation at the level of the characteristic velocity again has factorial-type decay.

Direct truncation of the expected signature can also fail to have factorial-type decay within this class. Indeed, consider the one-dimensional homogeneous compound Poisson process with unit intensity and deterministic jump size $\rho>1$, so that $K_t=\delta_\rho$. Its characteristic velocity satisfies
\begin{equation*}
\pi_n\mathfrak y
=
\frac{\rho^n}{n!},
\qquad n\in\N_{\ge1},
\end{equation*}
and hence
\begin{equation*}
\left\Vert
\E\bigl[\Sig{\lvyalg}_{0,1}\bigr]
-
\pi_{(0,M)}\E\bigl[\Sig{\lvyalg}_{0,1}\bigr]
\right\Vert_1
=
\sum_{n\ge M+1}\frac{\rho^nB_n}{n!}
\sim
\frac{\rho^{M+1}B_{M+1}}{(M+1)!}.
\end{equation*}
Hence, direct expected-signature truncation need not have factorial-type decay even for compactly supported jump laws. The translated circle law in \cref{exmpl:circle_jumps} is covered by the factorial upper bound above; the corresponding numerical comparison is shown there in \cref{fig:trunc_circle}.
\end{example}

\begin{example}[Exponentially distributed jumps]\label{exmpl:exponential_jumps}
Let $V=\R$ and, for some $\rho\in(0,1)$, consider the time-independent L\'evy measure
\begin{equation*}
K_t(\dd x)
=
\frac1\rho e^{-x/\rho}\indic{x>0}\dd x,
\qquad t\ge0.
\end{equation*}
The corresponding compound Poisson process has drift characteristic
$b
=
\int_{(0,1]}xK_0(\dd x)$,
so that the compensation term in the characteristic velocity is canceled. Then
\begin{equation*}
\mathfrak y
=
\int_{(0,\infty)}
\bigl(\exp_\otimes(x)-1\bigr)
K_0(\dd x),
\end{equation*}
and therefore
\begin{equation*}
\pi_n\mathfrak y
=
\frac1{n!}\int_0^\infty x^n\frac1\rho e^{-x/\rho}\dd x
=
\rho^n,
\qquad n\in\N_{\ge1}.
\end{equation*}
Consequently,
\begin{equation*}
\Vert\mathfrak y-\pi_{(0,M)}\mathfrak y\Vert_1
=
\frac{\rho^{M+1}}{1-\rho},
\end{equation*}
so truncation at the level of the characteristic velocity has geometric decay. On the other hand, at time $t=1$,
\begin{equation*}
\left\Vert
\E\bigl[\Sig{\lvyalg}_{0,1}\bigr]
-
\pi_{(0,M)}\E\bigl[\Sig{\lvyalg}_{0,1}\bigr]
\right\Vert_1
\sim
\frac{e^{2\sqrt{M+1}}}{2(M+1)^{\frac34}\sqrt{\pi e}}
\frac{\rho^{M+1}}{1-\rho}.
\end{equation*}
Thus, direct truncation of the expected signature is slower by an additional subexponential factor.
\end{example}

\section{Square integrability of continuous semimartingale signatures}

\newcommand{\varqvlambda}[2]{L_{#1}(#2)}

This section establishes two integrability estimates: a first-moment estimate for continuous semimartingale signatures and a second-moment estimate for conditionally Gaussian martingale signatures.

The proof of the following lemma shows that, for a $d$-dimensional continuous semimartingale with decomposition $X=X_0+A+M$, the process
\begin{equation*}
e^{-2\varqvlambda{\lambda}{X}_t}
\left\Vert\delta_\lambda\Sig{X}_{0,t}\right\Vert_2^2,
\qquad
\varqvlambda{\lambda}{X}_t
:=
\lambda\Vert A\Vert_{1\text{-var};[0,t]}
+
\lambda^2\operatorname{tr}\langle M\rangle_t,
\end{equation*}
is a non-negative supermartingale. Consequently, $\Vert\delta_\lambda\Sig{X}_{0,t}\Vert_2$ is integrable whenever $\varqvlambda{\lambda}{X}_t$ has an exponential moment of order two. In one dimension, this follows from a direct application of It\^o's formula to
\begin{equation*}
e^{-2\varqvlambda{\lambda}{X}_t}
\left\Vert\delta_\lambda\Sig{X}_{0,t}\right\Vert_2^2
=
e^{-2\varqvlambda{\lambda}{X}_t}
I_0\big(2\lambda(X_t-X_0)\big),
\end{equation*}
where $I_0$ is the modified Bessel function.
In higher dimensions, the analogous calculation is first applied to truncated signatures and then passed to the full signature.

\begin{lemma}\label{lem:martingale_signature_integrability}
Fix a stochastic basis $(\Omega,\F,(\F_t)_{t\in[0,T]},\P)$ satisfying the usual conditions. Let $X=X_0+A+M$ be a continuous $V\cong\R^d$-valued semimartingale, where $M$ is a continuous local martingale, $A$ is a continuous adapted process of finite variation, and $A_0=M_0=0$. Let $\tau\le T$ be a stopping time. Then, for every $\lambda>0$,
\begin{align}\label{eq:martingale_signature_integrability}
    \E\left[\left\Vert\delta_\lambda\Sig{X}_{0,\tau}\right\Vert_2\right]
    \le
    \E\left[
        \exp\left(
            2\varqvlambda{\lambda}{X}_\tau
        \right)
    \right]^{1/2}.
\end{align}
Furthermore, suppose that
\begin{equation*}
\rho(X,\tau)
:=
\sup\left\{
    r>0
    \;\middle|\;
    \E\left[
        \exp\left(
            2\varqvlambda{r}{X}_\tau
        \right)
    \right]
    <\infty
\right\}
\in(0,\infty],
\end{equation*}
then
$
\E[\Sig{X}_{0,\tau}]\in\mathcal R^{\rho(X,\tau)}.
$
\end{lemma}

\begin{proof}
Set $Q:=\operatorname{tr}\langle M\rangle$ and $F:=\Vert A\Vert_{1\text{-var};[0,\cdot]}$, so that
$
\varqvlambda{\lambda}{X}=\lambda F+\lambda^2Q.
$
By replacing $X$, $M$, and $A$ with their processes stopped at $\tau$, it suffices to prove the result for $\tau=T$. We first consider $\lambda=1$. For $N\in\N$, set
\begin{equation*}
S_t^{(0,N)}:=\pi_{(0,N)}\Sig{X}_{0,t}\in T^N(V),
\qquad
Y_t^N:=\Vert S_t^{(0,N)}\Vert_2^2.
\end{equation*}

By \cite[Propositions~I.3.13 and II.2.9]{jacod2003limit}, there exist predictable processes $a$ and $c$, taking values in $V$ and $V^{\otimes2}$, respectively, such that
\begin{equation*}
A_t-A_0=\int_0^t a_s\dd F_s,
\qquad
\langle M\rangle_t=\int_0^t c_s\dd Q_s.
\end{equation*}
Since $F$ is the total variation of $A$, we have $|a_s|=1$ for $\dd F_s$-almost every $s$. Moreover, $c_s$ is positive semidefinite and, since $Q=\operatorname{tr}\langle M\rangle$, satisfies $\operatorname{tr}(c_s)=1$ for $\dd Q_s$-almost every $s$.
By the It\^o form of the signature equation \eqref{ext:eq:ito_signature}, which for continuous $X$ carries no jump term,
\begin{equation*}
S_t^{(0,N)}
=
1
+
\sum_{i=1}^d\int_0^tS_s^{(0,N)}\otimes_Ne_i\dd M_s^i
+
\int_0^tS_s^{(0,N)}\otimes_Na_s\dd F_s
+
\frac12\int_0^tS_s^{(0,N)}\otimes_Nc_s\dd Q_s.
\end{equation*}
Applying It\^o's formula to $Y^N$ gives
\begin{equation*}
Y_t^N
=
1+Z_t^N
+
2\int_0^t\eta_s^N\dd F_s
+
\int_0^t\beta_s^N\dd Q_s,
\end{equation*}
where
\begin{equation*}
Z_t^N
:=
2\sum_{i=1}^d\int_0^t
\left\langle S_s^{(0,N)},S_s^{(0,N)}\otimes_Ne_i\right\rangle
\dd M_s^i
\end{equation*}
is a continuous local martingale, $\eta_s^N:=\langle S_s^{(0,N)},S_s^{(0,N)}\otimes_Na_s\rangle$, and, writing $c_s^{ij}:=\langle c_s,e_i\otimes e_j\rangle$,
\begin{equation*}
\beta_s^N
:=
\left\langle S_s^{(0,N)},S_s^{(0,N)}\otimes_Nc_s\right\rangle
+
\sum_{i,j=1}^dc_s^{ij}
\left\langle
S_s^{(0,N)}\otimes_Ne_i,
S_s^{(0,N)}\otimes_Ne_j
\right\rangle.
\end{equation*}
By \cref{ext:prop:young}, $\Vert\bx\otimes_Nv\Vert_2\le |v|\Vert\bx\Vert_2$ for $\bx\in T^N(V)$ and $v\in V$, and analogously $\Vert\bx\otimes_Nz\Vert_2\le |z|\Vert\bx\Vert_2$ for $z\in V^{\otimes2}$.
Hence, by Cauchy--Schwarz, $|\eta_s^N|\le Y_s^N$ for $\dd F_s$-almost every $s$. Moreover, since $c_s$ is positive semidefinite, $|c_s|\le\operatorname{tr}(c_s)=1$, and hence the first term in the definition of $\beta_s^N$ is bounded in absolute value by $Y_s^N$. By the Euclidean tensor inner product fixed in the preliminaries,
\begin{equation*}
\sum_{i,j=1}^dc_s^{ij}
\left\langle
S_s^{(0,N)}\otimes_Ne_i,
S_s^{(0,N)}\otimes_Ne_j
\right\rangle
=
\operatorname{tr}(c_s)
\left\Vert\pi_{(0,N-1)}S_s^{(0,N)}\right\Vert_2^2
\le Y_s^N.
\end{equation*}
Consequently, $\beta_s^N\le2Y_s^N$ for $\dd Q_s$-almost every $s$. It\^o's product rule therefore gives
\begin{equation*}
e^{-2\varqvlambda{1}{X}_t}Y_t^N
=
1
+
\int_0^te^{-2\varqvlambda{1}{X}_s}\dd Z_s^N
+
2\int_0^te^{-2\varqvlambda{1}{X}_s}(\eta_s^N-Y_s^N)\dd F_s
+
\int_0^te^{-2\varqvlambda{1}{X}_s}(\beta_s^N-2Y_s^N)\dd Q_s.
\end{equation*}
Thus $e^{-2\varqvlambda{1}{X}}Y^N$ is a non-negative local supermartingale and hence a supermartingale. Since
$
e^{-2\varqvlambda{1}{X}_t}Y_t^N
$ increases to $
e^{-2\varqvlambda{1}{X}_t}\Vert\Sig{X}_{0,t}\Vert_2^2
$ for $N\to \infty$
and every $t\in[0,T]$, monotone convergence shows that
$
e^{-2\varqvlambda{1}{X}}\Vert\Sig{X}_{0,\cdot}\Vert_2^2
$
is itself a non-negative supermartingale.

If $\E[e^{2\varqvlambda{1}{X}_T}]=\infty$, then \eqref{eq:martingale_signature_integrability} is immediate. Otherwise, by Cauchy--Schwarz,
\begin{equation*}
\E\left[\Vert S_T^{(0,N)}\Vert_2\right]
\le
\E\left[e^{2\varqvlambda{1}{X}_T}\right]^{1/2}
\E\left[e^{-2\varqvlambda{1}{X}_T}Y_T^N\right]^{1/2}
\le
\E\left[e^{2\varqvlambda{1}{X}_T}\right]^{1/2}.
\end{equation*}
Monotone convergence then yields
\begin{equation*}
\E\big[\Vert\Sig{X}_{0,T}\Vert_2\big]
=
\lim_{N\to\infty}\E\big[\Vert S_T^{(0,N)}\Vert_2\big]
\le
\E\left[e^{2\varqvlambda{1}{X}_T}\right]^{1/2}.
\end{equation*}

For general $\lambda>0$, since
$
\delta_\lambda\Sig{X}_{0,T}=\Sig{\lambda X}_{0,T}
$
and
$
\varqvlambda{1}{\lambda X}=\varqvlambda{\lambda}{X},
$
applying the argument just given to $\lambda X$ shows that
$
e^{-2\varqvlambda{\lambda}{X}}
\Vert\delta_\lambda\Sig{X}_{0,\cdot}\Vert_2^2
$
is a non-negative supermartingale and gives \eqref{eq:martingale_signature_integrability}.

It remains to prove the final assertion. The map
$
r\mapsto\E[\exp(2\varqvlambda{r}{X}_\tau)]
$
is non-decreasing. Hence the expectation is finite for every $r\in(0,\rho(X,\tau))$: indeed, by the definition of the supremum, there exists $s>r$ for which it is finite, and the claim follows by monotonicity.

Since $\rho(X,\tau)>0$, we may choose $r_0\in(0,\rho(X,\tau))$. By \eqref{eq:martingale_signature_integrability},
\begin{equation*}
\E[|\pi_n\Sig{X}_{0,\tau}|]
\le
r_0^{-n}\E[\Vert\delta_{r_0}\Sig{X}_{0,\tau}\Vert_2]
<\infty
\end{equation*}
for every $n\ge0$. We may therefore define the expected signature componentwise by
$
\esig(\tau):=(\E[\pi_n\Sig{X}_{0,\tau}])_{n\ge0}.
$
Fix $\alpha\in(0,\rho(X,\tau))$ and choose $r\in(\alpha,\rho(X,\tau))$. By monotone convergence and Jensen's inequality in each finite-dimensional space we have
\begin{equation*}
\Vert\delta_r\esig(\tau)\Vert_2
=
\lim_{N\to\infty}
\left\Vert
\E\left[\pi_{(0,N)}\delta_r\Sig{X}_{0,\tau}\right]
\right\Vert_2
\le
\E\left[\Vert\delta_r\Sig{X}_{0,\tau}\Vert_2\right]
\end{equation*}
and the right-hand side is finite by the choice of $r<\rho(X,\tau)$ and \eqref{eq:martingale_signature_integrability}.
Hence, by the Cauchy--Schwarz inequality for $\ell^p$-spaces (cf. \cref{ext:sec:Tpspaces}),
\begin{equation*}
\Vert\delta_\alpha\esig(\tau)\Vert_1
\le
\left(
    \sum_{n=0}^\infty\left(\frac{\alpha}{r}\right)^{2n}
\right)^{1/2}
\Vert\delta_r\esig(\tau)\Vert_2
=
\frac{1}{\sqrt{1-\alpha^2/r^2}}
\Vert\delta_r\esig(\tau)\Vert_2
<\infty.
\end{equation*}
Since $\alpha\in(0,\rho(X,\tau))$ was arbitrary, $\esig(\tau)\in\mathcal R^{\rho(X,\tau)}$.
\end{proof}

For conditionally Gaussian martingales, the additional independence assumption defining $\mathcal Q_G$ (see the text following \eqref{eq:time-dependent-BM}) allows us to obtain a second-moment estimate under the same exponential-moment condition. In contrast to the preceding It\^o argument, the proof exploits the group-like structure of the signature through the coproduct estimate in \cite[Equation~(2.14)]{chevyrev2026orthogonal}.

\begin{lemma}\label{lem:conditional_gaussian_square_integrability}
Let $X\sim P\in\mathcal Q_G$. Then
\begin{equation*}
\E\left[\left\Vert\Sig{X}_{0,T}\right\Vert_2^2\,\middle|\,\langle X\rangle\right]
\le
\exp\left(2\operatorname{tr}\langle X\rangle_T\right).
\end{equation*}
In particular, $\E_{X\sim P}[\Vert\Sig{X}_{0,T}\Vert_2^2]<\infty$.
\end{lemma}

\begin{proof}
By \eqref{eq:conditional_gaussian_law} and \cref{prop:levy_esig}, \eqref{eq:conditional_expected_signature} holds coordinatewise. Write $a_s:=\sum_{i=1}^m\sigma_s^i\otimes\sigma_s^i$.
For $n\geq 0$, let $C_n:V^{\otimes n}\otimes V^{\otimes n}\to\R$ denote the contraction induced by the Hilbert inner product, given by $C_n(u\otimes v)=\langle u,v\rangle$. Since $\Sig{X}_{0,T}$ is group-like, $C_n\big((\pi_n\otimes\pi_n)\Delta\pi_{2n}\Sig{X}_{0,T}\big)=|\pi_n\Sig{X}_{0,T}|^2$. Applying the coproduct estimate in \cite[Equation~(2.14)]{chevyrev2026orthogonal} conditionally and then applying $C_n$ gives
\begin{align*}
    \E\left[\left|\pi_n\Sig{X}_{0,T}\right|^2\,\middle|\,\langle X\rangle\right]
    \le
    2^{2n}
    \left|\pi_{2n}\E\left[\Sig{X}_{0,T}\,\middle|\,\langle X\rangle\right]\right|_{\mathrm{proj}} =
    2^{2n}
    \left|\pi_{2n}\mathcal S\left(\tfrac12\langle X\rangle\right)_{0,T}\right|_{\mathrm{proj}},
\end{align*}
where $|\cdot|_{\mathrm{proj}}$ denotes the projective tensor norm induced by the Euclidean norm on $V$. Since $\tfrac12\langle X\rangle$ is a homogeneous second-level tensor, we can directly estimate using the compatibility of the projective tensor norm
\begin{align*}
    \left|\pi_{2n}\mathcal S\left(\tfrac12\langle X\rangle\right)_{0,T}\right|_{\mathrm{proj}}
    &\le
    \int_{0<t_1<\cdots<t_n<T}
    \prod_{k=1}^n\frac12|a_{t_k}|_{\mathrm{proj}}
    \dd t_1\cdots\dd t_n\le
    \frac{\left(\operatorname{tr}\langle X\rangle_T/2\right)^n}{n!}.
\end{align*}
Combining the two previous estimates and summing over $n$ and using conditional monotone convergence gives
\begin{equation*}
\E\left[\left\Vert\Sig{X}_{0,T}\right\Vert_2^2\,\middle|\,\langle X\rangle\right]
\le
\exp\left(2\operatorname{tr}\langle X\rangle_T\right).
\end{equation*}
Consequently, \eqref{eq:moment_condition_cond_gaus} implies
$
\E_{X\sim P}[\Vert\Sig{X}_{0,T}\Vert_2^2]<\infty.
$
\end{proof}

\section{Numerical validation}\label{sec:numerics_montecarlo}

We provide a numerical example that validates the numerical implementation of the finite-dimensional system in \cref{cor:level2} against the theoretical formulas and illustrates the computational advantages of the analytic representation of the expected signature kernel. More precisely, we consider the Gaussian case under irregular observations, compare the analytic evaluation with a Monte Carlo approximation, and study the effect of increasing the number of irregularly observed sample paths.

\newcommand{\msamp}{m}
\newcommand{\nsamp}{n}

We compare two bivariate time-inhomogeneous Gaussian martingales on $[0,T]$ with $T=1$. More precisely, let
\begin{equation*}
P:=\mathscr{L}^1(0,a_P,0),
\qquad
Q:=\mathscr{L}^1(0,a_Q,0),
\end{equation*}
where
\begin{equation*}
a_P(t)=
\begin{pmatrix}
1 & \rho_P(t)\\
\rho_P(t) & 1
\end{pmatrix},
\qquad
a_Q(t)=
\begin{pmatrix}
1 & \rho_Q(t)\\
\rho_Q(t) & 1
\end{pmatrix},
\end{equation*}
with
\begin{equation*}
\rho_P(t)=0.65+0.20\sin(2\pi t),
\qquad
\rho_Q(t)=-0.65+0.20\cos(2\pi t).
\end{equation*}
We denote their expected signatures by
\begin{equation*}
\esig_P:=\E_{X\sim P}\big[\Sig{X}_{0,T}\big],
\qquad
\esig_Q:=\E_{Y\sim Q}\big[\Sig{Y}_{0,T}\big].
\end{equation*}
It then follows from \cref{exmpl:time-dependent-bm} that
\begin{equation*}
\ip{\esig_P}{\esig_Q}=u_\alpha(T,T),
\end{equation*}
where $u_\alpha$ denotes the solution of the Goursat problem \eqref{eq:goursat} with
\begin{equation*}
\alpha(s,t)=\frac14\ip{a_P(s)}{a_Q(t)}
=\frac12\bigl(1+\rho_P(s)\rho_Q(t)\bigr).
\end{equation*}
A numerical evaluation of the associated Goursat problem on a uniform time grid with $2^{10}$ steps yields
\begin{equation*}
\ip{\esig_P}{\esig_Q}\approx 1.31028.
\end{equation*}

\noindent \textbf{Irregular observations.}
We consider the evaluation of $\ip{\esig_{\widetilde P^{\msamp}}}{\esig_Q}$, where $\widetilde P^{\msamp}$ is the empirical measure obtained from irregular observations of $P$-distributed sample paths. More precisely,
\begin{align*}
\widetilde P^{\msamp}
:=
\frac1{\msamp}\sum_{i=1}^{\msamp}\delta_{\widetilde x^{(i)}},\qquad
\esig_{\widetilde P^{\msamp}}
=
\frac1{\msamp}\sum_{i=1}^{\msamp}
\Sig{\widetilde x^{(i)}}_{0,T},
\end{align*}
where $\widetilde x^{(i)}$ is the piecewise linear interpolation of the observed values
\begin{equation*}
\Big\{x^{(i)}_{\tau^{(i)}_0},\ldots,x^{(i)}_{\tau^{(i)}_{n_i}}\Big\},
\end{equation*}
with $\{x^{(i)}\}_{i=1}^{\msamp}$ i.i.d.\ samples from $P$, simulated on a grid with $2^{10}$ steps, and
\begin{equation*}
0=\tau^{(i)}_0<\tau^{(i)}_1<\cdots<\tau^{(i)}_{n_i}=T
\end{equation*}
independent Poisson observation times of intensity $\lambda=64$.
Numerically, these times are approximated by Bernoulli thinning of the fine base grid, always keeping the endpoints $0$ and $T$.

\noindent \textbf{Analytic vs.\ Monte Carlo approximation.}
We first hold one realization of $\widetilde P^{\msamp}$ fixed with $\msamp=16$. By \cref{cor:level2},
\begin{equation*}
\ip{\esig_{\widetilde P^{\msamp}}}{\esig_Q}
=
\frac1{\msamp}\sum_{k=1}^{\msamp}v_k(T,T),
\end{equation*}
where each $v_k$ is obtained as the solution of the PDE system \eqref{eq:continuous_case} with $b_k=\widetilde x^{(k)}$, $\mathfrak a_k=0$, and $a=a_Q$. For the randomly fixed empirical measure $\widetilde P^{16}$, a numerical solution of this PDE system yields
\begin{equation*}
\ip{\esig_{\widetilde P^{16}}}{\esig_Q}\approx 1.29508.
\end{equation*}

On the other hand, we approximate $Q$ by an empirical measure based on independent samples from $Q$. Employing the higher-order scheme from \cite{lemercier2024high}, each sample is first simulated on a fine grid and then replaced by its step-$4$ piecewise log-linear approximation. More precisely, in our notation, for each sample $y^{(r)}$ simulated on a grid with $2^{10}$ points, we set
\begin{equation*}
b^{(r)}_t
\equiv
\frac{1}{t_{i+\delta}-t_i}\,
\log \pi_{(0,4)}\Sig{y^{(r)}}_{t_i,t_{i+\delta}}
\in \mathfrak g^4(\mathbb R^2),
\qquad
t\in[t_i,t_{i+\delta}),
\end{equation*}
with $\delta=16$, and define
\begin{align*}
\widehat Q^{\nsamp}
=
\frac1{\nsamp}\sum_{r=1}^{\nsamp}\delta_{\mathbf y^{(r)}},
\qquad
\esig_{\widehat Q^{\nsamp}}
=
\frac1{\nsamp}\sum_{r=1}^{\nsamp}
\Sig{\mathbf y^{(r)}}_{0,T},
\end{align*}
where $\delta_{\mathbf y^{(r)}}=\mathscr L^4\!\bigl(b^{(r)},0,0\bigr)$. This gives the Monte Carlo approximation
\begin{align*}
\ip{\esig_{\widetilde P^{16}}}{\esig_{\widehat Q^{\nsamp}}}
&=
\frac1{\nsamp}\sum_{r=1}^{\nsamp}
\ip{\esig_{\widetilde P^{16}}}{\Sig{\mathbf y^{(r)}}_{0,T}}
\\
&=
\frac{1}{16\nsamp}
\sum_{k=1}^{16}\sum_{r=1}^{\nsamp}
\ip{\Sig{\widetilde x^{(k)}}_{0,T}}
{\Sig{\mathbf y^{(r)}}_{0,T}}.
\end{align*}
Each term
$
\ip{\Sig{\widetilde x^{(k)}}_{0,T}}
{\Sig{\mathbf y^{(r)}}_{0,T}}
$
is evaluated via the PDE representation in \cref{cor:level2}. For $\nsamp=2^{17}$ independent samples, we thus obtain the Monte Carlo estimate
\begin{equation*}
\ip{\esig_{\widetilde P^{16}}}{\esig_{\widehat Q^{\nsamp}}}
\approx 1.29359,
\end{equation*}
together with the empirical variance
\begin{equation*}
\hat\sigma_Q^2
=
\frac{1}{\nsamp-1}
\sum_{r=1}^{\nsamp}
\left(
\ip{\esig_{\widetilde P^{16}}}{\Sig{\mathbf y^{(r)}}_{0,T}}
-
\ip{\esig_{\widetilde P^{16}}}{\esig_{\widehat Q^{\nsamp}}}
\right)^2
\approx(0.37662)^2.
\end{equation*}
\Cref{fig:numerics_gaussian_irregular} compares these Monte Carlo approximations and their induced uncertainty regions with the analytic target.

\begin{figure}[h]
    \centering
    \begin{minipage}[t]{0.8\textwidth}
        \centering
        \includegraphics[width=\textwidth]{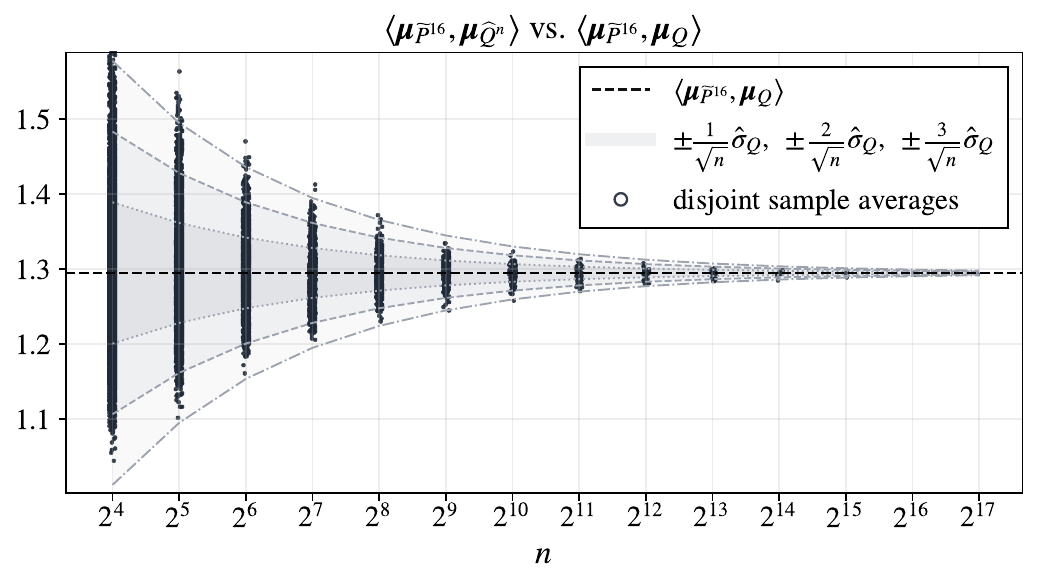}
    \end{minipage}
    \caption{
    Validation of the analytic cross-kernel formula under irregular sampling.
    For a fixed irregular empirical sample $\widetilde P^{16}$, Monte Carlo estimates of
    $\ip{\esig_{\widetilde P^{16}}}{\esig_{\widehat Q^{\nsamp}}}$
    as a function of $\nsamp$, obtained by partitioning a total sample of size $2^{17}$ into disjoint blocks of size $\nsamp$ and averaging within each block.
    The corresponding uncertainty bands are shown in the legend.}
    \label{fig:numerics_gaussian_irregular}
\end{figure}

\noindent \textbf{Increasing the number of observed paths.}
We next study the convergence of $\ip{\esig_{\widetilde P^{\msamp}}}{\esig_Q}$ as the number $\msamp$ of irregularly observed $P$-samples increases. To this end, we generate $\msamp=2^{17}$ independent latent samples from $P$, observe each of them on independent Poisson grids of intensity $\lambda=64$, and construct the corresponding piecewise linear reconstructions $\widetilde x^{(i)}$. In this larger experiment, the numbers of observations range from $34$ to $110$. For each reconstructed path we compute
\begin{equation*}
\ip{\esig_{\widetilde P^{\msamp}}}{\esig_Q}
=
\frac1{\msamp}\sum_{i=1}^{\msamp}
\ip{\Sig{\widetilde x^{(i)}}_{0,T}}{\esig_Q}
\approx 1.31065,
\end{equation*}
with empirical variance
$
\hat\sigma_P^2\approx(0.32607)^2.
$
This is remarkably close to the continuous-time target
$
\ip{\esig_P}{\esig_Q}\approx1.31028,
$
given that each path is observed at only about $64$ time points on average.
\Cref{fig:numerics_gaussian_paths} shows the corresponding convergence in $\msamp$.

\begin{figure}[h]
    \centering
    \begin{minipage}[t]{0.8\textwidth}
        \centering
        \includegraphics[width=\textwidth]{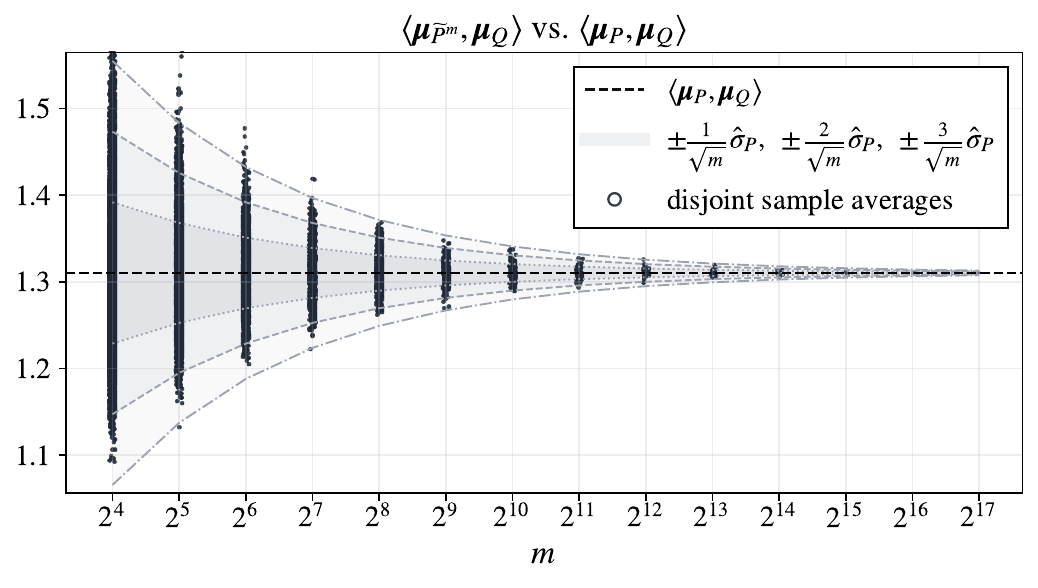}
    \end{minipage}
    \caption{
    Convergence in the number of irregularly observed paths.
    Estimates of
    $\ip{\esig_{\widetilde P^{\msamp}}}{\esig_Q}$
    as a function of $\msamp$, obtained by partitioning a total sample of size $2^{17}$ into disjoint blocks of size $\msamp$ and averaging within each block, against the continuous-time target $\ip{\esig_P}{\esig_Q}$.
    The corresponding uncertainty bands are shown in the legend.}
    \label{fig:numerics_gaussian_paths}
\end{figure}
\end{document}